\documentclass[reqno]{article}

\usepackage{amscd,amssymb,epsfig, epsf,pinlabel,graphicx,yhmath,amsthm}
\usepackage[mathscr]{euscript}
\usepackage{epic,eepic}
\usepackage[dvipsnames]{xcolor}

\usepackage{hyperref}
\hypersetup{colorlinks=true,citecolor=brown,linkcolor=brown,urlcolor=black,filecolor=black}

\usepackage[all]{xy}
\SelectTips{cm}{}

\newcommand{\up}{\vspace{-0.5cm}}
\allowdisplaybreaks
\numberwithin{equation}{section}

\newtheorem{propo}{Proposition}[section]
\newtheorem{corol}[propo]{Corollary}
\newtheorem{theor}[propo]{Theorem}
\newtheorem{lemma}[propo]{Lemma}

\theoremstyle{definition}

\theoremstyle{remark}

\newcommand{\ZZ}{\mathbb{Z}}

 \newcommand{\kk}{\mathbb{K}}

\newcommand{\End}{\operatorname{End}}
\newcommand{\Hom}{\operatorname{Hom}}
\newcommand{\Ker}{\operatorname{Ker}}

\newcommand{\Aut}{\operatorname{Aut}}

\newcommand{\id}{\operatorname{id}}

\newcommand{\GL}{\operatorname{GL}}

\newcommand{\abelian}{\operatorname{Com}}

\newcommand{\Perm}{\mathsf{P}}

\newcommand{\Mat}{\operatorname{Mat}}

\newcommand{\grAlg}{g\mathscr{A}}

\newcommand{\grCom}{cg\mathscr{A}}
\newcommand{\Com}{c\mathscr{A}}
\newcommand{\Gr}{\mathscr{G}r}
\newcommand{\Set}{\mathscr{S}et}
\newcommand{\Lie}{\mathscr{L}ie}
\newcommand{\Kscr}{\mathscr{K}}
\newcommand{\Xscr}{\mathscr{X}}
\newcommand{\Yscr}{\mathscr{Y}}
\newcommand{\Rscr}{\mathscr{R}}

\newcommand{\Gscr}{\mathscr{G}}
\newcommand{\Mscr}{\mathscr{M}}
\newcommand{\Mon}{\mathscr{M}on}

\newcommand{\bracket}[2]{\left\{#1,#2\right\}}
\newcommand{\double}[2]{\left\{\!\!\left\{#1,#2\right\}\!\!\right\}}
\newcommand{\triple}[3]{\left\{\!\!\left\{#1,#2,#3\right\}\!\!\right\}}
\newcommand{\triplep}[3]{\left\vert\!\left\vert#1,#2,#3\right\vert\!\right\vert}

\newcommand{\by}[1]{\stackrel{\eqref{#1}}{=}}

\title{Brackets in representation algebras \\ of  Hopf algebras}

\author{
  Gw\'ena\"el Massuyeau\thanks{IRMA,  Universit\'e de Strasbourg \& CNRS,
 7 rue Ren\'e Descartes, 67084 Strasbourg, France; \emph{current address:} {IMB}, Universit\'e  Bourgogne Franche-Comt\'e \& CNRS, 21000 Dijon, France;
 $\mathtt{gwenael.massuyeau@u\hbox{-}bourgogne.fr}$}
   \and
   Vladimir Turaev\thanks{Department of Mathematics,
Indiana University,
Bloomington IN47405, USA;
 $\mathtt{vturaev@yahoo.com}$}
}

\begin{document}

\maketitle

\vspace{1cm}

\begin{abstract}
For any graded bialgebras  $A$ and~$B$,
we  define a commutative graded algebra $A_B$    representing the functor of  $B$-representations of $A$.
When $A$ is a cocommutative graded Hopf algebra and $B$ is a commutative ungraded Hopf algebra,
we   introduce a method  deriving a     Gerstenhaber bracket in  $A_B$ from      a Fox pairing in~$A$ and a balanced biderivation in~$B$.
Our construction  is inspired by Van den Bergh's non-commutative Poisson  geometry,
  and  may be viewed as an algebraic generalization of  the Atiyah--Bott--Goldman Poisson structures on   moduli spaces of representations of surface groups.
\end{abstract}

\vspace{1cm}


\vspace{0.5cm}


\newpage

\section{Introduction}

 Given  bialgebras  $A$ and~$B$, we introduce a  commutative \emph{representation algebra} $A_B$ which  encapsulates   $B$-representations   of~$A$ (defined in the paper).
For example, if  $A$ is the group algebra of a group~$\Gamma$
and~$B$ is the coordinate algebra of a group scheme~$\Gscr$,
 then  $A_B$ is the coordinate algebra of the affine scheme $C\mapsto \Hom_{  \Gr} (\Gamma, \Gscr (C))$, where~$C$ runs over   all  commutative algebras.
  Another  example:  if $A$ is the enveloping algebra   of a Lie algebra~$\mathfrak{p}$
and  $B$ is  the coordinate  algebra   of a group scheme  with Lie algebra~$\mathfrak{g}$,   then   $A_B$ is the coordinate algebra of
the affine scheme   $C \mapsto
{\Hom_{\Lie}(\mathfrak{p},\mathfrak{g} \otimes    C)}$.

The goal of this paper is to  introduce an algebraic method producing Poisson brackets in the representation algebra~$A_B$.
   We      focus on the case where~$A$ is a cocommutative  Hopf algebra   and~$B$ is a commutative Hopf
algebra as   in the  examples above.
 We   assume      $A$ to be  endowed      with a  bilinear pairing $\rho:{A \times A }\to A$
which is an   antisymmetric   Fox pairing in the sense of \cite{MT_dim_2}.
We introduce  a notion of  a  balanced   biderivation in $B$,
  which is a   symmetric bilinear form $\bullet: B \times B \to \kk$ satisfying certain conditions.
Starting from  such~$\rho$ and~$\bullet$,  we construct an antisymmetric  bracket    in $A_B$  satisfying the Leibniz rules.
Under further  assumptions  on $\rho$ and $\bullet$,    this bracket   satisfies the Jacobi identity, i.e., is a Poisson bracket.

 Our approach is  inspired by Van den Bergh's  \cite{VdB}     Poisson geometry in   non-commuta\-ti\-ve algebras, see also \cite{Cb}.
  Instead  of double brackets and general linear groups as in \cite{VdB},  we work with Fox pairings and   arbitrary group schemes.
  Our  construction of   brackets   yields  as special cases
  the  Poisson   structures   on  moduli spaces of    representations of surface groups   introduced by
Atiyah--Bott  \cite{AB}   and studied by  Goldman \cite{Go1,Go2}.
  Our  construction   also yields   the   quasi-Poisson  refinements of      those structures
  due to Alekseev,  Kosmann-Schwarzbach and Meinrenken, see    \cite{AKsM,MT_dim_2,LS,Ni}.

 Most of our work applies in the    more general setting  of   graded  Hopf algebras.
 The corresponding representation algebras are also graded,  and    we    obtain
Gerstenhaber brackets rather than Poisson brackets.
This generalization combined with \cite{MT_high_dim} yields analogues of the  Atiyah--Bott--Goldman  brackets  for   manifolds of  all dimensions $\geq 3$.

 The  paper consists of  12   sections and 3 appendices. We  first recall  the
language of graded  algebras/coalgebras and related notions  (Section~\ref{Preliminaries}),
 and  we  discuss the representation algebras (Section~\ref{Representation algebras}).
 Then we introduce  Fox pairings (Section~\ref{Fox pairings and biderivations1}) and balanced biderivations (Section~\ref{Fox pairings and biderivations2}).
 We use them to define brackets in representation algebras in Section~\ref{Brackets in representation algebras}.
 In Section~\ref{bfc} we show how to obtain  balanced    biderivations from   trace-like  elements in a Hopf algebra   $B$
 and, in this case, we prove the equivariance of  our bracket  on the representation algebra $A_B$ with respect to a natural  coaction of~$B$.
 In Section~\ref{trace-like_examples}, we discuss    examples of trace-like elements arising  from classical matrix groups.
 The Jacobi identity for our  brackets    is   discussed in Section~\ref{The Jacobi identity},
  which constitutes the technical core of the paper.   In Section~\ref{The quasi-Poisson case} we study  quasi-Poisson brackets.
    In Section~\ref{invariant_subalgebra} we compute the bracket for certain $B$-invariant elements  of $A_B$.
    In Section~\ref{surfaces} we discuss the intersection Fox pairings of surfaces and the induced Poisson and quasi-Poisson brackets on moduli spaces.
In  Appendix~\ref{group_schemes}, we recall
 the basics of the theory of group schemes needed in the paper. In Appendix~\ref{relation_to_VdB} we  discuss  relations to Van den Bergh's theory.
  In Appendix~\ref{free}  we  discuss the case where $B$ is a free commutative Hopf algebra.

Throughout the paper we fix a commutative  ring~$\kk$ which
serves as the ground  ring of all   modules, (co)algebras,   bialgebras, and Hopf algebras.    In
particular, by    a    \emph{module}  we mean a    module over~$\kk$.
  For modules $X$ and $Y$, we denote by $\Hom(X,Y)$ the module  of $\kk$-linear maps $X \to Y$
and we  write $X \otimes Y$ for $X\otimes_\kk Y$. The \emph{dual} of a module $X$ is  the module $X^*=\Hom(X,\kk)$.

\section{Preliminaries}\label{Preliminaries}

We review  the graded versions of the  notions of  a  module, an  algebra, a   coalgebra, a   bialgebra, and a   Hopf algebra.
We also recall the convolution algebras   and various notions related to comodules.

\subsection{Graded  modules}\label{gram}   By a {\it graded  module} we mean a $\ZZ$-graded
module $X=\oplus_{p\in \ZZ} \, X^p$.
An element~$x$ of~$X$ is {\it homogeneous} if $x\in X^p$ for some $p$;
we call $p$ the {\it degree} of $x$ and write   $\vert x\vert =p$.  For any integer $n$, the  {\it $n$-degree} $\vert x\vert_n$ of a homogeneous element $x \in X$   is defined by  $\vert x\vert_n=  \vert x\vert+n$.
 The zero element  $0\in X$ is homogeneous and, by definition, $\vert 0\vert$ and $\vert 0\vert_n$ are arbitrary integers.

 Given graded   modules   $X$
and $Y$, a {\it graded linear map} $X\to Y$ is a linear map $X \to Y$ carrying $X^p$ to $Y^p$ for all $p\in \ZZ$.
  The tensor product $X \otimes Y$ is a graded module in the usual way:
$$
X \otimes Y  = \bigoplus_{p\in \ZZ}\,  \bigoplus_{\substack{ u,v\in \ZZ\\ u+v=p  }}\,  X^u   \otimes    Y^v.
$$

We will   identify modules without grading with graded
modules  concentrated in degree $0$. We   call such modules
\emph{ungraded}.  Similar terminology will be applied  to algebras, coalgebras, bialgebras, and Hopf algebras.

\subsection{Graded  algebras}\label{gral}

A {\it graded algebra} $A=\oplus_{p\in \ZZ} \, A^p$ is a graded module   endowed
with  an associative bilinear multiplication having a two-sided
unit $1_A\in A^0$ such that $A^p A^q\subset A^{p+q}$ for all $p,q \in \ZZ$.
For  graded algebras $A$ and $B$,   a {\it  graded algebra homomorphism} $A\to B$ is a graded  linear
map from~$A$ to~$B$ which is multiplicative  and  sends  $1_A$ to $1_B$.
The tensor product  $A \otimes B$   of   graded algebras $A$ and $B $ is  the graded algebra with underlying graded module   $A \otimes B$
and   multiplication
\begin{equation}\label{mugradedalg}
(x\otimes b) (y\otimes c)=  (-1)^{\vert b \vert\, \vert y \vert  } xy \otimes bc
\end{equation}
for any  homogeneous   $x,y\in A$ and $ b, c\in B$.

 A   graded  algebra $A$ is {\it commutative} if for any   homogeneous $x, y \in  A$, we have
\begin{equation} \label{commu}
xy=(-1)^{\vert x \vert \vert y\vert} yx .
\end{equation}
Every graded  algebra $A$ determines a commutative graded algebra  $\abelian(A)$
obtained as the quotient of $A$ by the 2-sided ideal   generated by  the expressions  $xy-(-1)^{\vert
 x \vert \vert y\vert} yx$ where $x,y$ run over all homogeneous  elements of $ A$.

\subsection{Graded    coalgebras}\label{gral++++}

 A {\it graded coalgebra} is a     graded module $A$ endowed with
  graded linear maps  ${\Delta}=\Delta_A:A\to A\otimes A$  and $\varepsilon=\varepsilon_A:A \to \kk$    such that   $\Delta$
   is a coassociative  comultiplication with counit  $\varepsilon$, i.e.,
  \begin{equation} \label{basics}
   (\Delta \otimes \id_A) \Delta= (\id_A \otimes \Delta) \Delta \quad {\rm and}
   \quad (\id_A \otimes\, \varepsilon) \Delta =\id_A = (\varepsilon  \otimes \id_A) \Delta.
\end{equation}
The graded condition  on   $\varepsilon$  means  that $\varepsilon(A^p)=0$ for all $p \neq 0$.
 The image of any ${x}\in A$ under ${{\Delta}}$ expands
(non-uniquely) as  a   sum $  \sum_i {x}'_i \otimes
 {x}''_i$ where  the index    $i$ runs over a finite set and $ {x}'_i,
 {x}''_i $ are homogeneous elements of~$A$. If $x$ is homogeneous, then we  always assume that    for all~$i$,
  \begin{equation} \label{homog} \vert x'_i \vert +\vert x''_i \vert =\vert x \vert . \end{equation} We  use
 Sweedler's notation, i.e., drop  the index  and the summation sign in the formula $\Delta(x)= \sum_i {x}'_i \otimes
 {x}''_i$ and write simply
${{\Delta}}({x})=   {x}'  \otimes {x}''$. In this notation, the second of the equalities \eqref{basics} may be rewritten as the identity
\begin{equation} \label{cocounit}
  \varepsilon  (x') x''= \varepsilon(x'') x'=x
\end{equation} for all $x\in A$.
 We will sometimes  write $x^{(1)}$ for $x'$ and  $x^{(2)}$ for $x''$,
 and we  will similarly expand the  iterated comultiplications of $x \in A$.
  For  example,  the first of the equalities \eqref{basics}  is written in this notation as
 $$ x'\otimes x'' \otimes
 x'''= x^{(1)}\otimes x^{(2)} \otimes
 x^{(3)}=(x')' \otimes (x')'' \otimes x''=x'\otimes (x'')' \otimes (x'')''.$$
A graded coalgebra~$A$ is  \emph{cocommutative} if for  any $x \in A$,
\begin{equation}\label{coco}x'\otimes x''=(-1)^{\vert x' \vert \vert x'' \vert} \, x''\otimes x' .\end{equation}

\subsection{Graded bialgebras and Hopf algebras}\label{gbaha}  A {\it graded bialgebra} is a     graded algebra $A$
endowed with
  graded algebra homomorphisms  ${\Delta}=\Delta_A:A\to A\otimes A$  and $\varepsilon=\varepsilon_A:A \to \kk$
  such that $(A, \Delta, \varepsilon)$ is a graded coalgebra. The multiplicativity of $\Delta$ implies that  for any $x,y\in A$, we have
   \begin{equation} \label{multimain} (xy)' \otimes (xy)''= (-1)^{\vert y'\vert \vert x''\vert}  x'y'  \otimes x''y''.\end{equation}
 A    graded   bialgebra $A$ is a {\it  graded  Hopf algebra}  if there is  a   graded linear map $s=s_A:A \to A$,  called the \emph{antipode}, such
  that   \begin{equation} \label{antip-prop} s(x')x'' =x' s(x'')=\varepsilon_A(x) 1_A  \end{equation}
  for all $x\in A$.
Such an $s$    is an  antiendomorphism of the underlying  graded algebra of~$A$
in the sense that $s(1_A)=1_A$ and  $s(xy)=(-1)^{\vert x \vert \vert y\vert}s(y) s(x)$ for  all homogeneous $x,y \in A$.
Also, $s$ is an  antiendomorphism  of the underlying  graded coalgebra of $A$ in the sense that $\varepsilon_A s= \varepsilon_A$ and  for  all $x \in A$,
  \begin{equation} \label{flip_antipode}(s(x))' \otimes (s(x))''= (-1)^{\vert x'\vert \vert x''\vert} s(x'') \otimes s(x').\end{equation}
  These properties of~$s$ are verified, for instance, in \cite[Lemma 2.3.1]{MT_high_dim}.

\subsection{Convolution algebras}   \label{convo}

For  a   graded coalgebra~$B$, a graded algebra $C $,   and an integer~$p$, we let $(H_B(C))^p$ be the    module of  all linear maps $f:B \to C$
such that $f(B^k) \subset   C^{k+p}  $ for all $k\in \ZZ$.
The internal direct sum
$$H=
H_B(C) =\bigoplus_{p \in \ZZ } \,  (H_B(C))^p   \subset \Hom(B,C)
$$
is  a graded module. It  carries  the following  \emph{convolution multiplication} $\ast$:
for     $f,g  \in H $, the map $f \ast g :B\to C$ is defined by  $(f  \ast g) (b)=   f (b') \, g(b'') $ for  any $ b\in B$.
Clearly, $H^p \ast H^q \subset H^{p+q}$ for any $p,q\in \ZZ$.
Hence,   the convolution  multiplication  turns  $H $ into  a graded algebra
with unit $\varepsilon_B \cdot 1_C \in H^0$.  The  map $  C\mapsto H_B(C) $  obviously  extends to an  endofunctor $H_B$ of the category of graded algebras.

For   $C=\kk$ , the convolution algebra  $ H_B(C)$  is   the \emph{dual   graded  algebra} $B^*$ of~$B$ consisting of all linear maps $f:B\to \kk$ such that $f(B^p)= 0$ for all but a finite number of $p \in \ZZ$.
 By definition, $(B^*)^p=  \Hom  (B^{-p}, \kk)$ for all $p\in \ZZ$.

As an application of the convolution multiplication,
note that the formulas \eqref{antip-prop}  say  that the antipode $s:A\to A$ in a graded Hopf algebra~$A$ is both a left and a right inverse of $\id_A$ in the   algebra $H_A(A)$.
 As a consequence,  $s$   is unique.

\subsection{Comodules}\label{comodules}

Given a   graded coalgebra    $B$,
  a   (right)    \emph{$B$-comodule} is   a  graded   module~$M$ endowed with  a graded linear map    $\Delta_M: M \to M \otimes B$ such that
\begin{equation}
\label{comodule-def} (\id_M \otimes \Delta_B)\Delta_{M} = (\Delta_{M}  \otimes \id_B)\Delta_{M} , \quad  (\id_M \otimes \varepsilon_B)\Delta_{M} =\id_M.
\end{equation}
  For   $m\in M$, we write $\Delta_{M} (m)=m^\ell \otimes m^r\in M\otimes B$
as   in Sweedler's notation (with $'$ and $''$ replaced by $\ell$ and $r$, respectively).

An element $m\in M$ is   \emph{$B$-invariant} if $\Delta_{M}(m)= m\otimes 1_B$.
  Given  $B$-comodules $M $ and $ N$, we say that
a bilinear   map   $q:M \times M \to N$ is \emph{$B$-equivariant} if for any $m_1,m_2 \in M$,
\begin{equation} \label{B-equivariance}
 q(m_1 ,m_2^\ell) \otimes  m_2^r =   q(m_1^\ell , m_2)^\ell \otimes   q(m_1^\ell, m_2)^r   s_B   (m_1^r) \ \in N \otimes B.
\end{equation}
For    $N=\kk$ with $\Delta_N(n)= n\otimes 1_B$ for all $n\in N$, the formula \eqref{B-equivariance} simplifies to
 \begin{equation} \label{B-equivariancetriv}
q(m_1 , m_2^\ell)  \, m_2^r =   q(m_1^\ell ,m_2) \,  s_B  (m_1^r) \ \in   B.
\end{equation}
A bilinear form   $q:M \times M \to \kk$ satisfying \eqref{B-equivariancetriv}  is said to be \emph{$B$-invariant}.

\section{Representation algebras}\label{Representation algebras}

 We introduce representation algebras of     graded bialgebras.

\subsection{The algebras $\widehat A_B$ and  $A_B$}  \label{A_B}

Let $A$ and~$B$ be    graded  bialgebras. We define a   graded algebra~$  \widehat A_B$  by generators and relations.
The generators are the symbols~$x_b$  where~ $x$ runs over~$A$ and $b$ runs over~$ B$,   and  the   relations are as follows:\\[-0.2cm]

 (i) \emph{The bilinearity relations:} for all  $k\in \kk$,  $x,y\in A$,  and $b,c\in {B}$,
 \begin{equation}\label{addid}
 (kx)_{b}=  x_{kb} =k \, x_{b} , \quad (x+y)_{b}=  x_{b} + y_{b}, \quad x_{b+c}= x_{b} +x_{c};
 \end{equation}

 (ii) \emph{The first multiplicativity relations:} for all   $x,y\in A$  and $b \in {B}$,
\begin{equation}\label{mult}
(xy)_{b}=    x_{b'} \, y_{b''};
\end{equation}

(iii) \emph{The first unitality relations:} for all    $b \in {B}$,
$$(1_A)_{b} = \varepsilon_B({b})\,1;$$

(iv) \emph{The second multiplicativity relations:} for all  $x\in A$ and $b,c\in B$,
\begin{equation}\label{cocomult}
x_{bc}=  x'_b\, x''_c;
\end{equation}

(v) \emph{The second unitality relations:} for all    $x\in A$,
$$x_{(1_B) }= \varepsilon_A(x) \,  1.$$

\noindent Here,  on the right-hand side of  the relations (iii) and (v), the symbol $1$  stands for the  identity   element of $  \widehat A_B$.
 The grading in $\widehat A_B$ is defined by the rule     $\vert x_b \vert=  \vert x \vert + \vert b \vert $ for all  homogeneous  $x\in A$ and   $b\in B$.
 The definition of $\widehat A_B$ is symmetric   in $A$ and $B$:   there is a     graded algebra isomorphism
 $ {\widehat A_B} \simeq \widehat B_A$ defined by $x_b\mapsto b_x$ for $x\in A$, $b\in B$.
 Clearly, the   construction  of  $  \widehat A_B$ is functorial with respect to   graded   bialgebra  homomorphisms of~$A$
 and~$B$.

 The commutative quotient $A_B=\abelian(\widehat A_B)$ of $ \widehat A_B$ is called the {\it  $B$-representation algebra of $A$}.
 It has the same generators and relations as $\widehat A_B$ with   additional commutativity relations
\begin{equation}\label{addirel} x_b\, y_c=(-1)^{\vert x_b\vert\, \vert y_c\vert  }\,  y_c \, x_b \end{equation}
 for all homogeneous $x,y\in A$ and $b,c\in B$.

 To state the universal properties of the algebras $\widehat{A}_B$ and $A_B$, we need the following definition.
A  \emph{$B$-representation of $A$ with coefficients in a graded algebra~$C$} is a  graded algebra homomorphism  $u: A\to H_B(C)\subset \Hom(B,C)$
such that for all $x\in A$ and  $b,c \in B$,
\begin{equation} \label{strange-property} {u}(x)(1_B) =
\varepsilon_A(x) 1_C \quad {\rm {and}} \quad
{u}(x)(b c)= {u}(x')(b )\cdot {u}(x'')(c) .
\end{equation}
  Let $\widehat\Rscr( C)  $ be the set  of all $B$-representations of $A$ with coefficients in $C$. For any graded algebra homomorphism $f:C\to C'$,
let $\widehat\Rscr(f): \widehat\Rscr(C)\to \widehat\Rscr(C')$  be the map that   carries a   homomorphism $u: A \to  H_B(C)$ as above to $H_B(f) \circ u : A \to  H_B(C')$.
This defines a functor  $\widehat\Rscr =\widehat\Rscr^A_B: \grAlg \to \Set$
from the category of   graded algebras and graded algebra homomorphisms $\grAlg$  to the category of sets and maps $ \Set$.
  The restriction of $\widehat\Rscr$ to the full subcategory $\grCom$ of $\grAlg$ consisting of graded commutative algebras is denoted by
  $\Rscr=\Rscr^A_B$.

\begin{lemma}\label{just--+} For any  graded algebra~$C$, there is a natural bijection
\begin{equation}\label{adjunction+ee}
\widehat\Rscr   (C)
\stackrel{\simeq}{\longrightarrow}  \Hom_{  \grAlg} (   \widehat A_B , C).
\end{equation} Consequently, for   any commutative graded algebra~$C$, there is a natural bijection
\begin{equation}\label{adjunction+eee}
\Rscr (C)
\stackrel{\simeq}{\longrightarrow}  \Hom_{  \grCom   } (    A_B , C).
\end{equation}
\end{lemma}

\begin{proof} Consider the map
$\widehat\Rscr (C) \to  \Hom_{  \grAlg} (   \widehat A_B , C)$
which  carries   a graded algebra homomorphism  $u:A \to H_B(C)$ satisfying \eqref{strange-property} to the
graded algebra homomorphism   $ v=v_u:  \widehat A_B\to C$ defined on the   generators    by   the rule    $v( x_{b })=u(x)(b)$.
We must check the compatibility of~$v$ with the defining relations (i)--(v) of $\widehat A_B$.
The compatibility with the relations (i)  follows from the   linearity of $u$.
The compatibility with the relations (ii) is verified as follows: for   $x,y\in A$ and $b\in B$,
\begin{eqnarray*}
v(  (xy)_{b})  &=& u(xy)(b)  \\
&= &   \big(u(x) * u(y)\big)   (b) \\ 
&= &  u(x)   (b' ) \, u(y)   (b'' ) \ = \ v( x_{b' }) \, v( y_{b'' }) \ =  \ v( x_{b' } y_{b'' } ).
\end{eqnarray*}
The compatibility with the relations (iii) is verified as follows: for    $b\in B$,
\begin{equation*} v( (1_A)_{b })= u(1_A)(b) = \varepsilon_B  ( b ) 1_C
=v( \varepsilon_B  ( b ) 1 ).  \end{equation*}
The compatibility with (iv) and (v) is a direct consequence of \eqref{strange-property}:  for   $x\in A$,
$$ v(x_{(1_B) }- \varepsilon_A(x) 1) =v(x_{(1_B) })- v( \varepsilon_A(x) 1)={u}(x)(1_B)- \varepsilon_A(x) 1_C=0$$
and,   for   $b,c\in B$,
$$ v(x_{bc}-  x'_b x''_c) = v(x_{bc})-   v(x'_b)\cdot  v( x''_c)= {u}(x)(b c)- {u}(x')(b )\cdot {u}(x'')(c)=0  .$$
Next, we define a map
$ \Hom_{ \grAlg } (  \widehat A_B , C) \to  \widehat\Rscr (C)$ carrying
a   graded algebra homomorphism $v: \widehat A_B\to C$    to the
graded linear  map   $ u=u_v: A \to H_B(C)$  defined by
  $u(x)(b)= v( x_{b }) $ for all $x\in A$ and $b\in B$.
  The map $u$ is     multiplicative:
  \begin{eqnarray*}
    \big(u(x) \ast u(y)\big) (b) &=&   u(x) (b') \,   u(y) (b'') \\
  &  = & v( x_{b' })\,  v( y_{b'' }) \ = \ v( x_{b' }  y_{b'' }) \ = \ v((xy)_b) \ = \ u(xy)(b)
  \end{eqnarray*}
 for any $x, y\in A$   and   $b\in B$.  Also,   $u $  carries $1_A$ to
  the map $$B\longrightarrow C,\, b \longmapsto  v( (1_A)_{b }) = v(\varepsilon_B  ( b ) 1)  = \varepsilon_B  ( b)   1_C  $$
which is   the unit of the algebra $  H_B(C)$. Thus, $u$ is a graded algebra homomorphism.
It is straightforward to verify that $u$ satisfies   \eqref{strange-property},  i.e.  $u\in \widehat\Rscr (C)$.

 Clearly,     the maps $u\mapsto v_u$ and $v\mapsto u_v$ are mutually inverse. The first of them is the required bijection~\eqref{adjunction+ee}. The naturality   is obvious from the definitions.
\end{proof}

 If both~$A$ and~$B$ are ungraded  (i.e.,    are concentrated in degree 0), then    $ \widehat A_B$ and~$A_B$ are ungraded algebras and, by
 Lemma~\ref{just--+},  the  restriction of the  functor $\Rscr^A_B$
 to the category of commutative ungraded algebras is an affine scheme  with coordinate algebra~$ A_B  $.

\subsection{The case of Hopf algebras}  \label{HopfHopf}
If $A$ and/or $B$ are Hopf algebras,  then we can say a little more about the graded algebras
  $ \widehat A_B$ and $A_B$.  We begin with the following lemma.

\begin{lemma}\label{justjust+} If $A$ and $B$ are graded Hopf algebras with antipodes $s_A$  and   $s_B$, respectively,
then   the following identity holds in $ \widehat A_B$: for any $ x\in A $ and $  b\in B $,
\begin{equation}\label{antipodeidentity} (s_A(x))_{b} = x_{s_B ({b})}. \end{equation}
 Consequently, the same identity holds in $A_B$.
\end{lemma}

\begin{proof} We  claim that for any  graded algebra~$C$ and       any  $x \in A$, $u \in \widehat\Rscr^A_B(C)$,
\begin{equation} \label{antipode_antipode}
{u}\big(s_A(x)\big) = {u}(x) \circ s_B: B \longrightarrow C .
\end{equation}
The proof  of this claim is modeled on the standard proof of the fact that a  bialgebra homomorphism of Hopf algebras commutes with the antipodes.
Namely, let
$$
U = H_{A}\big( H_B(C)) \subset \Hom\big(A,H_B(C)\big)
$$
be the convolution algebra associated to the underlying graded coalgebra of~$A$ and the graded algebra $H_B(C)$.
We denote the convolution multiplication in~$U$ by~$\star$ (not to be confused with the multiplication $\ast$ in   $H_B(C)$).
For   $u \in \widehat\Rscr^A_B(C)$, set $u^+= u s_A:A \to H_B(C)$ and let  $u^-:A \to H_B(C)$ be the map carrying any $x\in A$ to $u(x) s_B:B\to C$.
Observe that $u,u^+,u^-$ belong to $U$.
For all $x\in A$, we have
\begin{eqnarray*}
(u^+ \star u)(x) & = & u^+(x') \ast u(x'') \\
&=& u(s_A(x')) \ast u(x'') \\
& = & u(s_A(x') x'') \ = \ u\big(\varepsilon_A(x)\, 1_A\big) \ = \ \varepsilon_A(x)\, 1_{H_B(C)} \ = \ 1_U(x)  ,
\end{eqnarray*}
where the third and the fifth equalities hold because $u:A\to H_B(C)$ is an algebra homomorphism. Hence $u^+ \star u =1_U$.
Also,  for   $x\in A$ and $b\in B$, we have
\begin{eqnarray*}
(u \star u^-)(x) (b) \ = \ (u(x') \ast u^-(x''))(b)
&=& u(x')(b') \, u^-(x'')(b'') \\
& = &   u(x')   (b') \, u(x'') (s_B(b'')) \\
&= & u (x)\big(b' s_B(b'')\big) \\
& = &    u(x)( \varepsilon_B(b) 1_B) \\
& = & \varepsilon_B (b) \varepsilon_A(x)  1_C \\ 
&= &   \varepsilon_A(x)\,  1_{H_B(C)}(b) \ = \ 1_U(x)(b)
 \end{eqnarray*}
where the fourth and the sixth equalities follow  from  \eqref{strange-property}.
Hence,  $u  \star u^-= 1_U$. Using the associativity of $\star$, we conclude that
$$
u^+=u^+\star 1_U= u^+ \star u \star u^-= 1_U \star u^-=u^-.
$$
This proves the claim above. As a consequence, for any $x\in A$,  $b\in B$  and any graded algebra homomorphism~$v$ from $ \widehat A_B$ to a   graded algebra~$C$, we have
$$v\big((s_A(x))_{b} - x_{s_B ({b})}\big)= v\big((s_A(x))_b\big)- v(x_{s_B(b)}) =\big({u} (s_A(x) ) - {u}(x) \circ s_B\big)(b)=0,$$
where $u=u_v: A\to H_B(C)$ is the  graded algebra homomorphism   corresponding to~$v$ via~\eqref{adjunction+ee}.
Taking $C= \widehat A_B$ and $v=\id$, we obtain that $(s_A(x))_{b} - x_{s_B ({b})}=0$.
\end{proof}

 Given an ungraded   bialgebra~$B$,
 a \emph{(right) $B$-coaction} on a graded algebra~$M$  is   a graded algebra homomorphism     $\Delta   = \Delta_M   : M \to M \otimes B$
  satisfying \eqref{comodule-def}, i.e$.$, turning $M$ into  a (right) $B$-comodule.

\begin{lemma} \label{coaction_M}
Let $A$  be a graded bialgebra and  $B$ be an ungraded  commutative Hopf algebra.
The   graded algebra $ \widehat A_B$  has a unique $B$-coaction $\Delta: \widehat A_B \to \widehat A_B \otimes B$
such that
\begin{equation} \label{Delta_M}
\Delta (x_b) = x_{b''} \otimes s_B(b') b''' \quad \text{for any} \quad x\in A, \, b\in B.
\end{equation}
 Consequently, the graded algebra $A_B= \abelian (\widehat A_B)$  has a unique $B$-coaction   satisfying~\eqref{Delta_M}.
\end{lemma}

\begin{proof}
We first prove that \eqref{Delta_M} defines an  algebra homomorphism $\Delta: \widehat A_B \to \widehat A_B \otimes B$.
The compatibility with the bilinearity relations in the definition of $\widehat A_B$ is obvious. We check the
  compatibility with the  first multiplicativity relations: for  $x,y \in A$ and $b\in B$,
\begin{eqnarray*}
\Delta(x_{b'}) \, \Delta( y_{b''}) &=& \big(x_{b^{(2)}} \otimes s_B(b^{(1)}) b^{(3)}\big) \big(y_{b^{(5)}} \otimes s_B(b^{(4)}) b^{(6)}\big) \\
&= & x_{b^{(2)}}  y_{b^{(5)}}\otimes s_B(b^{(1)}) b^{(3)} s_B(b^{(4)}) b^{(6)} \\
&=&  x_{b^{(2)}}  y_{b^{(3)}} \otimes  s_B   (b^{(1)})  b^{(4)}
\ = \  (xy)_{b''} \otimes s_B(b') b''' \ = \ \Delta((xy)_b) .
\end{eqnarray*}
The compatibility with the  first unitality relations:   for   $b\in B$,
\begin{eqnarray*}
\Delta\big((1_A)_b\big) \ = \ (1_A)_{b''} \otimes s_B(b') b''' &=& \varepsilon_B(b'') 1 \otimes s_B(b') b''' \\
&=&  1 \otimes s_B(b') b'' \\
& = & \varepsilon_B(b)\, 1 \otimes 1_B  \ =\  \Delta( \varepsilon_B(b) 1).
\end{eqnarray*}
 The compatibility with the  second multiplicativity relations: for   $x \in A$ and ${b,c \in B}$,
\begin{eqnarray*}
\Delta(x'_b)\, \Delta(x''_c) &=&  \big(x'_{b''} \otimes s_B(b') b'''\big)\,  \big(x''_{c''} \otimes s_B(c') c'''\big) \\
&= & x'_{b''}  x''_{c''} \otimes  s_B(b')  b'''  s_B(c') c''' \\
&=& x_{b''c''} \otimes s_B(c') s_B(b')  b'''c'''
\ = \ x_{b''c''} \otimes s_B(b'c') b'''c'''  \ = \ \Delta(x_{bc}) ,
\end{eqnarray*}
where in the third equality we use the commutativity of~$B$.
  Finally, the compatibility with  the second unitality relations:
  for any $x\in A$, $$\Delta(x_{(1_B)}) = x_{(1_B)} \otimes 1_B=\varepsilon_A(x) 1 \otimes 1_B= \Delta( \varepsilon_A(x) 1) .$$

We now verify \eqref{comodule-def}.
Since $\Delta$ and $\Delta_B$ are algebra homomorphisms, it is enough to check \eqref{comodule-def}  on the generators. For any $x\in A$ and $b\in B$,
\begin{eqnarray*}
 (\id_{\widehat A_B} \otimes \Delta_B) \Delta (x_b) &=& x_{b''} \otimes \Delta_B\big( s_B(b') b'''\big) \\
 &=& x_{b''} \otimes \Delta_B(s_B(b')) \Delta_B(b''')  \\
&=& x_{b^{(3)}} \otimes  s_B(b^{(2)})\,   b^{(4)} \otimes s_B(b^{(1)})\, b^{(5)}\\
&=& \Delta(x_{b''}) \otimes s_B(b') b''' \ = \ ( \Delta \otimes \id_B) \Delta(x_b)
\end{eqnarray*}
 and
\begin{eqnarray*}
 (\id_{\widehat A_B} \otimes \varepsilon_B)\Delta(x_b)
  =   \varepsilon_B \big( s_B(b') b'''\big) x_{b''}
 =    \varepsilon_B   (b') \, \varepsilon_B( b''' ) x_{b''}
   =   \varepsilon_B   (b') x_{  b''}
 =  x_b .
\end{eqnarray*}
The last claim of the lemma follows from the fact that    any  $B$-coaction on a graded algebra~$M$ induces a $B$-coaction on the commutative graded algebra $\abelian (M)$.
\end{proof}

\subsection{Example:  from  monoids to representation algebras}\label{ex_monoid_algebras}

Given a     monoid~$G$, we let $\kk G $ be the  module freely generated by the set $G$.
Multiplications in~$\kk$ and~$G$   induce  a bilinear multiplication in $\kk G$ and turn   $\kk G$
into an ungraded bialgebra    with   comultiplication carrying each   $g\in G$ to  $  g\otimes g $ and with   counit carrying all     $g\in G$ to~$1_\kk$.
If~$G$ is finite, then we can  consider the dual  ungraded  bialgebra $B= (\kk G)^*=  \Hom   (\kk G, \kk)$ with basis $ \{\delta_g\}_{g\in G}$   dual to the basis $G$ of $\kk G$.
Multiplication in~$B$ is computed  by    $\delta_g^2=\delta_g$   for all $g\in G$ and  $\delta_g \delta_h=0$   for   distinct  $g,h\in G$.
Comultiplication in $B$ carries each  $\delta_g$ to   $  \sum_{h,j\in G, hj=g} \delta_h\otimes \delta_j$.
 The unit of $B$ is $\sum_{g\in G} \delta_g$ and the counit    is the evaluation on the    neutral   element of~$G$.

Consider now a  monoid $\Gamma$ with    neutral     element   $\eta$   and  a finite monoid $G$ with   neutral  element   $n$. Consider
   the ungraded   bialgebras $A=\kk \Gamma$ and    $B=(\kk G)^*$. By definition, the representation algebra $A_B$ is the ungraded commutative algebra generated by the symbols $x_g = x_{(\delta_g)}$ for all $x\in \Gamma, g\in G$,
subject to the relations
$$  \eta_g= \left\{ \begin{array}{ll}  1 & \hbox{if } g=  n  , \\ 0 & \hbox{if } g\neq   n  , \end{array}\right.
\ (xy)_g=\sum_{h,j\in G, hj=g} x_h y_j \quad \hbox{for all} \quad x,y \in \Gamma,\, g\in G,$$
and
$$
x_{ n  } = 1, \  \ x_g x_h=  \left\{\begin{array}{ll} x_g & \hbox{if } g=h, \\ 0 & \hbox{if } g\neq h,  \end{array}\right.
\quad \hbox{for all} \quad  x\in \Gamma,\, g,h  \in G.
$$
 Identifying the algebra $H_B(\kk)=B^*$ with $\kk G$ in the natural way,
we  identify the  set $\Rscr^A_B (\kk)$ with the set of multiplicative homomorphisms $\Gamma \to   \kk G$
whose image consists of  elements   $\sum_{g\in G} k_g g\in \kk G$ such that    $k_g^2=k_g$ for all $ g\in G$, $ k_g k_{ h}=0$ for distinct $ g,h\in G$, and $\sum_{g\in G} k_g=1$.
If $\kk$ has no zero-divisors, then this is   just  the set of monoid homomorphisms $ \Gamma \to G $.
  Then   the formula \eqref{adjunction+eee} gives a bijection
\begin{equation}\label{adjunction+ee---}
\Hom_{  \Mon  }(\Gamma, G)
\stackrel{\simeq}{\longrightarrow}  \Hom_{ \Com } (A_B , \kk),
\end{equation} where   $\Mon$ is the category of monoids  and monoid homomorphisms
and $\Com$ is the category of  commutative ungraded algebras
and algebra homomorphisms.
This  computes    the set $\Hom_{\Mon}(\Gamma, G)$   from $A_B$.

 This example   can be generalized   in terms of monoid schemes  (recalled in  Appendix~\ref{Monoid schemes}).
Let   $A=\kk \Gamma$ be the   bialgebra  associated with a  monoid~$\Gamma$
and let   now   $B=\kk[\Gscr]$ be the  coordinate   algebra  of a  monoid   scheme $\Gscr$;   both~$A$ and~$B$ are ungraded bialgebras.
We claim that, for any ungraded commutative algebra~$C$, there is a natural bijection of the set $\Rscr^A_B(C)$
onto the set of monoid homomorphisms $\Hom_{ \Mon  }(\Gamma, \Gscr(C)).$
Indeed, given an   algebra homomorphism    $u: A \to H_B(C)$,
the condition  \eqref{strange-property}  holds for all $x\in A$ if and only if it holds for all $x\in \Gamma\subset A$.
For $x \in \Gamma$, the condition \eqref{strange-property} means that  $u(x): B \to C $ is an algebra homomorphism, i.e$.$ $u(x)\in \Gscr(C)\subset H_B(C)$.
Then the map  $u\vert_\Gamma:\Gamma \to \Gscr(C)$ is a   monoid homomorphism. This  implies our claim.
  This claim and  Lemma~\ref{just--+} show   that   the functor $C\mapsto \Hom_{ \Mon }(\Gamma,\Gscr(C))$ is an affine scheme with coordinate algebra $(\kk \Gamma)_{\kk[\Gscr]}$.

For instance, if  $\Gamma$ is   the   monoid  freely generated by a single element   $x$,
then for any monoid   scheme $\Gscr$, the functor $\Hom_{\Mon}(\Gamma,\Gscr(-))$  is naturally isomorphic to  $\Gscr$.
On the level of coordinate algebras, this corresponds to the algebra isomorphism
\begin{equation*} \label{rank_one}
(\kk \Gamma)_{  \kk[\Gscr]  } \stackrel{\simeq}{\longrightarrow} \kk[\Gscr] , \,   x_b  \longmapsto  b.
\end{equation*}
  If $\Gscr$ is a group scheme, then $B=  \kk[\Gscr]$  is a Hopf algebra and  this  isomorphism   transports the $B$-coaction \eqref{Delta_M}
into the   usual    (right) \emph{adjoint coaction}  of  $B$ on   itself defined by
\begin{equation} \label{adjoint_coaction}
B  {\longrightarrow} B\otimes B,\,\, b   \longmapsto    b'' \otimes s_B(b')\, b''' .
\end{equation}

\subsection{Example:   from Lie algebras to representation algebras}  \label{enveloping_algebras}

Let  $A=U(\mathfrak{p})$ be the enveloping algebra  of a Lie
algebra $\mathfrak{p}$, and let $B=\kk[\Gscr]$ be the coordinate
algebra of an infinitesimally-flat group scheme $\Gscr$ with Lie
algebra $\mathfrak{g}$.  (See   Appendix~\ref{group_schemes} for the terminology.)
Both~$A$ and~$B$ are ungraded Hopf algebras. We claim that, for any ungraded
commutative algebra~$C$, there is a natural bijection of the set
$\Rscr^A_B(C)$ onto the set of Lie algebra homomorphisms
${\Hom_{\Lie}(\mathfrak{p},\mathfrak{g} \otimes C)}$.
   Indeed,  given an   algebra homomorphism    $u: A \to H_B(C)$,
  the condition  \eqref{strange-property} holds  for all $x\in A$ if and only if it holds  for all $x\in \mathfrak{p} \subset A$.
  For $x \in \mathfrak{p}$, the  condition \eqref{strange-property}    means that  $u(x): B \to C$ is a derivation
  with respect to the structure of $B$-module in~$C$ induced by the counit   of $B$.
   By \eqref{inf-flat}, this is equivalent to the inclusion  $u(x) \in \mathfrak{g} \otimes    C$.
  Then the map  $u\vert_{\mathfrak{p}}:\mathfrak{p}  \to \mathfrak{g}   \otimes    C$ is a Lie algebra homomorphism. This   implies our claim.
  This claim and  Lemma~\ref{just--+} show   that  the functor $C\mapsto \Hom_{\Lie}(\mathfrak{p},\mathfrak{g}  \otimes    C)$ is an affine scheme
with coordinate algebra $(U(\mathfrak{p}))_{\kk[\Gscr]}$.

 \section{Fox pairings}\label{Fox pairings and biderivations1}

We recall  the theory of  Fox pairings from  \cite{MT_dim_2}.

 \subsection{Fox pairings and transposition}

 Let $A$  be  a graded Hopf  algebra with  counit   $\varepsilon=\varepsilon_A$   and invertible antipode $s=s_A$.
Following  \cite{MT_dim_2},     a  {\it Fox pairing of degree $n\in \ZZ$} in~$A$  is a bilinear map   $\rho:A \times A \to A$ such
that $\rho (A^p, A^q)\subset A^{p+q+n}$ for all $p,q\in \ZZ$ and
\begin{equation}\label{Fox1}
\rho (x, y  z)= \rho (x, y ) z + \varepsilon  (y ) \rho (x, z)  ,
\end{equation}
\begin{equation}\label{Fox2}
\rho (x y, z) =\rho (x , z)\,  \varepsilon (y) +  x    \rho (y, z)
\end{equation}
for any $x,y,z\in A$.   These conditions imply  that $\rho(1_A, A)=\rho(A, 1_A)=0$.

 The \emph{transpose} of a Fox pairing $\rho: A \times A \to A $   of degree $n$ is the bilinear map  $\overline{\rho}: A \times A \to A$    defined by
\begin{equation} \label{transpose_formula}
\overline\rho(x,y) = (-1)^{\vert x \vert_n \vert y \vert_n} s^{-1} \rho\big(s(y),s(x)\big)
\end{equation}
for any homogeneous $x,y \in A$.

\begin{lemma}\label{transpose}
The transpose     of a Fox pairing  of degree~$n$   is a Fox pairing   of degree~$n$.
\end{lemma}

\begin{proof} Let $\rho$ be a Fox pairing  of degree~$n$ in~$A$.
  For   any homogeneous  $x, y,z \in A$,
\begin{eqnarray*}
 \overline{\rho}(x y,z) &=&  (-1)^{ \vert x y \vert_n \vert z\vert_n} s^{-1}\rho\!\left(s(z),s(x  y) \right)\\
&=& (-1)^{ \vert xy \vert_n \vert z\vert_n + \vert x \vert \vert y \vert}  s^{-1}\rho\!\left(s (z),s(y) s( x) \right)  \\
&=& (-1)^{ \vert xy \vert_n \vert z\vert_n + \vert x \vert \vert y \vert}  \, s^{-1} \big(\rho\!\left(s (z),s(y)  \right)  s( x)   +\varepsilon (s(y)) \, \rho\!\left(s (z),s(x)  \right) \big) \\
&=& (-1)^{ \vert y \vert_n \vert z\vert_n} x\,  s^{-1}  \rho\!\left(s (z),s(y)  \right)+  (-1)^{ \vert x \vert_n \vert z\vert_n} \varepsilon (y)\, s^{-1}  \rho\!\left(s (z),s(x)  \right) \\
&=& x  \overline{\rho}(y,z)+  \overline{\rho}(x,z) \, \varepsilon (y).
\end{eqnarray*}
This verifies \eqref{Fox2}, and \eqref{Fox1} is verified similarly.
 That $\overline \rho$ has degree~$n$ is obvious.
\end{proof}

We say that a  Fox pairing $\rho$ in~$A$   is \emph{antisymmetric} if $\overline \rho=-\rho$.
It is especially easy to produce antisymmetric Fox pairings in the case of involutive~$A$.
Recall that a graded Hopf algebra   $A$   is \emph{involutive} if its antipode   $s=s_A$   is an involution.
  For instance,   all  commutative graded Hopf algebras and all cocommutative graded Hopf algebras are involutive.
 In this case, we have  $\overline{\overline \rho }=\rho$  for any Fox pairing $\rho$  and, as a consequence, the Fox pairing $\rho -\overline \rho$ is antisymmetric.

   \begin{lemma}\label{transpose++++}
Let $\rho$ be  a Fox pairing  of degree~$n$  in  a cocommutative  graded Hopf algebra~$A$ with antipode $s=s_A$ and counit
$\varepsilon=\varepsilon_A$.  Then for any $x,y\in A$, we have
\begin{equation}\label{+++s}  \rho (s(x),s(y) )
  =      s(x')\,  \rho(x'',y') \, s(y'')    . \end{equation}
If $\rho$ is antisymmetric, then so is the bilinear form  $\varepsilon \rho: A \times A \to \kk$.
\end{lemma}

\begin{proof} We have
\begin{eqnarray*}
 0 \ = \ {\rho}(\varepsilon (x)  1_A  ,y)&=&   \rho(s(x') x'',y)\\
&=&  \rho(s(x') ,y)  \, \varepsilon  (x'') +s(x') \rho(  x'',y)  \\
&=&  \rho(s(x' \varepsilon  (x'')) ,y)    +s(x') \rho(  x'',y) \\
&  = & \rho(s(x),y)+   s(x') \rho(  x'',y)  .
\end{eqnarray*}
Therefore
\begin{equation}\label{firstrho+}
\rho(s(x),y)=-  s(x') \rho( x'',y).
\end{equation}
A similar computation shows that
$ \rho(x ,y)=- \rho( x, s(y'))y'' $. Replacing  here  $y$ by $s(y)$ and using the involutivity of  $s$ and the cocommutativity of $A$, we obtain  that
\begin{equation}\label{firstrho++}
\rho(x, s(y))=-  (-1)^{\vert y' \vert \vert y'' \vert}  \rho( x, y'' ) s(y')=  -   \rho( x, y' ) s(y'') .
\end{equation}
The formulas \eqref{firstrho+} and  \eqref{firstrho++} imply  that
$$
\rho (s(x),s(y) ) = -      s(x')\,  \rho(x'', s(y) )  =      s(x') \rho(x'',y') s(y'')    .
$$

  If we now assume that   $\overline   \rho= -\rho$, then  for any homogeneous $x,y \in A$, we have
\begin{eqnarray*}
\varepsilon  \rho(x,y) & \by{transpose_formula}&  - (-1)^{\vert x \vert_n \vert y \vert_n} \varepsilon \rho\big(s (y),s (x)\big)  \\
& \by{+++s} &  - (-1)^{\vert x \vert_n \vert y \vert_n} \varepsilon (y')\, \varepsilon  \rho(y'',x')\, \varepsilon (x'') \ = \   - (-1)^{\vert x \vert_n \vert y \vert_n}  \varepsilon  \rho(y,x).
\end{eqnarray*}

\up
 \end{proof}

 \subsection{Examples}\label{Examples-rho}

1.  Given a graded Hopf algebra $A$, any   $a\in A^n$ with $n\in \ZZ$  gives rise to  a    Fox pairing $\rho_a$ in~$A$ of degree $n$ by
\begin{equation}
\label{trivial_Fox}
 \rho_a  (x,y)=\big(x -\varepsilon_A   (x)\, 1_A  \big) \,  a \, \big(y- \varepsilon_A   (y)\,   1_A  \big)
\end{equation} for any $x,y\in A$. If   $ s_A  (a)=(-1)^{ n+1 }a$, then $\rho_a$   is antisymmetric.

  2. Consider the tensor algebra
  $$
  A  =T(X)=\bigoplus_{p \geq 0} \, X^{\otimes p}
  $$
  of an ungraded  module~$X$, where the $p$-th homogeneous summand  $X^{\otimes p}$ is the tensor product    of~$p$ copies  of~$X$.
 We provide~$A$ with    the  usual   structure of a   cocommutative graded  Hopf algebra,
 where
$\Delta_A (x)=x\otimes 1+ 1\otimes x$, $\varepsilon_A (x)=0$, and $s_A(x)=-x$ for all $x  \in X=X^{\otimes 1} $.
Each  bilinear   pairing    ${\ast:X \times X \to X}$  extends uniquely  to a Fox pairing $\rho_\ast$ of degree~$-1$ in~$A $ such that $\rho_\ast(x,y)=x\ast y$
for all $x,y \in X $. It is easy to see that  $\rho_\ast$ is antisymmetric if and only if~$\ast$ is   commutative.

3.  Let $M$ be a smooth oriented  manifold of dimension $d>2$ with non-empty boundary.
Suppose  for simplicity that  the ground ring $\kk$ is a field and consider the graded algebra  $ H(\Omega ;\kk )$,
where     $\Omega $ is  the loop space of   $M$  based at a point  of $  \partial M$
and  $H(-;\kk)$ is the    singular homology of a space with coefficients in $\kk$.
  Using intersections of   families of loops   in $M$,
we define in  \cite{MT_high_dim}    a canonical operation in~$ H (\Omega ;\kk )$   which is equivalent (see Appendix~\ref{FPvDB})
to an   antisymmetric  Fox pairing   of degree  $  2-d$ in~$ H (\Omega ;\kk )$.
A parallel construction for surfaces is quite elementary; it  will be   reviewed and   discussed   in Section~\ref{surfaces}.

  \section{Balanced biderivations}\label{Fox pairings and biderivations2}

 We      introduce   balanced biderivations in ungraded Hopf algebras.

 \subsection{Biderivations}\label{Balanced bilinear forms} Let  $B$ be an ungraded   algebra endowed
with   an algebra homomorphism $\varepsilon: B \to \kk$. A linear map $\mu:B\to
\kk$ is   a \emph{derivation} if $\mu(bc)= \varepsilon(b) \mu (c)+
\varepsilon(c) \mu (b)$  for all $b,c \in B$. Clearly,    $\mu $
is a derivation if and only if $\mu(1_B+I^2)=0$, where $I^2\subset B$
is the square of the ideal $I=\Ker(\varepsilon)$ of $ B$.
A bilinear form $\bullet:B \times B \to \kk$   is a  \emph{biderivation} if it is a derivation in each variable, i.e.,
\begin{eqnarray}
\label{bider}
 (bc) \bullet d &=&    \varepsilon  (b)\, c\bullet d+     \varepsilon   (c) \, b \bullet  d,\\
 b \bullet (cd) &=&  \varepsilon  (c)\, b\bullet d+     \varepsilon   (d) \, b \bullet  c
 \end{eqnarray}
 for any $b,c,d \in B$.  Clearly, $\bullet$ is a biderivation if and only if both its left and right annihilators contain $1_B+I^2$.
 Thus, there is a one-to-one correspondence\\[-0.3cm]
\begin{equation} \label{correspondence}
\xymatrix{
\left\{\hbox{\small biderivations in $B$} \right\}  \ar@/^1.5pc/[rr]^-{\hbox{\scriptsize restriction}}
&& \ar@/^1.5pc/[ll]^-{\hbox{\scriptsize pre-composition with $p\times p$}} \left\{\hbox{\small bilinear forms in $I/I^2$} \right\}
}
\end{equation}
where $p: B \to I/I^2$ is the  linear map defined by $p(b) = b-\varepsilon(b) 1_B  \!\! \mod I^2$ for any~$b\in B$.

 For further use, we state  a well-known method   producing a presentation of the module~$I/I^2$
by  generators and relations from a presentation of the algebra $B$ by generators and relations. Suppose that   $B$
is generated by a set $X \subset B$ and that $R $ is a set of defininig relations for $B$ in these generators.
Then the  vectors $\{p(x)\}_{x\in X} $ generate the module $ I/I^2$.
Each relation $r\in R$ is a non-commutative polynomial in the variables $x\in X$ with coefficients in $\kk$.
Replacing every  entry of $x$ in $r$ by $\varepsilon (x)+ p(x)$ for all $x\in X$, and taking the linear part of the resulting polynomial, we obtain
a formal  linear combination of the symbols $\{p(x)\}_{x\in X} $ representing zero in $I/I^2$. Doing this for all $r\in R$,  we obtain
a set of defining relations for the module $I/I^2$ in the generators $\{p(x)\}_{x\in X} $.

\subsection{Balanced bilinear forms}\label{Balanced bilinear forms2}
Let  $B$ be an ungraded Hopf algebra with   counit $\varepsilon=\varepsilon_B$ and antipode $s=s_B$.
A~bilinear form $\bullet: B \times B \to \kk$ is  \emph{balanced} if
\begin{equation} \label{balanced}
(b \bullet c'')\,  s(c') c'''  =   (c  \bullet  b'' )\,   s(b''') b'
\end{equation}
for any $b,c \in B$.  Balanced forms are symmetric: to see it,  apply~$\varepsilon $ to both sides of  \eqref{balanced}.
The following lemma gives  a  useful reformulation of   \eqref{balanced}    for commutative~$B$.

\begin{lemma}
A   bilinear form $\bullet: B \times B \to \kk$   in a commutative  ungraded  Hopf algebra $B$ is balanced  if and only if for any $b,c\in B$,
\begin{equation} \label{balanced_bis}
(b'' \bullet c')\, b' s(c'')  = \   (c'' \bullet b' )\,  s(c') b''.
\end{equation}
\end{lemma}

\begin{proof}
For any $b, c \in B$, we have
\begin{eqnarray*}
(b'' \bullet c')\, b' s(c'') &=& (b^{(2)} \bullet c')\, \varepsilon(b^{(3)}) b^{(1)} s(c'') \\
 &=& (b^{(2)} \bullet c')\, s(b^{(3)}) b^{(1)} b^{(4)} s(c'') \\
 & \by{balanced}  &    (c^{(2)} \bullet b' )\,   s(c^{(1)}) c^{(3)} b'' s(c^{(4)}) \\
 &=&    (c^{(2)} \bullet b')\,  s(c^{(1)}) \varepsilon(c^{(3)}) b'' \ = \   (c'' \bullet b')\,  s(c') b''.
\end{eqnarray*}
Conversely,
\begin{eqnarray*}
(b \bullet c'') s(c') c''' &=& (b' \bullet c'')\, s(c') \varepsilon(b'') c''' \\
 &=& (b' \bullet c'')\, s(c') b'' s(b''') c''' \\
& \by{balanced_bis}&  ( c' \bullet b'')\,   b' s(c'') s(b''') c'''\\
&=&  (c' \bullet b'' )\,   b'  s(b''') \varepsilon(c'') \ = \  (c \bullet b'' )\,   b' s(b''').
\end{eqnarray*}

\up
\end{proof}

 We will     mainly   consider balanced biderivations   in    commutative    ungraded Hopf algebras.
Examples of   balanced biderivations will be given   in Sections~\ref{bfc} and~\ref{trace-like_examples}.

\subsection{Remarks} \label{about_bd}

Let $B$ be an ungraded Hopf algebra.

1. If $B$ is cocommutative, then all symmetric  bilinear forms in $B$ are balanced.

2. Assume that $B$ is commutative.
It follows from the definitions that  a symmetric bilinear form in  $B$ is balanced  if and only if it is  $B$-invariant
with respect to the adjoint  coaction   \eqref{adjoint_coaction} of~$B$.
This is  equivalent to the invariance under the conjugation   action of  the  group  scheme associated with~$B$,
see  Appendices~{\ref{induced_actions}--\ref{equivariance}}.

\section{Brackets in representation algebras} \label{Brackets in representation algebras}

In this section, we   construct  brackets  in   representation algebras.

\subsection{Brackets}

  Let $n\in \ZZ$.    An  \emph{${{n}}$-graded bracket}     in a graded algebra $A$ is a  bilinear map
$\bracket{-}{-} :A\times A\to A$ such that
$\bracket{A^p}{A^q}\subset A^{p+q+{{n}}}$
 for all $p,q\in \ZZ$ and   the
following  {\it ${{n}}$-graded Leibniz rules} are met   for all homogeneous $x,y,z\in A$:
\begin{eqnarray}
\label{poisson1} \bracket{x}{yz} &= &\bracket{x}{y}z + (-1)^{\vert x\vert_{{n}} \vert y\vert} \, y \bracket{x}{z},\\
\label{poisson2} \bracket{xy}{z} &= & x  \bracket{y}{z}   +
(-1)^{\vert y\vert \vert z\vert_{{n}}}     \bracket {x}{z} y.
\end{eqnarray}
An ${{n}}$-graded bracket  $\bracket{-}{-}$ in $A$  is {\it  antisymmetric}
if for all homogeneous $x,y  \in A$,
\begin{equation}\label{antis} \bracket{x}{y}=- (-1)^{ \vert x \vert_{{n}} \vert
y\vert_{{n}}} \bracket{y}{x}.
\end{equation}
For an antisymmetric bracket,  the identities \eqref{poisson1} and \eqref{poisson2} are equivalent to   each other.

 Given an $n$-graded bracket $\bracket{-}{-}$ in a graded algebra~$A$, the {\it   Jacobi identity} says that
   \begin{equation}\label{jaco}
   (-1)^{ \vert x \vert_{{n}} \vert z \vert_{{n}}} \bracket {x}  {\bracket{y}{z}}
+ (-1)^{ \vert x \vert_{{n}}\vert y \vert_{{n}}} \bracket  {y} {\bracket{z}{x}}
+(-1)^{  \vert y \vert_{{n}} \vert z \vert_{{n}} } \bracket   {z} {\bracket{x}{y}}    =0
\end{equation}
for all homogeneous $x,y,z\in A$. An antisymmetric   $n$-graded bracket satisfying the Jacobi identity
 is called a {\it  Gerstenhaber  bracket   of degree $n $.}
 Gerstenhaber  brackets  of degree $0$ in  ungraded algebras are called   {\it   Poisson brackets}.

\subsection{The main construction}\label{The main construction}

 We   formulate  our main construction  which,  under certain assumptions on Hopf algebras~$A$ and~$B$,
produces  a bracket in   $A_B$    from   an antisymmetric  Fox pairing in~$A$ and a balanced biderivation  in~$B$.

\begin{theor} \label{double_to_simple}
Let~$\rho$ be an antisymmetric Fox pairing    of degree $n\in \ZZ$ in a cocommutative  graded Hopf algebra~$A$.
Let~$\bullet$ be a balanced biderivation in a commutative ungraded Hopf algebra~$B$.
  Then   there  is a unique $n$-graded bracket $\bracket{-}{-}$ in  $A_B  $    such that
\begin{eqnarray}
\label{main_formula} \bracket{x_b}{y_c}    &=&        (-1)^{ \vert   x''\vert   \vert y' \vert_n}  (c''\bullet b^{(2)})\, \rho(x',y')_{ {{  s_B }}(b^{(3)})\, b^{(1)}}\,   x''_{b^{(4)}}\,   y''_{c' }
\end{eqnarray}
for all   $x,y \in A$ and   $b,c \in B$. This $n$-graded bracket   is antisymmetric.
\end{theor}

\begin{proof}
Observe first that   the condition \eqref{balanced}  allows us to rewrite the formula \eqref{main_formula} in the following equivalent form
\begin{eqnarray}
\label{b_Fox_2C} \bracket{x_b}{y_c}    &=&      (-1)^{ \vert   x''\vert   \vert y' \vert_n}  (b'\bullet c^{(3)})\, \rho(x',y')_{ { {s}_B }(c^{(2)})\, c^{(4)}}\,   x''_{b''} \,  y''_{c^{(1)}}\, .
\end{eqnarray}
  Every graded module $X$ determines a graded tensor algebra
  $T(X )=  \bigoplus_{k \geq 0}  \, X^{\otimes k}$ with the grading
 $$\vert x_1 \otimes x_2 \otimes  \cdots \otimes x_k \vert =\vert x_1 \vert  + \vert x_2 \vert  + \cdots + \vert
 x_k \vert$$ for any $k\geq 0$ and any homogeneous $ x_1,  \dots  , x_k \in X$.
Applying this construction to   $X=A \otimes B$, we obtain a graded algebra   $T=T\!\left(A\otimes B\right)$. For any $x\in A$ and $ b\in B$, we set $x_b=x\otimes b \in X \subset T$.
  Let $\pi:T\to A_B$ be  the   projection carrying each such~$x_b $ to the corresponding generator~$x_b$ of~$ A_B$.   It follows from the definition of~$T$ that the formula  \eqref{main_formula} defines uniquely   a  bilinear map
$\bracket{-}{-}:T \times T \to A_B$  such that for all homogeneous
$\alpha, \beta, \gamma\in T$,
\begin{eqnarray}
\label{poisson1new} \bracket{\alpha}{\beta \gamma} &= &\bracket{\alpha}{\beta} \pi (\gamma) + (-1)^{\vert \alpha\vert_{{n}} \vert \beta\vert} \, \pi (\beta) \bracket{\alpha}{\gamma},\\
\label{poisson2new} \bracket{\alpha \beta}{\gamma} &= & \pi
(\alpha) \bracket{\beta}{\gamma} + (-1)^{\vert \beta\vert \vert
\gamma\vert_{{n}}} \bracket {\alpha}{\gamma} \pi (\beta).
\end{eqnarray}
It is clear from the definitions that $\bracket{T^p}{T^q} \subset
A_B^{p+q+n}$  for any $p,q \in \ZZ$.

We check now that  for any homogeneous
$\alpha, \beta \in T$,    $$\bracket{\beta}{\alpha}
   = - (-1)^{\vert \alpha \vert_n \vert
\beta \vert_n} \bracket{\alpha}{\beta  }.$$
 In view of the Leibniz rules
\eqref{poisson1new}  and   \eqref{poisson2new}, it suffices to verify
this equality for the generators $\alpha=x_b$ and $\beta=y_c$ with
homogeneous $x,y\in A$ and $b,c \in B$. In this computation and in the rest of the proof, we denote the (involutive) antipodes in~$A$ and~$B$ by the same letter~$s$; this should not lead to a confusion. We have
  \begin{eqnarray*}
 && \bracket{y_c}{x_b} \\
&\stackrel{ \eqref{main_formula}}{=}&    (-1)^{ \vert   y''\vert   \vert x' \vert_n}  (b''\bullet c^{(2)})\, \rho(y',x')_{ {{s}}(c^{(3)}) c^{(1)}}\,   y''_{c^{(4)}}\,   x''_{b' } \\
 &\stackrel{  \eqref{transpose_formula}  }{=}&  -  (-1)^{ \vert   y''\vert   \vert x' \vert_n + \vert   y'\vert_n   \vert x' \vert_n}  (b''\bullet c^{(2)})\, (s\rho(s(x'), s(y')))_{ {{s}}(c^{(3)}) c^{(1)}}\,   y''_{c^{(4)}}\,   x''_{b' } \\
&\stackrel{ \eqref{antipodeidentity}}{=}&  -  (-1)^{ \vert   y \vert_n   \vert x' \vert_n }  (b''\bullet c^{(2)})\, ( \rho(s(x'), s(y')))_{ {{s}}(c^{(1)}) c^{(3)}}\,   y''_{c^{(4)}}\,   x''_{b' }   \\
&\stackrel{ \eqref{+++s}}{=}&    -  (-1)^{ \vert   y \vert_n    \vert x'x'' \vert_n   }  (b''\bullet c^{(2)})\, ( s(x') \rho(x'',  y' ) s(y''))_{ {{s}}(c^{(1)}) c^{(3)}}\,   y'''_{c^{(4)}}\,   x'''_{b' } \\
&\stackrel{  \eqref{flip_antipode}, \eqref{mult}}{=}&    -  (-1)^{ \vert   y \vert_n      \vert x'x'' \vert_n    }  (b''\bullet c^{(4)})\,   \\
&& \quad s(x')_{ {{s}}(c^{(3)}) c^{(5)}} \rho(x'',  y' )_{ {{s}}(c^{(2)}) c^{(6)}} s(y'')_{ {{s}}(c^{(1)}) c^{(7)}}\,   y'''_{c^{(8)}}\,   x'''_{b' } \\
&\stackrel{ \eqref{antipodeidentity}}{=}&    -  (-1)^{ \vert   y \vert_n      \vert x'x'' \vert_n    }  (b''\bullet c^{(4)})\,   \\
&& \quad  x'_{ {{s}}(c^{(5)}) c^{(3)}}  \rho(x'',  y' )_{ {{s}}(c^{(2)}) c^{(6)}} \,  y''_{ {{s}}(c^{(7)}) c^{(1)}}\,   y'''_{c^{(8)}}\,   x'''_{b' } \\
&\stackrel{ \eqref{cocomult}}{=}&    -  (-1)^{ \vert   y \vert_n     \vert x'x'' \vert_n   }  (b''\bullet c^{(4)})\,   \\
&& \quad  x'_{ {{s}}(c^{(5)}) c^{(3)}} \rho(x'',  y' )_{ {{s}}(c^{(2)}) c^{(6)}} \,  y''_{ {{s}}(c^{(7)}) c^{(1)} c^{(8)}}\,   x'''_{b' } \\
&\stackrel{ \eqref{antip-prop}}{=}&   -   (-1)^{ \vert   y \vert_n     \vert x'x'' \vert_n  }  (b''\bullet c^{(4)})\,    x'_{ {{s}}(c^{(5)}) c^{(3)}} \rho(x'',  y' )_{ {{s}}(c^{(2)}) c^{(6)}} \,  y''_{  c^{(1)}} \,   x'''_{b' } \\
&\stackrel{ \eqref{coco}}{=}&   -   (-1)^{ \vert   y \vert_n     \vert x'x'' \vert_n   + \vert   x' \vert    \vert x'' \vert }  (b''\bullet c^{(4)})\,    x''_{ {{s}}(c^{(5)}) c^{(3)}} \rho(x',  y' )_{ {{s}}(c^{(2)}) c^{(6)}}  \, y''_{  c^{(1)}} \,   x'''_{b' } \\
&\stackrel{ \eqref{commu}}{=}&    -  (-1)^{  \vert   y \vert_n  \vert x'  \vert_n  }  (b''\bullet c^{(4)})\,    \rho(x',  y' )_{ {{s}}(c^{(2)}) c^{(6)}} \,  y''_{  c^{(1)}} \,  x''_{ {{s}}(c^{(5)}) c^{(3)}} \,  x'''_{b' } \\
&\stackrel{ \eqref{cocomult}}{=}&    -  (-1)^{\vert   y \vert_n  \vert x'  \vert_n   }  (b''\bullet c^{(4)})\,    \rho(x',  y' )_{ {{s}}(c^{(2)}) c^{(6)}} \,  y''_{  c^{(1)}} \,  x''_{ {{s}}(c^{(5)}) c^{(3)} b' } \\
&\stackrel{ \eqref{balanced}}{=}&   -   (-1)^{\vert   y \vert_n  \vert x'  \vert_n   }  (c^{(3)}\bullet b^{(3)})\,    \rho(x',  y' )_{ {{s}}(c^{(2)}) c^{(4)}} \, y''_{  c^{(1)}} \,  x''_{ {{s}}(b^{(2)}) b^{(4)} b^{(1)}}   \\
&\stackrel{ \eqref{antip-prop}}{=}&   -   (-1)^{ \vert   y \vert_n  \vert x'  \vert_n    }  (c^{(3)}\bullet b')\,    \rho(x',  y' )_{ {{s}}(c^{(2)}) c^{(4)}}  \, y''_{  c^{(1)}} \,  x''_{ b''}   \\
&\stackrel{\eqref{commu}}{=} &    -  (-1)^{ \vert   y \vert_n  \vert x'  \vert_n  +\vert   x'' \vert \vert   y'' \vert  }  (b'\bullet c^{(3)} )\,    \rho(x',  y' )_{ {{s}}(c^{(2)}) c^{(4)}} \,  x''_{ b''}  \,   y''_{  c^{(1)}} \,  \\
&\stackrel{\eqref{b_Fox_2C}}{=}&  - (-1)^{\vert x\vert_n \vert y\vert_n}\bracket{x_b}{y_c}\, ,
\end{eqnarray*}
where at the end we use the congruence
$$   \vert y \vert_n \vert x'\vert_n  +\vert x''\vert \vert
y''\vert \equiv   \vert x\vert_n  \vert y\vert_n+  \vert x''\vert
\vert y'\vert_n\,   \mod 2 .$$

The antisymmetry of the   pairing $\bracket{-}{-}:T \times
T \to A_B$  implies that its left and right annihilators are equal.
We show now  that the annihilator contains   $\Ker \pi$. This will imply that the   pairing $\bracket{-}{-} $ descends to a bracket in $A_B  $ satisfying all the
requirements of the theorem.
We need only  to verify that   the defining  relations   of $A_B$
 annihilate  $\bracket{-}{-} $.  For
any homogeneous $x,y\in A$ and $b,c,d\in B$, we have
\begin{eqnarray*}
&&  \bracket{x_{b }}{y_{cd}} \\
&\stackrel{ \eqref{main_formula}}{=}&    (-1)^{ \vert   x''\vert   \vert y' \vert_n}  ((cd)''\bullet b^{(2)})\, \rho(x',y')_{ {{s}}(b^{(3)}) b^{(1)}}\,   x''_{b^{(4)}}\,   y''_{(cd)' } \\
&\stackrel{ \eqref{multimain}}{=}& (-1)^{ \vert   x''\vert   \vert y' \vert_n}  (c''d''\bullet b^{(2)})\, \rho(x',y')_{ {{s}}(b^{(3)}) b^{(1)}}\,   x''_{b^{(4)}}\,   y''_{c'd' }
\\
&\stackrel{ \eqref{bider}, \eqref{cocomult}  }{=}& (-1)^{ \vert   x''\vert   \vert y' \vert_n} \varepsilon (c'') (d''\bullet b^{(2)})\, \rho(x',y')_{ {{s}}(b^{(3)}) b^{(1)}}\,   x''_{b^{(4)}}\,   y''_{c'} \,y'''_{d' }\\ &&  + (-1)^{ \vert   x''\vert   \vert y' \vert_n} \varepsilon (d'') (c''\bullet b^{(2)})\, \rho(x',y')_{ {{s}}(b^{(3)}) b^{(1)}}\,   x''_{b^{(4)}}\,   y''_{c'} \,y'''_{d' }
\\
&\stackrel{  \eqref{addid},   \eqref{cocounit}}{=}&    (-1)^{ \vert   x''\vert   \vert y' \vert_n}  (d''\bullet b^{(2)})\, \rho(x',y')_{ {{s}}(b^{(3)}) b^{(1)}}\,   x''_{b^{(4)}}\, y''_c \, y'''_{ d' }\\ &&  +   (-1)^{ \vert   x''\vert   \vert y' \vert_n}
 (  c'' \bullet b^{(2)})\, \rho(x',y')_{ {{s}}(b^{(3)}) b^{(1)}}\,   x''_{b^{(4)}}\,   y''_{c'   } \,y'''_d
\\
&\stackrel{ \eqref{coco}, \eqref{main_formula}}{=}&   (-1)^{ \vert   x''\vert   \vert y' \vert_n+  \vert   y''\vert   \vert y''' \vert}  (d''\bullet b^{(2)})\, \rho(x',y')_{ {{s}}(b^{(3)}) b^{(1)}}\,   x''_{b^{(4)}}\, y'''_c \, y''_{ d' }\\ &&  +  \bracket{x_{b }}{y'_{c}} y''_d\\
&\stackrel{ \eqref{addirel}}{=}&   (-1)^{ \vert   x''\vert   \vert y' \vert_n }  (d''\bullet b^{(2)})\, \rho(x',y')_{ {{s}}(b^{(3)}) b^{(1)}}\,   x''_{b^{(4)}}\,   y''_{ d' }\, y'''_c  +  \bracket{x_{b }}{y'_{c}} y''_d\\
&\stackrel{ \eqref{main_formula}}{=}&
\bracket{x_b}{y'_d} y''_c+ \bracket{x_b}{y'_c} y''_d \\
&\stackrel{  \eqref{coco}, \eqref{addirel} }{=}&     (-1)^{\vert x \vert_n \vert
y' \vert } \, y'_c \bracket{x_b}{y''_d} +\bracket{x_b}{y'_c} y''_d   \  \stackrel{   \eqref{poisson1new}}{=} \ \bracket{x_b}{y'_{c}\, y''_d}.
\end{eqnarray*}
For any homogeneous  $x,y,z \in A$ and $b,c \in B$, we have
\begin{eqnarray*}
\!\!\!\!&&\!\!\!\!\!\!\!\! \bracket{(xy)_b}{z_c}\\
\!\!\!\! &\stackrel{ \eqref{main_formula}}{=}&\!\!\!\!\!\!\!\!    (-1)^{ \vert   (xy)''\vert   \vert z' \vert_n}  (c''\bullet b^{(2)})\, \rho((xy)', z ')_{ {{s}}(b^{(3)}) b^{(1)}}\,   (xy)''_{b^{(4)}}\,   z''_{c' }  \\
\!\!\!\! &\stackrel{ \eqref{multimain}}{=}&\!\!\!\!\!\!\!\!   (-1)^{ \vert   x''y''\vert   \vert z' \vert_n+\vert   x'' \vert   \vert y' \vert}  (c''\bullet b^{(2)})\, \rho( x'y', z ')_{ {{s}}(b^{(3)}) b^{(1)}}\,   (x''y'')_{b^{(4)}}\,   z''_{c' }      \\
 \!\!\!\! &\stackrel{ \eqref{Fox2}}{=}&\!\!\!\!\!\!\!\!   (-1)^{ \vert   x''y''\vert   \vert z' \vert_n+\vert   x'' \vert   \vert y' \vert}  (c''\bullet b^{(2)})\, (x' \rho(  y', z '))_{ {{s}}(b^{(3)}) b^{(1)}}\,   (x''y'')_{b^{(4)}}\,   z''_{c' }     \\
 \!\!\!\!   && + (-1)^{ \vert   x''y''\vert   \vert z' \vert_n+\vert   x'' \vert   \vert y' \vert}  (c''\bullet b^{(2)})\, \varepsilon(y') \rho( x' , z ')_{ {{s}}(b^{(3)}) b^{(1)}}\,   (x''y'')_{b^{(4)}}\,   z''_{c' }      \\
 \!\!\!\!  &\stackrel{ \eqref{mult}, \eqref{cocounit}}{=}&\!\!\!\!\!\!\!\!   (-1)^{ \vert   x''y''\vert   \vert z' \vert_n+\vert   x'' \vert   \vert y' \vert}  (c''\bullet b^{(3)})\,  x'_{{s}(b^{(5)}) b^{(1)}}  \rho(  y', z ')_{ {{s}}(b^{(4)}) b^{(2)}}\,    x''_{b^{(6)}} y''_{b^{(7)}}\,   z''_{c' }     \\
  && + (-1)^{ \vert   x''y \vert   \vert z' \vert_n }  (c''\bullet b^{(2)})\,   \rho( x' , z ')_{ {{s}}(b^{(3)}) b^{(1)}}\,   (x''y)_{b^{(4)}}\,   z''_{c' }      \\
 \!\!\!\!  &\stackrel{ \eqref{addirel}, \eqref{mult}}{=}&\!\!\!\!\!\!\!\!   (-1)^{ \vert    y''\vert   \vert z' \vert_n }  (c''\bullet b^{(3)})\,  x'_{{s}(b^{(5)})   b^{(1)}} x''_{b^{(6)}} \rho(  y', z ')_{ {{s}}(b^{(4)}) b^{(2)}}\,    y''_{b^{(7)}}\,   z''_{c' }     \\
   && + (-1)^{ \vert   x''y \vert   \vert z' \vert_n }  (c''\bullet b^{(2)})\,   \rho( x' , z ')_{ {{s}}(b^{(3)}) b^{(1)}}\,    x''_{b^{(4)}}\, y_{b^{(5)}} \,  z''_{c' }      \\
 \!\!\!\!   &\stackrel{ \eqref{cocomult}, \eqref{addirel}}{=}&\!\!\!\!\!\!\!\!   (-1)^{ \vert    y''\vert   \vert z' \vert_n }  (c''\bullet b^{(3)})\,  x_{{s}(b^{(5)})   b^{(1)} b^{(6)}} \rho(  y', z ')_{ {{s}}(b^{(4)}) b^{(2)}}\,    y''_{b^{(7)}}\,   z''_{c' }     \\
    && + (-1)^{ \vert   x''  \vert   \vert z' \vert_n +  \vert   y  \vert   \vert z \vert_n    }  (c''\bullet b^{(2)})\,   \rho( x' , z ')_{ {{s}}(b^{(3)}) b^{(1)}}\,    x''_{b^{(4)}} \,  z''_{c' } \, y_{b^{(5)}}\\
 \!\!\!\!    &\stackrel{  \eqref{antip-prop},  \eqref{cocounit},   \eqref{main_formula}}{=}&\!\!\!\!\!\!   (-1)^{ \vert    y''\vert   \vert z' \vert_n }  (c''\bullet b^{(3)})\,  x_{   b^{(1)}  }
      \rho(  y', z ')_{ {{s}}(b^{(4)}) b^{(2)}}\,    y''_{b^{(5)}}\,   z''_{c' }     \\
     && + (-1)^{\vert y \vert \vert z\vert_n} \bracket{x_{b'}} {z_{c}} y_{b''}\\
 \!\!\!\! &\stackrel{   \eqref{main_formula}}{=}&\!\!\!\!\!\!\!\! x_{b'}\bracket{y_{b''}} {z_{c}} + (-1)^{\vert y \vert \vert z\vert_n} \bracket{x_{b'}} {z_{c}} y_{b''}\ \stackrel{ \eqref{poisson2new}}{=} \
\bracket{x_{b'} y_{b''}}{z_{c}}.
\end{eqnarray*}
 For any $y\in A$ and $b,c\in B$, the equality
$\rho(  1_A, y)=0$ implies that  $\bracket{(1_A)_b}{y_c} =0$. The
Leibniz rule \eqref{poisson1new} implies that $\bracket{
1_{T}}{y_c}=0$. Hence,
$$ \bracket{(1_A)_b-\varepsilon_B(b) 1_{T}}{y_c} =  \bracket{(1_A)_b }{y_c}- \varepsilon_B(b)  \bracket{  1_{T}}{y_c}=  0 - 0 =  0 .
$$
The formula  \eqref{bider} implies that $1_B \bullet B=0 $. Hence,
for any $x,y\in A$ and $c \in B$,
$$\bracket{x_{(1_B)}-\varepsilon_A(x)
1_{T}} {y_c} = \bracket{x_{(1_B)}}{y_c}- \varepsilon_A(x) \bracket{
1_{T}}{y_c} =0-0=0.
$$
 Finally,  the Leibniz rule \eqref{poisson1new} easily implies that
$\bracket{T}{ \beta \gamma - (-1)^{\vert \beta \vert \vert
 \gamma \vert } \gamma \beta  }  =0$ for any
  homogeneous $ \beta, \gamma \in T$.
This concludes the proof of the claim that all   defining   relations   of $A_B$
 annihilate   $\bracket{-}{-} $ and concludes the proof of the theorem.
\end{proof}

  \subsection{A special  case}\label{The cocommutative case}
  The bracket constructed  in Theorem~\ref{double_to_simple}  may  not satisfy the Jacobi identity.
  We will  formulate further  conditions  on our data    guaranteeing  the Jacobi identity.
The next  theorem is the simplest result in this direction.

\begin{theor}
 If, under the conditions of  Theorem~\ref{double_to_simple},    $B$ is   cocommutative, then
   the bracket   constructed in that theorem   is Gerstenhaber of degree~$n$.
\end{theor}

\begin{proof}
For any homogeneous $x,y\in A$ and any $b,c\in B$, we have
\begin{eqnarray*}
\bracket{x_b}{y_c} & \by{main_formula} & (-1)^{ \vert   x''\vert   \vert y' \vert_n}  (c''\bullet b^{(2)})\, \rho(x',y')_{ {{s}_B}(b^{(3)}) b^{(1)}}\,   x''_{b^{(4)}}\,   y''_{c' } \\
&=&   (-1)^{ \vert   x''\vert   \vert y' \vert_n}  (c''\bullet b^{(2)})\, \rho(x',y')_{ {{s}_B}(b^{(3)}) b^{(4)}}\,   x''_{b^{(1)}}\,   y''_{c' } \\
&=&   (-1)^{ \vert   x''\vert   \vert y' \vert_n}  (c''\bullet b'')\, \rho(x',y')_{ 1_B }\,   x''_{b'}\,   y''_{c' } \\
&=& (-1)^{ \vert   x''\vert   \vert y' \vert_n}  (c''\bullet b'')\, \varepsilon_A \rho(x',y')\,   x''_{b'}\,   y''_{c' } \\
&=& (-1)^{ \vert   x''\vert   \vert x' \vert}  (c''\bullet b'')\, \varepsilon_A \rho(x',y')\,   x''_{b'}\,   y''_{c' } \\
&=&   (c''\bullet b'')\, \varepsilon_A \rho(x'',y')\,   x'_{b'}\,   y''_{c' }.
\end{eqnarray*}
 Here, in the  second  equality we use the cocommutativity of $B$,   in the penultimate equality
 we use  that  if  ${\varepsilon_A \rho(x',y') \neq 0}$, then  $\vert x' \vert = -\vert y' \vert_n$, and in the  last   equality, we use the graded cocommutativity of $A$.
Therefore, for any homogeneous ${x,y,z\in A}$ and  any   $b,c,d \in B$,
\begin{eqnarray*}
 (-1)^{  \vert y \vert_{{n}} \vert z \vert_{{n}} } \bracket   {z_d} {\bracket{x_b}{y_c}}
&=& (-1)^{  \vert y \vert_{{n}} \vert z \vert_{{n}} } (c''\bullet b'')\, \varepsilon_A \rho(x'',y')\,    \bracket{z_d}{x'_{b'}y''_{c'}} \\
&=& (-1)^{  \vert y x' \vert_{{n}} \vert z \vert_{{n}} } (c''\bullet b'')\, \varepsilon_A \rho(x'',y')\,    x'_{b'} \bracket{z_d}{y''_{c'}} \\
&& + (-1)^{  \vert y \vert_{{n}} \vert z \vert_{{n}} } (c''\bullet b'')\, \varepsilon_A \rho(x'',y')\,    \bracket{z_d}{x'_{b'}} y''_{c'} \\
&=& P(z,x,y;d,b,c) + Q(z,x,y;d,b,c)
\end{eqnarray*}
where
\begin{eqnarray*}
P(z,x,y;d,b,c)  \!\!\!\!\!  &=& \!\!\!\!\!  (-1)^{  \vert y x' \vert_{{n}} \vert z \vert_{{n}} } (c'''\bullet b'')\,  (d''\bullet c'')\, \varepsilon_A \rho(x'',y')\,  \varepsilon_A \rho(z'',y'')\,   x'_{b'} z'_{d'} {y'''_{c'}}, \\
Q(z,x,y;d,b,c) \!\!\!\!\!   &= & \!\!\!\!\!   (-1)^{  \vert y \vert_{{n}} \vert z \vert_{{n}} } (c''\bullet b''')\,  (d'' \bullet b'')\, \varepsilon_A \rho(x''',y')\, \varepsilon_A\rho(z'',x')\,  z'_{d'} x''_{b'} y''_{c'}.
\end{eqnarray*}
We have
\begin{eqnarray*}
&& P(z,x,y;d,b,c)  \\
&=&  (-1)^{  \vert y x' \vert_{{n}} \vert z \vert_{{n}} }  (d''\bullet c''')\, ( b'' \bullet c'')\,  \varepsilon_A \rho(z'',y'')\, \varepsilon_A \rho(x'',y')\,     x'_{b'} z'_{d'} {y'''_{c'}}\\
&=&   (-1)^{  \vert x \vert_n \vert z \vert_{{n}} + \vert y'' y''' \vert_n \vert z \vert_n}  (d''\bullet c''')\, ( b'' \bullet c'')\,  \varepsilon_A \rho(z'',y'')\, \varepsilon_A \rho(x'',y')\,     x'_{b'} z'_{d'} {y'''_{c'}} \\
&=&    (-1)^{  \vert x \vert_n \vert z \vert_{{n}} + \vert y'' \vert_n \vert z \vert_n + \vert y''' \vert \vert z'' \vert_n}  (d''\bullet c''')\, ( b'' \bullet c'')\\
&&  \qquad \varepsilon_A \rho(z'',y'')\, \varepsilon_A \rho(x'',y')\,     x'_{b'} {y'''_{c'}}  z'_{d'} \\
&=&   - (-1)^{  \vert x \vert_n \vert z \vert_{{n}} + \vert y'' \vert_n \vert z' \vert + \vert y''' \vert \vert z'' \vert_n}  (d''\bullet c''')\, ( b'' \bullet c'') \\
&& \qquad \quad \varepsilon_A \rho(y'',z'')\, \varepsilon_A \rho(x'',y')\,     x'_{b'} {y'''_{c'}}  z'_{d'} \\
&=&   - (-1)^{  \vert x \vert_n \vert z \vert_{{n}} + \vert z'' \vert  \vert z' \vert + \vert y''' \vert \vert z'' \vert_n}  (d''\bullet c''')\, ( b'' \bullet c'') \\
&& \qquad \quad  \varepsilon_A \rho(y'',z'')\, \varepsilon_A \rho(x'',y')\,     x'_{b'} {y'''_{c'}}  z'_{d'} \\
&=&   - (-1)^{  \vert x \vert_n \vert z \vert_{{n}} + \vert y''' \vert \vert z' \vert_n}  (d''\bullet c''')\, ( b'' \bullet c'')\,  \varepsilon_A \rho(y'',z')\, \varepsilon_A \rho(x'',y')\,     x'_{b'} {y'''_{c'}}  z''_{d'} \\
&=&   - (-1)^{  \vert x \vert_n \vert z \vert_{{n}} + \vert y''' \vert \vert y'' \vert}  (d''\bullet c''')\, ( b'' \bullet c'')\,  \varepsilon_A \rho(y'',z')\, \varepsilon_A \rho(x'',y')\,     x'_{b'} {y'''_{c'}}  z''_{d'} \\
&=&   - (-1)^{  \vert x \vert_n \vert z \vert_{{n}} }  (d''\bullet c''')\, ( b'' \bullet c'')\,  \varepsilon_A \rho(y''',z')\, \varepsilon_A \rho(x'',y')\,     x'_{b'} {y''_{c'}}  z''_{d'}\\
&=&  -Q(x,y,z;b,c,d).
\end{eqnarray*}
 These nine  equalities are consequences,  respectively, of the following facts:
(1)~the bilinear form $\bullet$ is symmetric and~$B$ is cocommutative;
(2)~if  ${\varepsilon_A \rho(x'',y') \neq 0}$, then  $\vert x' \vert \equiv \vert x x'' \vert \equiv  \vert x \vert_n -\vert y' \vert  \, \mod 2$;
(3)~the   (graded)    commutativity of $A_B$;
(4)~the antisymmetry of $\varepsilon_A \rho$ (Lemma~\ref{transpose++++});
(5)~if  ${\varepsilon_A \rho(y'',z'') \neq 0}$, then  $\vert y'' \vert_n=- \vert z'' \vert$;
(6)~the cocommutativity of $A$;
(7)~if  ${\varepsilon_A \rho(y'',z') \neq 0}$, then  $\vert z' \vert_n=- \vert y'' \vert$;
(8)~the cocommutativity of $A$;
(9)~the definition of $Q(x,y,z;b,c,d)$. The Jacobi identity easily follows.
\end{proof}

\section{Balanced biderivations from trace-like elements}   \label{bfc}

 From now on, we focus on    balanced biderivations associated with so-called trace-like  elements of Hopf algebras.
 Here  we introduce trace-like elements and define the associated balanced biderivations.

\subsection{Trace-like elements}\label{ndce}

Consider an ungraded  Hopf algebra $B$ with   comultiplication $\Delta=\Delta_B$, counit $\varepsilon =\varepsilon_B  $  and antipode~$s  =s_B $.
An element $t$ of $B$ is \emph{cosymmetric}  if the tensor $\Delta(t)\in B \otimes B$   is invariant under the flip   map, that is
 \begin{equation} \label{sym}
 t' \otimes t''=t'' \otimes t'.
 \end{equation}

\begin{lemma}\label{lemma-cosym}
If $t\in B$ is cosymmetric, then  for  any  integer $n\geq 2$, the $(n-1)$-st iterated comultiplication of~$t$  is
invariant under   cyclic permutations:
\begin{equation} \label{cyclic_inv}
 t^{(1)} \otimes t^{(2)} \otimes \cdots \otimes t^{(n-1)}  \otimes t^{(n)} =  t^{(2)}  \otimes t^{(3)} \otimes \cdots \otimes t^{(n)} \otimes t^{(1)} .
\end{equation}
\end{lemma}

\begin{proof}
For $n=2$, this is~\eqref{sym}.
If  \eqref{cyclic_inv} holds for some  $n \geq 2$,   then it holds for~$n+1$ too:
\begin{eqnarray*}
&& t^{(1)}  \otimes t^{(2)}   \otimes \cdots \otimes t^{(n)} \otimes t^{(n+1)}\\
&=& \left(\id_{B} \otimes \Delta  \otimes  \id_{B^{\otimes (n-2)}}
\right)  \big(t^{(1)} \otimes t^{(2)} \otimes \cdots \otimes
t^{(n-1)} \otimes t^{(n)} \big)\\
&=& \left(\id_{B} \otimes \Delta \otimes  \id_{B^{\otimes (n-2)}} \right)  \big(  t^{(2)}  \otimes t^{(3)} \otimes \cdots \otimes t^{(n)} \otimes t^{(1)}  \big) \\
&=&  t^{(2)}  \otimes  t^{(3)} \otimes t^{(4)}  \otimes \cdots
\otimes t^{(n+1)} \otimes t^{(1)}.
\end{eqnarray*}

\up
\end{proof}

Recall  the notion of a derivation  $B\to \kk$ from Section~\ref{Balanced bilinear forms},
  and    let  $\mathfrak{g}=  \mathfrak{g}_B   $  be the module consisting of all derivations $ B\to \kk$.
(When  $B$ is commutative,  $\mathfrak{g}$ is the Lie algebra of the group scheme associated to $B$.)
Restricting the  derivations   to $I=\Ker \varepsilon$, we
obtain a $\kk$-linear isomorphism $\mathfrak{g}  \simeq  (I/I^2)^*=   \Hom (I/I^2, \kk)$.
Let $p:B \to I/I^2$ be the surjection defined by $p(b)= b- \varepsilon(b) \mod I^2$ for  $b\in B$.
 An element $t$ of $B$  is \emph{infinitesimally-nonsingular} if  the linear map
\begin{equation} \label{psi_t}
 \mathfrak{g} \longrightarrow I/I^2, \ \mu \longmapsto \mu(t')\, p(t'')
\end{equation}
is an isomorphism.
Given such a $t$,   for  any $b\in B$, we let    $\overline b  =\overline b _t \in \mathfrak{g}$ be the pre-image of
$ p(b) \in I/I^2$   under the  isomorphism~\eqref{psi_t}.

An element  of $B$ is \emph{trace-like} if it is cosymmetric and infinitesimally-nonsingular.

\begin{lemma}\label{mainlemma}
 If~$B$ is commutative and $t\in B$ is   trace-like,   then the bilinear form  $\bullet_t: B \times B\to \kk$
defined by   $b \bullet_t  c = \overline b (c)  $   is a balanced biderivation in~$B$.
Moreover, for any    $b,c \in B$, we have
\begin{equation} \label{bullets}
b \bullet_t c   = (b \bullet_t t')\, (c\bullet_t t'')  .
\end{equation}
\end{lemma}

\begin{proof}
 It is clear that  both the  left and the right annihilators   of    $\bullet=\bullet_t$ contain $1_B+I^2$;   hence $\bullet$ is a biderivation.
 To verify that it  is balanced, we check     \eqref{balanced} for any $b,c\in B$.
It follows from the definitions  that
$$
 b - \varepsilon(b) =\overline b (t') (t'' -\varepsilon(t''))      \mod I^2
$$
  and, since $t'\varepsilon (t'') =t$, we obtain
$$
b=\varepsilon(b)   - \overline{b}(t)   + \overline b (t') t'' +\sum_i d_ie_i
$$
where the index $i$ runs over a finite set and $d_i, e_i\in I$ for all $i$.
Hence
\begin{eqnarray*}
&&b'\otimes b'' \otimes b''' \\
&=&   \big( \varepsilon(b) - \overline{b}(t) \big)\,   1_B \otimes 1_B \otimes 1_B+ \overline b (t')\, t''\otimes t'''\otimes t''''   +\sum_i  d'_i e'_i \otimes d''_i e''_i \otimes d'''_i e'''_i.
\end{eqnarray*}
Using this expansion and the equality $\overline c(1_B)=0$, we obtain  that
$$
(c \bullet  b'')\,  s(b ''') b '  =  \overline c  (b'') s(b ''') b '
 =   \overline b (t')\,  \overline c (t''') \,    s(t'''') t''+ \sum_i \overline c (d''_i e''_i)\, s(d'''_i e'''_i)\, d'_i e'_i.
$$
The $i$-th term   is equal to zero for all $i$. Indeed, for any   $d,e\in I$,   we have
\begin{eqnarray*}
\overline c (d'' e'' ) s(d'''  e''' ) d'  e' &=&  \varepsilon (d'') \overline c (  e'' ) s(e''') s(  d''' ) d'  e'
+\varepsilon (e'') \overline c (  d'' ) s(e''') s(  d''' ) d'  e' \\
&=& \overline c (  e'' ) s(e''') s(  d'' ) d'  e'
+  \overline c (  d'' ) s(e'') s(  d''' ) d'  e' \\
 &=& \varepsilon (d) \overline c (  e'' ) s(e''')    e'+ \varepsilon (e) \overline c (  d'' ) s(  d''' ) d'   \ = \ 0
\end{eqnarray*}
where we use the commutativity of $B$ and the equalities $\varepsilon (d)=\varepsilon (e)=0$. Thus,
\begin{equation}\label{fico1}  (c \bullet  b'')\,  s(b ''') b ' =     \overline b (t')\, \overline c (t''') \, s(t'''') t''.\end{equation}
Similarly,  starting from the expansion   $c=\varepsilon(c)   - \overline{c}(t)   + \overline c (t') t''\!  \mod I^2 $, we obtain
\begin{equation}\label{fico2}  (b  \bullet  c'')\,  s(c') c'''=   \overline c (t')\, \overline b (t''') \, s(t'') t''''.\end{equation}
 It follows from \eqref{cyclic_inv} that
the right-hand sides of the equalities \eqref{fico1} and \eqref{fico2} are equal.
 We conclude that   $ (c \bullet  b'')\,  s(b ''') b ' = (b  \bullet  c'')\,  s(c') c'''$.

 Formula \eqref{bullets} is proved as follows: $$
b \bullet_t c= \overline{b}  (c)= \overline{b}(t')\, \overline{c}(t'')  = (b \bullet_t t')\, (c\bullet_t t'')  .
$$ Here the   second equality holds because  $c  =\overline c (t') t''  =\overline c (t'') t'     \mod ( \kk   1_B+I^2)$.
\end{proof}

\subsection{Brackets re-examined}

We reformulate the bracket    constructed in Theorem~\ref{double_to_simple} in the case where the balanced biderivation   arises from a trace-like element.

 \begin{theor} \label{equivariant_bracket}
  Assume, under the conditions of Theorem~\ref{double_to_simple}, that $\bullet=\bullet_t$ for a trace-like element $t\in B$.
Then the resulting bracket $\bracket{-}{-}:A_B \times A_B \to A_B $      is computed   by
\begin{equation} \label{formulafromt3}
 \bracket{x_b}{y_c}
 =    (-1)^{ \vert   x''\vert   \vert y' \vert_n}  (b' \bullet t^{(2)})\, (c'' \bullet t^{(4)}) \,  \rho(x',y')_{  s_B(t^{(1)}) t^{(3)}}\,   x''_{b''}\,   y''_{c'}
\end{equation}
for any $x,y \in A$ and $b,c \in B$.
  Furthermore, the bracket  $\bracket{-}{-}$ is $B$-equivariant  with respect to the $B$-coaction on $A_B$  defined in Lemma~\ref{coaction_M}.
\end{theor}

\begin{proof}
Set $s=s_B$.  We first prove formula \eqref{formulafromt3}:
\begin{eqnarray*}
 \bracket{x_b}{y_c}&=&   (-1)^{ \vert   x''\vert   \vert y' \vert_n}  (b^{(2)} \bullet c'')\,  \rho(x',y')_{ {{s}}(b^{(3)}) b^{(1)}}\,   x''_{b^{(4)}}\,   y''_{ c'} \\
 &\by{bullets}&   (-1)^{ \vert   x''\vert   \vert y' \vert_n}  (b^{(2)} \bullet t')\, (c'' \bullet t'') \,  \rho(x',y')_{ {{s}}(b^{(3)}) b^{(1)}}\,   x''_{b^{(4)}}\,   y''_{ c'} \\
& \by{balanced} &  (-1)^{ \vert   x''\vert   \vert y' \vert_n}  (b' \bullet t^{(2)})\, (c'' \bullet t^{(4)}) \,  \rho(x',y')_{ {{s}}(t^{(1)}) t^{(3)}}\,   x''_{b''}\,   y''_{c'}.
\end{eqnarray*}
  In order to prove the $B$-equivariance of $\bracket{-}{-}$,  we must show  that for any   $m_1, m_2 \in A_B$,
\begin{equation} \label{B-equivariance++}
\bracket{m_1}{  m_2^\ell  } \otimes  m_2^r =   \bracket{m_1^\ell }{ m_2}^\ell \otimes   \bracket{m_1^\ell}{ m_2}^r  s(m_1^r).
\end{equation}
   Using the $n$-graded Leibniz rules   for $\bracket{-}{-}$   and the fact that the comodule map $\Delta: {A_B\to A_B\otimes B}$  is a graded  algebra homomorphism,
one easily checks that if  \eqref{B-equivariance++} holds for   pairs $(m_1,m_2)$ and  $(m_3,m_4)$,
then \eqref{B-equivariance++}  holds for the pair $(m_1m_3,m_2m_4)$. Also, both sides of \eqref{B-equivariance++}  are equal to $0$ if $m_1=1$ or $m_2=1$.
Therefore it  suffices   to verify \eqref{B-equivariance++} for $m_1=x_b$ and $m_2=y_c$ with $x,y\in A$ and $b,c\in B$. In this case,  \eqref{B-equivariance++} may be rewritten as
\begin{equation} \label{equi-rewritten}
 \bracket{x_b}{ y_{c''}} \otimes  s(c')  c'''   =   \bracket{x_{b''}}{y_c}^\ell \otimes   \bracket{x_{b''}}{y_c}^r  s(b''') b'.
\end{equation}
Applying   $\Delta: A_B\to A_B\otimes B$ to both sides of~\eqref{formulafromt3}, we obtain
\begin{eqnarray*}
&& \Delta(\bracket{x_{b}}{y_c}) \\
&  =  &      (-1)^{ \vert   x''\vert   \vert y' \vert_n}  (b'\bullet t^{(2)})\,(c''\bullet t^{(4)})\,  \Delta( \rho(x',y')_{  s (t^{(1)}) t^{(3)} } ) \, \Delta(  x''_{b''}  )\, \Delta(  y''_{c' } )\\
&\by{cocomult}&   (-1)^{ \vert   x''\vert   \vert y' \vert_n}   (b'\bullet t^{(2)})\,(c''\bullet t^{(4)})\\
& & \cdot     \Delta\big( \rho(x',y')'_{ {{s}}(t^{(1)})}\big) \, \Delta\big( \rho(x',y')''_{ t^{(3)}  }  \big) \, \Delta(  x''_{b''}  )\, \Delta(  y''_{c' } )\\
&\by{Delta_M}&   (-1)^{ \vert   x''\vert   \vert y' \vert_n}  (b^{(1)}\bullet t^{(4)})\,(c^{(4)}\bullet t^{(8)})\,  \big( \rho(x',y')'_{ {{s}}(t^{(2)})} \otimes   t^{(3)}  s(t^{(1)} )\big)   \\
&&  \cdot     \big( \rho(x',y')''_{ t^{(6)}  } \otimes s(t^{(5)} ) t^{(7)} \big)\,  \big(  x''_{b^{(3)}} \otimes s(b^{(2)}) b^{(4)} \big)\, \big(  y''_{c^{(2)} } \otimes s(c^{(1)}) c^{(3)} \big)\\
&\by{mugradedalg}& (-1)^{ \vert   x''\vert   \vert y' \vert_n}  (b^{(1)}\bullet t^{(4)})\,(c^{(4)}\bullet t^{(8)})\,   \rho(x',y')'_{ {{s}}(t^{(2)})}  \rho(x',y')''_{ t^{(6)}  }  x''_{b^{(3)}}     y''_{c^{(2)} } \\
&&    \otimes  \, t^{(3)}  s(t^{(1)} ) s(t^{(5)} ) t^{(7)} s(b^{(2)}) b^{(4)} s(c^{(1)}) c^{(3)}    \\
&\by{cocomult}& (-1)^{ \vert   x''\vert   \vert y' \vert_n}  (b^{(1)}\bullet t^{(4)})\,(c^{(4)}\bullet t^{(8)})\,   \rho(x',y')_{ {{s}}(t^{(2)})  t^{(6)}  } \,  x''_{b^{(3)}}     y''_{c^{(2)} } \\
&&    \otimes  \, t^{(3)}  s(t^{(1)} ) s(t^{(5)} ) t^{(7)} s(b^{(2)}) b^{(4)} s(c^{(1)}) c^{(3)}  \\
&\by{balanced}& (-1)^{ \vert   x''\vert   \vert y' \vert_n}  (b^{(2)}\bullet t^{(3)})\,(c^{(4)}\bullet t^{(6)})\,   \rho(x',y')_{ {{s}}(t^{(2)})  t^{(4)}  }  x''_{b^{(5)}}     y''_{c^{(2)} } \\
&&    \otimes  \,    s(t^{(1)} )  t^{(5)} s(b^{(1)}) b^{(3)} s(b^{(4)}) b^{(6)} s(c^{(1)}) c^{(3)} \\
&\by{antip-prop}& (-1)^{ \vert   x''\vert   \vert y' \vert_n}  (b^{(2)}\bullet t^{(3)})\,(c^{(4)}\bullet t^{(6)})\,   \rho(x',y')_{ {{s}}(t^{(2)})  t^{(4)}  }  x''_{b^{(3)}}     y''_{c^{(2)} } \\
&&    \otimes  \,    s(t^{(1)} )  t^{(5)} s(b^{(1)})   b^{(4)}  s(c^{(1)}) c^{(3)}  \\
&\by{cyclic_inv}& (-1)^{ \vert   x''\vert   \vert y' \vert_n}  (b^{(2)}\bullet t^{(2)})\,(c^{(4)}\bullet t^{(5)})\,   \rho(x',y')_{ {{s}}(t^{(1)})  t^{(3)}  }  x''_{b^{(3)}}     y''_{c^{(2)} } \\
&&    \otimes  \,    s(t^{(6)} )  t^{(4)} s(b^{(1)})   b^{(4)}
s(c^{(1)}) c^{(3)}  \\
&\by{balanced}& (-1)^{ \vert   x''\vert   \vert y' \vert_n}  (b^{(2)}\bullet t^{(2)})\,(c^{(5)}\bullet t^{(4)})\,   \rho(x',y')_{ {{s}}(t^{(1)})  t^{(3)}  }  x''_{b^{(3)}}     y''_{c^{(2)} } \\
&&    \otimes  \,      s(b^{(1)})   b^{(4)}
s(c^{(1)}) c^{(3)}    s(c^{(4)}) c^{(6)}\\
& \by{antip-prop} & (-1)^{ \vert   x''\vert   \vert y' \vert_n}  (b^{(2)}\bullet t^{(2)})\,(c^{(3)}\bullet t^{(4)})\,   \rho(x',y')_{ {{s}}(t^{(1)}) t^{(3)}  }  x''_{b^{(3)}}     y''_{c^{(2)} } \\
&&    \otimes  \,      s(b^{(1)})   b^{(4)}
s(c^{(1)})   c^{(4)} \\&\by{formulafromt3}& \bracket{x_{b''}}{y_{c''}}   \otimes  \,      s(b')   b'''
s(c')   c'''.
\end{eqnarray*}
It follows that
\begin{eqnarray*}   \bracket{x_{b''}}{y_c}^\ell \otimes   \bracket{x_{b''}}{y_c}^r  s(b''') b'&=& \bracket{x_{b^{(3)}}}{y_{c''}}   \otimes  \,      s(b^{(2)})   b^{(4)}
s(c')   c'''  s(b^{(5)})   b^{(1)}\\
&=& \bracket{x_{b }}{y_{c''}}   \otimes  \,
s(c')   c'''  . \end{eqnarray*}
This proves \eqref{equi-rewritten} and concludes the proof of the theorem.
\end{proof}

\subsection{Remarks} \label{about_tle}

  Let $B$ be   a commutative ungraded Hopf algebra.

    1.   It can be  verified that an element  of $B$ is cosymmetric
if and only if it is $B$-invariant under the adjoint coaction \eqref{adjoint_coaction} of~$B$.
 Note that an element of $B$ is invariant  under  the adjoint coaction  if and only if this element is invariant  under the conjugation  action of the   group scheme  determined by $B$,
see Appendix~\ref{induced_actions}.

2. We call  a symmetric bilinear form   $ X\times X \to \kk$ in a module $X$   \emph{nonsingular}
if   the adjoint linear map $X \to X^*$  is an isomorphism.  For a trace-like   $t\in B$, the symmetric bilinear form in $I/I^2$ induced by   $\bullet_t$ is nonsingular.
 As a consequence,  not all balanced biderivations in $B$  arise from trace-like elements. For instance, the zero bilinear form $B \times B \to \kk$ is a balanced biderivation not arising from a trace-like element of $B$.

 3. In   general, a  trace-like element   $t\in B$  cannot  be recovered from~$ \bullet_t$.
 For instance, $s(t)\in B$ is also a trace-like element  and   $ \bullet_{s(t)}=\bullet_{t} $. However,
  in many  examples,    $s(t) \neq t$.

\section{Examples of trace-like elements} \label{trace-like_examples}

We give   examples of   trace-like elements  in commutative   ungraded   Hopf algebras arising from   classical group schemes,
and we compute the corresponding brackets in representation algebras.
 Throughout  this section, we fix an integer $N\geq 1$ and   set  $\overline N = \{1, \ldots, N\}$.

\subsection{The general linear group}  \label{GL_ex}

Consider  the group scheme $\GL_N$
 assigning to every commutative  ungraded algebra~$C$ the group   $\GL_N(C)$   of invertible $N\times N$ matrices over~$C$.
 The  coordinate algebra,  $B$, of  $\GL_N$   is the  commutative  ungraded  Hopf algebra generated by the symbols $u$ and   $\{t_{ij}\}_{ i,j\in \overline N}$
 subject to the single    relation ${u \det(T) =1}$, where~$T$ is the $N\times N$ matrix with entries $ t_{ij}$.
The   comultiplication~$\Delta$, the counit $\varepsilon$, and the antipode $s$  in~$B$ are  computed   by
$$
\Delta(t_{ij}) =  \sum_{k\in \overline N}  t_{ik} \otimes t_{kj}, \quad \Delta(u)= u \otimes u, \quad \varepsilon(t_{ij}) =\delta_{ij}, \quad  \varepsilon(u)=1
$$
and
$$
s(u) = \det(T) , \quad
s(t_{ij})= (-1)^{i+j} u \cdot  \big((j,i)\hbox{-th minor of }  T\big) .
$$
It is clear from the definitions that   the element   \begin{equation} \label{element t}  t=\sum_{i \in \overline N}   t_{ii}  \in B \end{equation}    is  cosymmetric.
  We claim that $t$ is infinitesimally-nonsingular. To see this, for  any $i,j \in \overline{N}$,
  denote by  $\tau_{ij}$   the  class of $t_{ij}- \delta_{ij} \in I=\Ker(\varepsilon)$ in $I/I^2$.
  Computing   $I/I^2$   from the presentation of $B$ above, we obtain that this module  is free   with basis    $\{\tau_{ij}\}_{ i, j}$.
  Let $\{\tau_{ij}^*\}_{ i, j}$ be   the dual basis of $\mathfrak{g} \simeq (I/I^2)^*$.
  It is easy to check that the linear map~\eqref{psi_t} sends $\tau_{ij}^*$ to $\tau_{ji}$ for any $i,j$.   This map is an isomorphism, and so~$t$ is  infinitesimally-nonsingular and  trace-like.
 The  balanced biderivation   $  \bullet_t   : B \times B \to \kk$
is computed   by  $t_{ij}  \bullet_t   t_{kl} =   \delta_{il} \delta_{jk} $  for all $i,j,k,l \in \overline N$.

Consider a cocommutative graded Hopf algebra~$A$ carrying    an antisym\-metric Fox pairing~$\rho$  of degree $n\in \ZZ$.
Theorem \ref{double_to_simple}  produces a bracket $\bracket{-}{-}$ in the representation algebra~$A_B$. We compute this bracket
on the elements $x_{ij} = x_{(t_{ij})}$  and $y_{kl} = y _{(t_{kl})}$  for any $x,y\in A$ and $i,j,k,l \in \overline{N}$. In the following computation (and in similar computations below) we sum up over all repeating indices:
\begin{eqnarray}
\notag && \bracket{x_{ij}}{y_{kl}}  \\
\notag   & \by{main_formula} &        (-1)^{ \vert   x''\vert   \vert y' \vert_n}  (t_{vl}  \bullet_t   t_{pq})\, \rho(x',y')_{ {{s}}(t_{qr}) t_{ip} }\,   x''_{rj}\,   y''_{kv} \\
\notag  &=&     (-1)^{ \vert   x''\vert   \vert y' \vert_n}  \rho(x',y')_{ {{s}}(t_{vr}) t_{il} }\,   x''_{rj}\,   y''_{kv}   \\
\notag  &=&     (-1)^{ \vert   x''\vert   \vert y' \vert_n}  s_A(\rho(x',y')')_{vr}\,  \rho(x',y')''_{ il }\,  x''_{rj}\,   y''_{kv}\\
\notag  &=&     (-1)^{ \vert   x''\vert   \vert y' \rho(x',y')''\vert_n + \vert y'' \vert \vert x y' \vert_n}   y''_{kv} s_A(\rho(x',y')')_{vr}\, x''_{rj}\,  \rho(x',y')''_{ il }\\
\notag  &=&     (-1)^{ \vert   x''\vert   \vert y' \rho(x',y')''\vert_n + \vert y'' \vert \vert x y' \vert_n}   \big(y'' s_A(\rho(x',y')') x''\big)_{kj}\,  \rho(x',y')''_{ il }\\
\label{almost_VdB-} &=&  (-1)^{ \vert   x'\vert   \vert  \rho(x'',y'')' \vert + \vert y' \vert \vert x  \vert_n}   \big(y' s_A(\rho(x'',y'')') x'\big)_{kj}\,  \rho(x'',y'')''_{ il },
\end{eqnarray}
 where   the last equality follows from the cocommutativity of~$A$.
 The formula \eqref{almost_VdB-} fully determines the bracket $\bracket{-}{-}$  in $A_B$ because
the  algebra $A_B$ is generated by the set $\{x_{ij}\, \vert \, x\in A, i,j \in  \overline{N}\}$. The latter follows from the identity
$$
x_u = x_{s(\det(T))} = \big(s(x)\big)_{\det(T)} \quad {\text{for any}} \quad x\in A.
$$

\subsection{The special linear group}

Assume   that $N$ is invertible in $\kk$,
and  consider  the group scheme $\operatorname{SL}_N$
assigning to every commutative  ungraded  algebra $C$ the group $\operatorname{SL}_N(C)$ of $N \times N$ matrices over $C$  with determinant$1$.
The    coordinate algebra, $B$, of  $\operatorname{SL}_N$    is the commutative  ungraded Hopf algebra generated by the symbols $\{t_{ij}\}_{i,j \in \overline{N}}$ subject to the single relation $\det(T)=1$
where $T$ is the $N \times N$ matrix with entries $t_{ij}$.
The   comultiplication~$\Delta$, the counit $\varepsilon$, and the antipode $s$  in~$B$ are  computed   by
$$
\Delta(t_{ij}) =  \sum_{k\in \overline N}  t_{ik} \otimes t_{kj}, \quad \varepsilon(t_{ij}) =\delta_{ij}, \quad
s(t_{ij})= (-1)^{i+j} \cdot  \big((j,i)\hbox{-th minor of }  T\big) .
$$
The same formula \eqref{element t}  as above defines a cosymmetric   $t\in B$.
To show that $t$ is infinitesi\-mal\-ly-nonsingular, let  $\tau_{ij} $   be   the  class of $ t_{ij}- \delta_{ij}    \in I=\Ker(\varepsilon)$   in $I/I^2$
for   $i,j \in \overline{N}$. Computing   $I/I^2$   from the presentation of $B$ above, we obtain that this module is generated by the  $\{\tau_{ij}\}_{i,j }$ subject to the single relation $\tau_{11} + \cdots + \tau_{NN}=0$.
Hence $I/I^2$ is free  with basis   $\{\tau_{ij}\}_{ i \neq j} \cup \{\tau_{ii}\}_{ i \neq N}$.
Let  $\{\tau_{ij}^*\}_{ i \neq j} \cup \{\tau_{ii}^*\}_{ i\neq N }$ be   the dual basis of  $\mathfrak{g} \simeq (I/I^2)^*$.
The   linear map \eqref{psi_t} defined by~$t$ carries $\tau_{ij}^*$ to $\tau_{ji}$ for any $i \neq j$
and  carries  $\tau_{ii}^*$ to ${\tau_{ii}+ \sum_{j\neq N } \tau_{jj}}$ for any $i\neq N$; since $1/N \in \kk$, this map     is an isomorphism. So,  $t$ is trace-like.
The  balanced biderivation  $  \bullet_t   $ in $B$
is   computed   by  $t_{ij} \bullet_t  t_{kl} = {\delta_{il} \delta_{jk}  - \delta_{ij} \delta_{kl}/N}$  for all $i,j,k,l \in \overline N$.

 Consider a cocommutative graded Hopf algebra~$A$ carrying    an antisym\-metric Fox pairing~$\rho$  of degree $n\in \ZZ$.
The bracket $\bracket{-}{-}$ in $A_B$ given by Theorem \ref{double_to_simple} is determined by its values on the elements $x_{ij}=x_{(t_{ij})}$ and $y_{kl}= y_{(t_{kl})}$,
where $x,y \in A$ and $i,j,k,l\in \overline{N}$.  We have
\begin{eqnarray*}
\bracket{x_{ij}}{y_{kl}}    & \by{main_formula} &        (-1)^{ \vert   x''\vert   \vert y' \vert_n}  (t_{vl} \bullet_t   t_{pq})\, \rho(x',y')_{ {{s}}(t_{qr}) t_{ip} }\,   x''_{rj}\,   y''_{kv} \\
&=&     (-1)^{ \vert   x''\vert   \vert y' \vert_n}  \rho(x',y')_{ {{s}}(t_{vr}) t_{il} }\,   x''_{rj}\,   y''_{kv} \\
&& -\frac{  (-1)^{ \vert   x''\vert   \vert y' \vert_n} }{N} \rho(x',y')_{ {{s}}(t_{pr}) t_{ip} }\,   x''_{rj}\,   y''_{kl} \\
&=&     (-1)^{ \vert   x''\vert   \vert y' \vert_n}  s_A(\rho(x',y')')_{vr}\,  \rho(x',y')''_{ il }\,  x''_{rj}\,   y''_{kv}\\
&& -\frac{  (-1)^{ \vert   x''\vert   \vert y' \vert_n} }{N} \rho(x',y')_{ \varepsilon(t_{ir}) }\,   x''_{rj}\,   y''_{kl} \\
&=&     (-1)^{ \vert   x''\vert   \vert y' \rho(x',y')''\vert_n + \vert y'' \vert \vert x y' \vert_n}   y''_{kv} s_A(\rho(x',y')')_{vr}\, x''_{rj}\,  \rho(x',y')''_{ il }\\
&&    -\frac{  (-1)^{ \vert   x''\vert   \vert y' \vert_n} }{N} \rho(x',y')_{\delta_{ir} } \, x''_{rj}\,   y''_{kl} \\
&=&     (-1)^{ \vert   x''\vert   \vert y' \rho(x',y')''\vert_n + \vert y'' \vert \vert x y' \vert_n}   \big(y'' s_A(\rho(x',y')') x''\big)_{kj}\,  \rho(x',y')''_{ il }\,   \\
&& -\frac{  (-1)^{ \vert   x''\vert   \vert y' \vert_n} }{N}  \delta_{ir}\,  \varepsilon_A \rho(x',y')\,   x''_{rj}\,   y''_{kl}\\
  &=&   (-1)^{ \vert   x'\vert   \vert  \rho(x'',y'')' \vert + \vert y' \vert \vert x  \vert_n}   \big(y' s_A(\rho(x'',y'')') x'\big)_{kj}\,  \rho(x'',y'')''_{ il } \\
&& -\frac{  (-1)^{ \vert   x''\vert   \vert y' \vert_n} }{N} \varepsilon_A \rho(x',y')\,   x''_{ij}\,   y''_{kl}.
\end{eqnarray*}

\subsection{The orthogonal group}
 A matrix over an ungraded algebra is {\it orthogonal} if it is a 2-sided inverse of the transpose matrix. Assume   that $2$ is invertible in $\kk$,
and  consider  the group scheme  $\operatorname{O}_N$
assigning to every commutative ungraded  algebra $C$ the group $\operatorname{O}_N(C)$ of $N \times N$ orthogonal  matrices over $C$.
The    coordinate algebra, $B$, of  $\operatorname{O}_N$   is the commutative   ungraded  Hopf algebra generated by the symbols    $\{t_{ij}\}_{i,j \in \overline{N}}$
subject to the  relations $t_{ik} t_{jk} =\delta_{ij}$ for all $i,j \in \overline N$  (here and below we sum over repeated indices.)
The   comultiplication~$\Delta$, the counit $\varepsilon$, and the antipode $s$  in~$B$ are  computed   by
$$
\Delta(t_{ij}) =  \sum_{k\in \overline N}  t_{ik} \otimes t_{kj}, \quad \varepsilon(t_{ij}) =\delta_{ij}, \quad  s(t_{ij})= t_{ji}.
$$
 The   formula \eqref{element t}    defines a trace-like   $t\in B$. To show that $t$ is infinitesimally-nonsingular,  let $\tau_{ij}\in I/I^2$ be the  class of $t_{ij}- \delta_{ij} \in I=\Ker(\varepsilon)$ for any $i,j \in \overline{N}$.
 Computing   $I/I^2$   from the presentation of $B$ above, we obtain that this module is generated by the  $\{\tau_{ij}\}_{i,j }$ subject to the   relations $\tau_{ij} + \tau_{ji}=0$ for all $i,j\in \overline{N}$.
The set $\{\tau_{ij}\}_{ i < j}$ is a basis of $I/I^2$,     and we let $\{\tau_{ij}^*\}_{ i<j}$ be  the dual basis of  $\mathfrak{g} \simeq (I/I^2)^*$.
The   linear map \eqref{psi_t} defined by $t$ carries $\tau_{ij}^*$ to $-2\tau_{ij}$ for any $i < j$, so that \eqref{psi_t}  is an isomorphism.
The  balanced biderivation  $  \bullet_t $ in $B$
is   computed   by  $t_{ij} \bullet_t t_{kl} = {(\delta_{il} \delta_{jk}  - \delta_{ik} \delta_{jl})/2}$  for all $i,j,k,l \in \overline N$.

Let $\rho $  be an antisymmetric Fox pairing of degree  $n\in \ZZ$  in a cocommutative graded Hopf algebra $A$.
The bracket $\bracket{-}{-}$ in $A_B$ given by Theorem \ref{double_to_simple} is determined by its values on the elements $x_{ij}=x_{(t_{ij})}$ and $y_{kl}= y_{(t_{kl})}$
for  $x,y \in A$ and $i,j,k,l\in \overline{N}$. We compute
\begin{eqnarray*}
&& 2\bracket{x_{ij}}{y_{kl}}    \\
& \by{main_formula} &        (-1)^{ \vert   x''\vert   \vert y' \vert_n} 2  (t_{vl}  \bullet_t   t_{pq})\, \rho(x',y')_{ {{s}}(t_{qr}) t_{ip} }\,   x''_{rj}\,   y''_{kv} \\
&=&  (-1)^{ \vert   x''\vert   \vert y' \vert_n}  \rho(x',y')_{ {{s}}(t_{vr}) t_{il} }\,   x''_{rj}\,   y''_{kv} -    (-1)^{ \vert   x''\vert   \vert y' \vert_n} \rho(x',y')_{ {{s}}(t_{lr}) t_{iv} }\,   x''_{rj}\,   y''_{kv}.
\end{eqnarray*}
 The first term in the last expression is computed   as in the $\GL_N$ case and is equal~to
$$(-1)^{ \vert   x'\vert   \vert  \rho(x'',y'')' \vert + \vert y' \vert \vert x  \vert_n}   \big(y' s_A(\rho(x'',y'')') x'\big)_{kj}\,  \rho(x'',y'')''_{ il }.$$
The second term in the expansion of $2\bracket{x_{ij}}{y_{kl}}$   is computed as follows:
\begin{eqnarray*}
&&  (-1)^{ \vert   x''\vert   \vert y' \vert_n} \rho(x',y')_{ {{s}}(t_{lr}) t_{iv} }\,   x''_{rj}\,   y''_{kv}\\
&=& (-1)^{ \vert   x''\vert   \vert y' \vert_n} s_A(\rho(x',y')')_{ lr }\,  \rho(x',y')''_{  iv }\,  x''_{rj}\,   s_A(y'')_{vk} \\
&=&  (-1)^{ \vert   x''\vert   \vert y' \rho(x',y')''\vert_n} s_A(\rho(x',y')')_{ lr }\, x''_{rj}\,  \rho(x',y')''_{  iv }\,    s_A(y'')_{vk} \\
&=&  (-1)^{ \vert   x''\vert   \vert y' \rho(x',y')''\vert_n} \big(s_A(\rho(x',y')')\, x''\big)_{ lj}\,  \big(\rho(x',y')''s_A(y'')\big)_{  ik} \\
&=&  (-1)^{ \vert   x''\vert   \vert x' \rho(x',y')'\vert} \big(s_A(\rho(x',y')')\, x''\big)_{ lj}\,  \big(\rho(x',y')''s_A(y'')\big)_{  ik} \\
&=&  (-1)^{ \vert   x'\vert   \vert \rho(x'',y')'\vert} \big(s_A(\rho(x'',y')')\, x'\big)_{ lj}\,  \big(\rho(x'',y')''s_A(y'')\big)_{  ik}.
\end{eqnarray*}
We conclude that
\begin{eqnarray*}
\bracket{x_{ij}}{y_{kl}}    & =& \frac{(-1)^{ \vert   x'\vert   \vert  \rho(x'',y'')' \vert + \vert y' \vert \vert x  \vert_n}}{2}     \big(y' s_A(\rho(x'',y'')') x'\big)_{kj}\,  \rho(x'',y'')''_{ il } \\
&& -\frac{(-1)^{ \vert   x'\vert   \vert \rho(x'',y')'\vert} }{2} \big(s_A(\rho(x'',y')')\, x'\big)_{ lj}\,  \big(\rho(x'',y')''s_A(y'')\big)_{  ik}.
\end{eqnarray*}

\section{The Jacobi identity in representation algebras}\label{The Jacobi identity}

We formulate  additional   conditions   in Theorem~\ref{double_to_simple}   ensuring  the   Jacobi identity.

 \subsection{Tritensor maps} \label{triple}

  Given a graded algebra~$A$ and  a permutation $(i_1, i_2 ,i_3)$ of the sequence $(1, 2,3)$,
  we  let $\Perm_{i_1 i_2 i_3}:A^{\otimes 3}\to A^{\otimes 3}$ be    the linear map   carrying  any     $x_1\otimes x_2 \otimes x_3 $
with homogeneous $x_1, x_2, x_3\in A$ to $(-1)^{t} x_{i_1}
\otimes x_{i_2} \otimes   x_{i_3}$
 where    $t\in \ZZ$ is the sum of the products  $\vert x_{i_p} \vert \vert x_{i_q} \vert$
over all pairs of indices  $ p < q $  such that $i_p>i_q$. We call  $\Perm_{i_1 i_2 i_3} $      the  \emph{graded permutation}.  For $n\in \ZZ$, we
similarly define  the   {\it ${{n}}$-graded permutation}
$\Perm_{i_1 i_2 i_3,n} :A^{\otimes 3}\to A^{\otimes 3}$  using
  $\vert\!-\!\vert_{{n}}=\vert\!-\!\vert +n $ instead of $\vert\!-\!\vert$.

  Assume  now    that $A$ is a  graded Hopf algebra with antipode $s=s_A$.
   Any   antisymmetric   Fox pairing $\rho:A \times A \to A$   of degree~$n$ determines a linear map  $\digamma= \digamma_{\!\rho}  : A^{\otimes 3} \to A^{\otimes 3}$ by
\begin{eqnarray*}
&&  \digamma  (x,y,z)\\
&=&   (-1)^{\vert y' \vert  \vert x\vert_n + \vert z'\vert \vert x'' y'' \vert + \vert x' z'\vert \vert  \rho(x'',y'')' \vert + \vert\rho(x'',y'')''\vert \vert \rho( \rho(x'',y'')''', z'')'\vert  }\\
&& \cdot y' s  \big(\rho( x'', y'')'\big)  x' \otimes z' s \big(\rho\big( \rho(x'',y'')''', z''\big)' \big)\rho(x'',y'')'' \otimes  \rho\big( \rho(x'',y'')''', z''\big)''
\end{eqnarray*}
for any homogeneous $x,y,z\in A$.  The \emph{tritensor map   $\triplep{-}{-}{-} = \triplep{-}{-}{-}_\rho$    induced by}   $\rho$ is defined by
\begin{equation} \label{triring}
\triplep{-}{-}{-}   =  \sum_{i=0}^2 \Perm^i_{312} \circ \digamma \circ \Perm^{-i}_{312,n} \in  \End(A^{\otimes 3}).
\end{equation}
If this  endomorphism of $A^{\otimes 3}$   is  identically equal to zero, then we say that $\rho$ is \emph{Gerstenhaber}    \emph{of degree~$n$}.
For instance, as explained in  Appendix~\ref{FPvDB},
the Fox pairing   in Example~\ref{Examples-rho}.3     is Gerstenhaber of degree $ 2-d$.

\subsection{The main  theorem}  \label{Main result}

   We    state   our main theorem  concerning the Jacobi identity.

\begin{theor}\label{MAINMAIN}
Let $A$ be a cocommutative graded Hopf algebra carrying an   antisymmetric Fox pairing $\rho$ of degree $n\in \ZZ$.
Let $B$ be a commutative ungraded Hopf algebra   endowed    with a trace-like element $t\in B$.
If $\rho$ is Gerstenhaber, then so is the $n$-graded  bracket $\bracket{-}{-}$ in the algebra $A_B$  produced by
   Theorem \ref{double_to_simple}    from   $\rho$ and    $\bullet = \bullet_t : B \times B \to \kk$.
\end{theor}

  This theorem is   a direct consequence of the following lemma.
   To state the lemma,  we let $s=s_B$   be  the antipode of   $B$ and  define a bilinear map  $\curlyvee =  \curlyvee_t   :B \times B \to B$    by
\begin{equation} \label{ltimes}
b \curlyvee  c  =  ( b' \bullet  c'' )  \,   b''     s    (c')
\end{equation}
for any $b,c\in  B$. By \eqref{balanced_bis}, we   also   have
\begin{equation} \label{ltimes_bis}
b \curlyvee  c  =     ( b'' \bullet  c' )  \,   b'     s    (c'').
\end{equation}

\begin{lemma} \label{Jacobitor}
 Set   $p=t+s (t)\in B$. Then, for   any  homogeneous $x,y,z\in A$  and~$b,c,d\in B$,
$$
 (-1)^{\vert x \vert_n \vert z \vert_n } \bracket{x_b}{\bracket{y_c}{z_d}} + (-1)^{\vert z \vert_n \vert y \vert_n } \bracket{z_d}{\bracket{x_b}{y_c}} + (-1)^{\vert y \vert_n \vert x \vert_n } \bracket{y_c}{\bracket{z_d}{x_b}} 
 $$ \vspace{-0.7cm}
 \begin{eqnarray}
\notag &=& -(-1)^{\vert x'' \vert \vert y' z' \vert + \vert y'' \vert \vert z'\vert_n+\vert x \vert_n \vert z \vert_n} \cdot\\
\label{||-,-,-||} && \quad  \triplep{x'}{y'}{z'}^\ell_{p^{(1)}}\, \triplep{x'}{y'}{z'}^m_{p^{(5)}}\, \triplep{x'}{y'}{z'}^r_{p^{(3)}}   x''_{b \curlyvee p^{(2)}}   y''_{ c \curlyvee p^{(6)}} z''_{d \curlyvee p^{(4)}}
\end{eqnarray}
where  the   tritensor map    induced by   $\rho$ is expanded  in  the form  
$$\triplep{-}{-}{-}=\triplep{-}{-}{-}^\ell \otimes \triplep{-}{-}{-}^m \otimes \triplep{-}{-}{-}^r.$$
\end{lemma}

\begin{proof}
 Applying  \eqref{formulafromt3} and   the Leibniz rule, we obtain
\begin{eqnarray*}
\bracket{x_b}{\bracket{y_c}{z_d}} &=&
(-1)^{\vert y'' \vert \vert z'\vert_n} (c'\bullet t^{(2)}) (d'' \bullet t^{(4)} ) \bracket{x_b} { \rho(y',z')_{ s(t^{(1)}) t^{(3)}}\,   y''_{c''}\,   z''_{d'}}\\
&=& P(x,y,z;b,c,d) +  Q(x,y,z;b,c,d) +   R(x,y,z;b,c,d)
\end{eqnarray*}
where
\begin{eqnarray*}
 P(x,y,z;b,c,d) &=&  (-1)^{\vert y'' \vert \vert z'\vert_n} (c'\bullet t^{(2)}) (d'' \bullet t^{(4)} ) \bracket{x_b} { \rho(y',z')_{ s(t^{(1)}) t^{(3)}}}   y''_{c''}\,   z''_{d'} ,\\[0.2cm]
 Q(x,y,z;b,c,d)  &=& (-1)^{\vert y'' \vert \vert z'\vert_n+ \vert y' z' \vert_n \vert x \vert_n } (c'\bullet t^{(2)}) (d'' \bullet t^{(4)} )\cdot\\
 &&  \rho(y',z')_{ s(t^{(1)}) t^{(3)}}  \bracket{x_b} {    y''_{c''}} \,   z''_{d'}   ,\\[0.2cm]
R(x,y,z;b,c,d) &=&  (-1)^{ \vert y'' \vert \vert z'\vert_n + \vert y z'\vert_n \vert x \vert_n } (c'\bullet t^{(2)}) (d'' \bullet t^{(4)} )\cdot \\
&& \rho(y',z')_{ s(t^{(1)}) t^{(3)}}\,  y''_{c''} \bracket{x_b} {    z''_{d'}  }  .
\end{eqnarray*}
We claim that   \begin{equation}\label{formulaidid}
 (-1)^{\vert x \vert_n \vert z \vert_n} Q(x,y,z;b,c,d) =- (-1)^{\vert y \vert_n \vert z \vert_n} R (z,x,y;d,b,c) .
\end{equation}
To prove \eqref{formulaidid},    we let  $u$ be another ``copy'' of the element $t\in B$. Then
\begin{eqnarray*}
&& Q(x,y,z;b,c,d) \\
&=&  (-1)^{\vert y'' \vert \vert z'\vert_n+ \vert y'z'\vert_n \vert x \vert_n } (c'\bullet t^{(2)}) (d'' \bullet t^{(4)} )\, \rho(y',z')_{ s(t^{(1)}) t^{(3)}}  \bracket{x_b} {    y''_{c''}}    z''_{d'} \\
& \by{formulafromt3} &  (-1)^{\vert y'' y''' \vert \vert z'\vert_n+ \vert y' z' \vert_n \vert x \vert_n + \vert x'' \vert  \vert y''\vert_n  } \,
 (c'\bullet t^{(2)}) (d'' \bullet t^{(4)} ) (b' \bullet u^{(2)}) (c''' \bullet u^{(4)}) \cdot  \\
&&   \rho(y',z')_{ s(t^{(1)}) t^{(3)}}\, \rho(x',y'')_{s(u^{(1)})u^{(3)}}\, x''_{b''}\,  y'''_{c''}  \,   z''_{d'}
\end{eqnarray*}
and
\begin{eqnarray*}
 && R(z,x,y;d,b,c) \\
 &=&   (-1)^{\vert x'' \vert \vert y'\vert_n + \vert x y'\vert_n \vert z \vert_n  } (b'\bullet t^{(2)}) (c'' \bullet t^{(4)} )\,  \rho(x',y')_{ s(t^{(1)}) t^{(3)}}\, x''_{b''}\,  \bracket{z_d} {    y''_{c'}  } \\
&=& - (-1)^{ \vert x'' \vert \vert y'\vert_n + \vert x y\vert \vert z \vert_n } (b'\bullet t^{(2)}) (c'' \bullet t^{(4)} )\, \rho(x',y')_{ s(t^{(1)}) t^{(3)}}\,  x''_{b''}\,  \bracket{    y''_{c'}  }{z_d}   \\
& \by{formulafromt3} & - (-1)^{\vert x'' \vert \vert y'\vert_n + \vert x y\vert \vert z \vert_n + \vert y'''\vert \vert z'\vert_n} \,
 (b'\bullet t^{(2)}) (c''' \bullet t^{(4)} ) ( c' \bullet u^{(2)}) ( d'' \bullet u^{(4)}) \cdot \\
&& \quad  \rho(x',y')_{ s(t^{(1)}) t^{(3)}}\,  x''_{b''}\,  \rho(y'',z')_{s(u^{(1)})u^{(3)}}\, y'''_{c''}\,  z''_{d'} \\
&= & - (-1)^{\vert x'' \vert \vert y'\vert_n + \vert x y\vert \vert z \vert_n + \vert y'''\vert \vert z'\vert_n+\vert x y' \vert_n \vert y'' z' \vert_n} \cdot\\
&& \quad (b'\bullet t^{(2)}) (c''' \bullet t^{(4)} ) ( c' \bullet u^{(2)}) ( d'' \bullet u^{(4)}) \cdot \\
&&   \quad  \rho(y'',z')_{s(u^{(1)})u^{(3)}}\, \rho(x',y')_{ s(t^{(1)}) t^{(3)}}\,  x''_{b''}\, y'''_{c''}\,  z''_{d'} \\
&= & - (-1)^{\vert x'' \vert \vert y''\vert_n + \vert x y\vert \vert z \vert_n + \vert y'''\vert \vert z'\vert_n+\vert x y'' \vert_n \vert y' z' \vert_n + \vert y' \vert \vert y'' \vert} \cdot \\
&& \quad (b'\bullet t^{(2)}) (c''' \bullet t^{(4)} ) ( c' \bullet u^{(2)}) ( d'' \bullet u^{(4)}) \cdot \\
&&   \quad  \rho(y',z')_{s(u^{(1)})u^{(3)}}\, \rho(x',y'')_{ s(t^{(1)}) t^{(3)}}\,  x''_{b''}\, y'''_{c''}\,  z''_{d'} \\
&= & - (-1)^{\vert x'' \vert \vert y''\vert_n + \vert x y\vert \vert z \vert_n + \vert y''y'''\vert \vert z'\vert_n+\vert x \vert_n \vert y' z' \vert_n } \cdot \\
&& \quad  (b'\bullet t^{(2)}) (c''' \bullet t^{(4)} ) ( c' \bullet u^{(2)}) ( d'' \bullet u^{(4)}) \cdot \\
&& \quad    \rho(y',z')_{s(u^{(1)})u^{(3)}}\, \rho(x',y'')_{ s(t^{(1)}) t^{(3)}}\,  x''_{b''}\, y'''_{c''}\,  z''_{d'}.
\end{eqnarray*}
 Comparing these expressions, we   obtain \eqref{formulaidid}.
Formula \eqref{formulaidid}  implies that all   $Q$-terms and   $R$-terms   on the left-hand side of
\eqref{||-,-,-||} cancel out.   It remains to compute
$$
\hbox{\small $(-1)^{\vert x \vert_n \vert z \vert_n} P(x,y,z;b,c,d) + (-1)^{\vert y \vert_n \vert x \vert_n} P(y,z,x;c,d,b) 
+ (-1)^{\vert z \vert_n \vert y \vert_n} P(z,x,y;d,b,c).$}
$$
 To this end,   we expand
\begin{eqnarray*}
&& P(x,y,z;b,c,d)\\
&=&  (-1)^{\vert y'' \vert \vert z'\vert_n} (c'\bullet t^{(2)}) (d'' \bullet t^{(4)} ) \bracket{x_b} { \rho(y',z')_{ s(t^{(1)}) t^{(3)}}}   y''_{c''}\,   z''_{d'}  \\
 &=&  (-1)^{\vert y'' \vert \vert z'\vert_n} (c'\bullet t^{(2)}) (d'' \bullet t^{(4)} ) \bracket{x_b} { \rho(y',z')'_{ s(t^{(1)}) }   \rho(y',z')''_{  t^{(3)}}  }   y''_{c''}\,   z''_{d'}  \\
&  =  & P_1(x,y,z;b,c,d) + P_2(x,y,z;b,c,d)
\end{eqnarray*}
where
\begin{eqnarray*}
P_1(x,y,z;b,c,d) &=&  (-1)^{\vert y'' \vert \vert z'\vert_n} (c'\bullet t^{(2)}) (d'' \bullet t^{(4)} ) \cdot \\
&& \bracket{x_b} { \rho(y',z')'_{ s(t^{(1)}) }   }  \rho(y',z')''_{  t^{(3)}}\,   y''_{c''}\,   z''_{d'}   \\
\end{eqnarray*}
and
\begin{eqnarray*}
 P_2(x,y,z;b,c,d)
&=& (-1)^{\vert y'' \vert \vert z'\vert_n+ \vert x \vert_n \vert \rho(y',z')' \vert} (c'\bullet t^{(2)}) (d'' \bullet t^{(4)} ) \cdot \\
&& \rho(y',z')'_{ s(t^{(1)}) } \bracket{x_b} {    \rho(y',z')''_{  t^{(3)}}  }   y''_{c''}\,   z''_{d'}    .
\end{eqnarray*}
We claim that
\begin{eqnarray}
\notag &&  (-1)^{\vert x \vert_n \vert z \vert_n}    P_2(x,y,z;b,c,d)  \\
&& +  (-1)^{\vert y \vert_n \vert x \vert_n}  P_2(y,z,x;c,d,b) +  (-1)^{\vert z \vert_n \vert y \vert_n}  P_2(z,x,y;d,b,c) \\
\notag  &=&  - (-1)^{ \vert x'' \vert \vert y' z' \vert + \vert y'' \vert \vert z'\vert_n+\vert x \vert_n \vert z \vert_n} \cdot \\
\label{P2} &&\quad  \triplep{x'}{y'}{z'}^\ell_{t^{(1)}} \triplep{x'}{y'}{z'}^m_{t^{(5)}}\, \triplep{x'}{y'}{z'}^r_{t^{(3)}}
x''_{b \curlyvee t^{(2)}}   y''_{ c \curlyvee t^{(6)}} z''_{d \curlyvee t^{(4)}}
\end{eqnarray}
and similarly that
\begin{eqnarray}
\notag &&  (-1)^{\vert x \vert_n \vert z \vert_n}    P_1(x,y,z;b,c,d) \\
&& +  (-1)^{\vert y \vert_n \vert x \vert_n}  P_1(y,z,x;c,d,b) +  (-1)^{\vert z \vert_n \vert y \vert_n}  P_1(z,x,y;d,b,c) \\
\notag   &=&   - (-1)^{\vert x'' \vert \vert y' z' \vert + \vert y'' \vert \vert z'\vert_n+\vert x \vert_n \vert z \vert_n}  \cdot\\
\label{P1} && \quad \triplep{x'}{y'}{z'}^\ell_{v^{(1)}} \triplep{x'}{y'}{z'}^m_{v^{(5)}}\, \triplep{x'}{y'}{z'}^r_{v^{(3)}}   x''_{b \curlyvee v^{(2)}}   y''_{ c \curlyvee v^{(6)}} z''_{d \curlyvee v^{(4)}}
\end{eqnarray}
where  $v=s(t)$.
Since $p=t+v$, these two claims will  imply \eqref{||-,-,-||}.  Observe that
\begin{eqnarray*}
\!\!\! && \!\!\! P_1(x,y,z;b,c,d) \\
\!\!\! &=& \!\!\!  (-1)^{\vert y'' \vert \vert z'\vert_n} \big(c'\bullet s(v^{(3)})\big) \big(d'' \bullet s(v^{(1)} )\big) \bracket{x_b} { \rho(y',z')'_{v^{(4)}}   }  \rho(y',z')''_{  s(v^{(2)})}\,   y''_{c''}\,   z''_{d'}  \\
\!\!\! &=& \!\!\! (-1)^{\vert y'' \vert \vert z'\vert_n} \big(c'\bullet v^{(3)}\big) \big(d'' \bullet v^{(1)} \big) \bracket{x_b} { \rho(y',z')'_{v^{(4)}}   }  \rho(y',z')''_{  s(v^{(2)})}\,   y''_{c''}\,   z''_{d'}  \\
\!\!\! &  \by{cyclic_inv}& \!\!\! (-1)^{\vert y'' \vert \vert z'\vert_n} \big(c'\bullet v^{(2)}\big) \big(d'' \bullet v^{(4)} \big) \bracket{x_b} { \rho(y',z')'_{v^{(3)}}   }  \rho(y',z')''_{  s(v^{(1)})}\,   y''_{c''}\,   z''_{d'} \\
\!\!\! &=& \!\!\!  (-1)^{\vert y'' \vert \vert z'\vert_n+ \vert \rho(y',z')''\vert \vert x \rho(y',z')'\vert_n} \big(c'\bullet v^{(2)}\big) \big(d'' \bullet v^{(4)} \big) \cdot \\
\!\!\! && \!\!\! \rho(y',z')''_{  s(v^{(1)})}\,  \bracket{x_b} { \rho(y',z')'_{v^{(3)}}   }   y''_{c''}\,   z''_{d'} \\
 \!\!\! &=& \!\!\!  (-1)^{\vert y'' \vert \vert z'\vert_n+ \vert \rho(y',z')'\vert \vert x\vert_n} \big(c'\bullet v^{(2)}\big) \big(d'' \bullet v^{(4)} \big)\cdot\\
 \!\!\! & & \!\!\!  \rho(y',z')'_{  s(v^{(1)})}\,  \bracket{x_b} { \rho(y',z')''_{v^{(3)}}   }   y''_{c''}\,   z''_{d'},
\end{eqnarray*}
which shows that $P_1(x,y,z;b,c,d)$ is obtained from $P_2(x,y,z;b,c,d)$ by the change $t\leadsto v$.
Since the balanced biderivation $\bullet= \bullet_t$ coincides with $\bullet_v$,  \eqref{P1} is equivalent to \eqref{P2}.
Thus we need only   to prove \eqref{P2}. To this end,   we  compute
{\small
\begin{eqnarray*}
\!\!\!\!\!\!\!\!\!&& P_2(x,y,z;b,c,d) \\
\!\!\!\!\!\!\!\!\!&\!\!\!=\!\!\!&\!\!\! -(-1)^{\vert y'' \vert \vert z'\vert_n+ \vert x \vert_n \vert y'z' \vert} (c'\bullet t^{(2)}) (d'' \bullet t^{(4)} )
 \rho(y',z')'_{ s(t^{(1)}) } \bracket{    \rho(y',z')''_{  t^{(3)}}  }{x_b}   y''_{c''}\,   z''_{d'}  \\
 \!\!\!\!\!\!\!\!\!& \!\!\!\by{main_formula} \!\!\!& \!\!\!  -(-1)^{\vert y'' \vert \vert z'\vert_n+ \vert x \vert_n \vert y'z' \vert+ \vert  \rho(y',z')''' \vert \vert x'\vert_n}
 (c'\bullet t^{(2)}) (d'' \bullet t^{(7)} ) ( b'' \bullet t^{(4)} )\cdot \\
\!\!\!\!\!\!\!\!\!&&\!\!\! \quad \rho(y',z')'_{ s(t^{(1)}) }\, \rho\big(\rho(y',z')'' ,x'\big)_{s(t^{(5)}) t^{(3)}}\, \rho(y',z')'''_{t^{(6)}}\, x''_{b'}\,  y''_{c''}\,   z''_{d'} \\
 \!\!\!\!\!\!\!\!\!&\!\!\! = \!\!\! &\!\!\!   -(-1)^{\vert y'' \vert \vert z'\vert_n+ \vert x \vert_n \vert y'z' \vert+ \vert  \rho(y',z')''' \vert \vert x'\vert_n}
 (c'\bullet t^{(2)}) (d'' \bullet t^{(7)} ) ( b'' \bullet t^{(4)} )\cdot \\
 \!\!\!\!\!\!\!\!\!&\!\!\!&\!\!\! \quad \rho(y',z')'_{ s(t^{(1)}) }\, s\big(\rho\big(\rho(y',z')'' ,x'\big)'\big)_{t^{(5)} }\, \rho\big(\rho(y',z')'' ,x'\big)''_{t^{(3)}}\, \rho(y',z')'''_{t^{(6)}}\, x''_{b'}\,  y''_{c''}\,   z''_{d'} \\
 \!\!\!\!\!\!\!\!\!&\!\!\! =\!\!\!  &\!\!\!   -(-1)^{\vert y'' \vert \vert z'\vert_n+ \vert x \vert_n \vert y'z' \vert+ \vert  \rho(y',z')''' \vert \vert x' \rho(\rho(y',z')'' ,x')''\vert_n}
 (c'\bullet t^{(2)}) (d'' \bullet t^{(6)} ) ( b'' \bullet t^{(4)} )\cdot \\
 \!\!\!\!\!\!\!\!\!&\!\!\!&\!\!\! \quad \rho(y',z')'_{ s(t^{(1)}) }\, \left(s\big(\rho\big(\rho(y',z')'' ,x'\big)'\big)\, \rho(y',z')'''\right)_{t^{(5)}}\, \rho\big(\rho(y',z')'' ,x'\big)''_{t^{(3)}}\, x''_{b'}\,  y''_{c''}\,   z''_{d'} \\
 \!\!\!\!\!\!\!\!\!&\!\!\! = \!\!\! &\!\!\!   -(-1)^{\vert y'' \vert \vert z'\vert_n+ \vert x \vert_n \vert y'z' \vert+ \vert  \rho(y',z')''' \vert \vert x' \rho(\rho(y',z')'' ,x')''\vert_n+ \vert xyz'\vert\vert z''\vert} \cdot \\
 \!\!\!\!\!\!\!\!\!&\!\!\!&\!\!\!  \quad (c'\bullet t^{(2)}) (d'' \bullet t^{(6)} ) ( b'' \bullet t^{(4)} )\cdot \\
 \!\!\!\!\!\!\!\!\!&\!\!\!&\!\!\! \quad z''_{d'}\, \rho(y',z')'_{ s(t^{(1)}) }\, \left(s\big(\rho\big(\rho(y',z')'' ,x'\big)'\big)\, \rho(y',z')'''\right)_{t^{(5)}}\, \rho\big(\rho(y',z')'' ,x'\big)''_{t^{(3)}}\, x''_{b'}\,  y''_{c''} \\
 \!\!\!\!\!\!\!\!\!&\!\!\! = \!\!\! &\!\!\!   -(-1)^{\vert y'' \vert \vert z' x \rho(y',z')'' \rho(y',z')''' \vert+ \vert x \vert_n \vert y'z' \vert+ \vert  \rho(y',z')''' \vert \vert  x' \rho(\rho(y',z')'' ,x')''\vert_n+ \vert xyz'\vert\vert z''\vert} \cdot \\
 \!\!\!\!\!\!\!\!\!&\!\!\!&\!\!\!  \quad (c'\bullet t^{(2)}) (d'' \bullet t^{(6)} ) ( b'' \bullet t^{(4)} )\cdot \\
 \!\!\!\!\!\!\!\!\!&\!\!\!&\!\!\! \quad z''_{d'}\,   \rho(y',z')'_{ s(t^{(1)}) }\, y''_{c''}\, \left(s\big(\rho\big(\rho(y',z')'' ,x'\big)'\big)\, \rho(y',z')'''\right)_{t^{(5)}}\, \rho\big(\rho(y',z')'' ,x'\big)''_{t^{(3)}}\, x''_{b'} \\
 \!\!\!\!\!\!\!\!\!&\!\!\! =  \!\!\!&\!\!\!   -(-1)^{\vert y'' \vert \vert   \rho(y',z')' y'  \vert+ \vert x \vert_n \vert y z' \vert+ \vert  \rho(y',z')''' \vert \vert  x' \rho(\rho(y',z')'' ,x')''\vert_n+ \vert xyz'\vert\vert z''\vert} \cdot \\
 \!\!\!\!\!\!\!\!\!&\!\!\!&\!\!\!  \quad (c'\bullet t^{(2)}) (d'' \bullet t^{(6)} ) ( b'' \bullet t^{(4)} )\cdot \\
 \!\!\!\!\!\!\!\!\!&\!\!\!&\!\!\! \quad z''_{d'}\,   \rho(y',z')'_{ s(t^{(1)}) }\, y''_{c''}\, \left(s\big(\rho\big(\rho(y',z')'' ,x'\big)'\big)\, \rho(y',z')'''\right)_{t^{(5)}}\, \rho\big(\rho(y',z')'' ,x'\big)''_{t^{(3)}}\, x''_{b'} \\
 \!\!\!\!\!\!\!\!\!&\!\!\! = \!\!\! &\!\!\!   -(-1)^{\vert y'' \vert \vert   \rho(y',z')' y'  \vert+ \vert x \vert_n \vert y z' \vert+ \vert  \rho(y',z')''' \vert \vert  x' \rho(\rho(y',z')'' ,x')''\vert_n}\cdot\\
 \!\!\!\!\!\!\!\!\!&\!\!\!&\!\!\! \quad (-1)^{  \vert xyz'\vert\vert z''\vert+\vert x'' \vert \vert x'y'z' \rho(y',z')'\vert} \cdot  \\
 \!\!\!\!\!\!\!\!\!&\!\!\!&\!\!\!  \quad (c'\bullet t^{(2)}) (d'' \bullet t^{(6)} ) ( b'' \bullet t^{(4)} )\cdot \\
 \!\!\!\!\!\!\!\!\!&\!\!\!&\!\!\! \quad z''_{d'}\,   \rho(y',z')'_{ s(t^{(1)}) }\, y''_{c''}\, x''_{b'}\, \left(s\big(\rho\big(\rho(y',z')'' ,x'\big)'\big)\, \rho(y',z')'''\right)_{t^{(5)}}\, \rho\big(\rho(y',z')'' ,x'\big)''_{t^{(3)}} \\
 \!\!\!\!\!\!\!\!\!& \!\!\!\by{cyclic_inv} \!\!\!&\!\!\!    -(-1)^{\vert y'' \vert \vert   \rho(y',z')' y'  \vert+ \vert x \vert_n \vert y z' \vert+ \vert  \rho(y',z')''' \vert \vert  x' \rho(\rho(y',z')'' ,x')''\vert_n} \cdot\\
  \!\!\!\!\!\!\!\!\!&\!\!\!&\!\!\! \quad (-1)^{ \vert xyz'\vert\vert z''\vert+\vert x'' \vert \vert x'y'z' \rho(y',z')'\vert} \cdot \\
 \!\!\!\!\!\!\!\!\!&\!\!\!&\!\!\!  \quad (c'\bullet t^{(5)}) (d'' \bullet t^{(3)} ) ( b'' \bullet t^{(1)} )\cdot \\
 \!\!\!\!\!\!\!\!\!&\!\!\!&\!\!\! \quad z''_{d'}\,   s\big(\rho(y',z')'\big)_{ t^{(4)} }\, y''_{c''}\, x''_{b'}\, \left(s\big(\rho\big(\rho(y',z')'' ,x'\big)'\big)\, \rho(y',z')'''\right)_{t^{(2)}}\, \rho\big(\rho(y',z')'' ,x'\big)''_{t^{(6)}}\\
 \!\!\!\!\!\!\!\!\!&\!\!\! = \!\!\!& \!\!\!   -(-1)^{\vert y'' \vert \vert   \rho(y',z')' y'  \vert+ \vert x \vert_n \vert y z' \vert+ \vert  \rho(y',z')''' \vert \vert  x' \rho(\rho(y',z')'' ,x')''\vert_n }\cdot \\
 \!\!\!\!\!\!\!\!\!&\!\!\!&\!\!\! \quad (-1)^{  \vert xyz'\vert\vert z''\vert+\vert x'' \vert \vert x'y'z' \rho(y',z')'\vert} \cdot\\
 \!\!\!\!\!\!\!\!\!&\!\!\!&\!\!\!  \quad (c'\bullet t^{(11)}) (d'' \bullet t^{(5)} ) ( b'' \bullet t^{(1)} ) \quad z''_{d' s(t^{(6)}) t^{(7)}  }\,   s\big(\rho(y',z')'\big)_{ t^{(8)} }\, y''_{t^{(9)} s(t^{(10)}) c''   } \cdot \\
 \!\!\!\!\!\!\!\!\!&\!\!\!&\!\!\!   \quad x''_{b' s(t^{(2)}) t^{(3)}  }\, \left(s\big(\rho\big(\rho(y',z')'' ,x'\big)'\big)\, \rho(y',z')'''\right)_{t^{(4)}}\, \rho\big(\rho(y',z')'' ,x'\big)''_{t^{(12)}} \\
 \!\!\!\!\!\!\!\!\!&\!\!\! = \!\!\!&\!\!\!    -(-1)^{\vert y'' y''' \vert \vert   \rho(y',z')' y'  \vert+ \vert x \vert_n \vert y z' \vert+ \vert  \rho(y',z')''' \vert \vert  x' \rho(\rho(y',z')'' ,x')''\vert_n} \cdot \\
 \!\!\!\!\!\!\!\!\!&\!\!\!&\!\!\! \quad (-1)^{\vert xyz'\vert\vert z'' z'''\vert+\vert x'' x''' \vert \vert x'y'z' \rho(y',z')'\vert} \cdot \\
 \!\!\!\!\!\!\!\!\!&\!\!\!&\!\!\!  \quad (c'\bullet t^{(8)}) (d'' \bullet t^{(4)} ) ( b'' \bullet t^{(1)} ) \ z''_{d' s(t^{(5)}) }\,   \left(z''' s\big(\rho(y',z')'\big)y'' \right) _{ t^{(6)} }\, y'''_{ s(t^{(7)}) c''   } \cdot \\
 \!\!\!\!\!\!\!\!\!&\!\!\!&\!\!\!   \quad x''_{b' s(t^{(2)})  }\, \left( x''' s\big(\rho\big(\rho(y',z')'' ,x'\big)'\big)\, \rho(y',z')'''\right)_{t^{(3)}}\, \rho\big(\rho(y',z')'' ,x'\big)''_{t^{(9)}} \\
 \!\!\!\!\!\!\!\!\!& \!\!\!\stackrel{\eqref{ltimes}, \eqref{ltimes_bis}}{=}  \!\!\!&\!\!\!    -(-1)^{\vert y'' y''' \vert \vert   \rho(y',z')' y'  \vert+ \vert x \vert_n \vert y z' \vert+ \vert  \rho(y',z')''' \vert \vert  x' \rho(\rho(y',z')'' ,x')''\vert_n}\cdot \\
 \!\!\!\!\!\!\!\!\!&\!\!\!&\!\!\! \quad (-1)^{ \vert xyz'\vert\vert z''z'''\vert+\vert x'' x'''\vert \vert x'y'z' \rho(y',z')'\vert} \cdot\\
 \!\!\!\!\!\!\!\!\!&\!\!\!&\!\!\!   \quad z''_{d \curlyvee t^{(3)} }\,   \left(z''' s\big(\rho(y',z')'\big)y'' \right) _{ t^{(4)} }\, y'''_{ c \curlyvee t^{(5)}   } \cdot \\
 \!\!\!\!\!\!\!\!\!&\!\!\!&\!\!\!   \quad x''_{b \curlyvee  t^{(1)} }\, \left( x''' s\big(\rho\big(\rho(y',z')'' ,x'\big)'\big)\, \rho(y',z')'''\right)_{t^{(2)}}\, \rho\big(\rho(y',z')'' ,x'\big)''_{t^{(6)}}\\
 \!\!\!\!\!\!\!\!\!&\!\!\! =\!\!\! &\!\!\!     -(-1)^{\vert y'' y''' \vert \vert   \rho(y',z'')' y'  \vert+ \vert x \vert_n \vert y z'' \vert+ \vert  \rho(y',z'')''' \vert \vert  x' \rho(\rho(y',z'')'' ,x')''\vert_n} \cdot \\
 \!\!\!\!\!\!\!\!\!&\!\!\!&\!\!\! \quad (-1)^{ \vert xy\vert\vert z'z'''\vert+ \vert z'''\vert \vert z' z'' \vert +\vert x'' x'''\vert \vert x'y'z'' \rho(y',z'')'\vert} \cdot \\
 \!\!\!\!\!\!\!\!\!&\!\!\!&\!\!\!   \quad z'''_{d \curlyvee t^{(3)} }\,   \left(z' s\big(\rho(y',z'')'\big)y'' \right) _{ t^{(4)} }\, y'''_{ c \curlyvee t^{(5)}   } \cdot \\
 \!\!\!\!\!\!\!\!\!&\!\!\!&\!\!\!   \quad x''_{b \curlyvee  t^{(1)} }\, \left( x''' s\big(\rho\big(\rho(y',z'')'' ,x'\big)'\big)\, \rho(y',z'')'''\right)_{t^{(2)}}\, \rho\big(\rho(y',z'')'' ,x'\big)''_{t^{(6)}} \\
 \!\!\!\!\!\!\!\!\!&\!\!\! =\!\!\! &\!\!\!  -(-1)^{\vert  y''' \vert \vert   \rho(y'',z'')' y''  \vert+  \vert y' \vert \vert \rho(y'',z'')' \vert + \vert x \vert_n \vert y z'' \vert+ \vert  \rho(y'',z'')''' \vert \vert  x' \rho(\rho(y'',z'')'' ,x')''\vert_n} \cdot \\
 \!\!\!\!\!\!\!\!\!&\!\!\!&\!\!\! \quad (-1)^{ \vert xy\vert\vert z'z'''\vert+ \vert z'''\vert \vert z' z'' \vert +\vert x'' x'''\vert \vert x'y''z'' \rho(y'',z'')'\vert} \cdot  \\
 \!\!\!\!\!\!\!\!\!&\!\!\!&\!\!\!   \quad z'''_{d \curlyvee t^{(3)} }\,   \left(z' s\big(\rho(y'',z'')'\big)y' \right) _{ t^{(4)} }\, y'''_{ c \curlyvee t^{(5)}   } \cdot \\
 \!\!\!\!\!\!\!\!\!&\!\!\!&\!\!\!   \quad x''_{b \curlyvee  t^{(1)} }\, \left( x''' s\big(\rho\big(\rho(y'',z'')'' ,x'\big)'\big)\, \rho(y'',z'')'''\right)_{t^{(2)}}\, \rho\big(\rho(y'',z'')'' ,x'\big)''_{t^{(6)}} \\
 \!\!\!\!\!\!\!\!\!&\!\!\! = \!\!\!& \!\!\!   -(-1)^{\vert  y''' \vert \vert   \rho(y'',z'')' y''  \vert+  \vert y' \vert \vert \rho(y'',z'')' \vert + \vert x \vert_n \vert y z'' \vert+ \vert  \rho(y'',z'')''' \vert \vert  x'' \rho(\rho(y'',z'')'' ,x'')''\vert_n}\cdot \\
 \!\!\!\!\!\!\!\!\!&\!\!\!&\!\!\! \quad  (-1)^{ \vert xy\vert\vert z'z'''\vert+ \vert z'''\vert \vert z' z'' \vert +\vert x''' x'\vert \vert y''z'' \rho(y'',z'')'\vert+ \vert x'''\vert \vert x' x'' \vert} \cdot  \\
 \!\!\!\!\!\!\!\!\!&\!\!\!&\!\!\!   \quad z'''_{d \curlyvee t^{(3)} }\,   \left(z' s\big(\rho(y'',z'')'\big)y' \right) _{ t^{(4)} }\, y'''_{ c \curlyvee t^{(5)}   } \cdot \\
 \!\!\!\!\!\!\!\!\!&\!\!\!&\!\!\!   \quad x'''_{b \curlyvee  t^{(1)} }\, \left( x' s\big(\rho\big(\rho(y'',z'')'' ,x''\big)'\big)\, \rho(y'',z'')'''\right)_{t^{(2)}}\, \rho\big(\rho(y'',z'')'' ,x''\big)''_{t^{(6)}} \\
 \!\!\!\!\!\!\!\!\!&\!\!\! = \!\!\!& \!\!\!   -(-1)^{\vert  y''' \vert \vert   \rho(y'',z'')' y''  \vert+  \vert y' \vert \vert \rho(y'',z'')' \vert+ \vert x \vert_n \vert y z'' \vert+ \vert  \rho(y'',z'')'' \vert \vert  \rho(\rho(y'',z'')''' ,x'')'\vert} \cdot \\
 \!\!\!\!\!\!\!\!\!&\!\!\!&\!\!\! \quad (-1)^{\vert xy\vert\vert z'z'''\vert+ \vert z'''\vert \vert z' z'' \vert +\vert x''' x'\vert \vert y''z'' \rho(y'',z'')'\vert+ \vert x'''\vert \vert x' x'' \vert} \cdot  \\
 \!\!\!\!\!\!\!\!\!&\!\!\!&\!\!\!   \quad z'''_{d \curlyvee t^{(3)} }\,   \left(z' s\big(\rho(y'',z'')'\big)y' \right) _{ t^{(4)} }\, y'''_{ c \curlyvee t^{(5)}   } \cdot \\
 \!\!\!\!\!\!\!\!\!&\!\!\!&\!\!\!   \quad x'''_{b \curlyvee  t^{(1)} }\, \left( x' s\big(\rho\big(\rho(y'',z'')''' ,x''\big)'\big)\, \rho(y'',z'')''\right)_{t^{(2)}}\, \rho\big(\rho(y'',z'')''' ,x''\big)''_{t^{(6)}} \\
 \!\!\!\!\!\!\!\!\!& \!\!\!= \!\!\!& \!\!\!  -(-1)^{\vert  y''' \vert \vert   \rho(y'',z'')' y''  \vert+  \vert y' \vert \vert \rho(y'',z'')' \vert +\vert x \vert_n \vert y z'' \vert+ \vert  \rho(y'',z'')'' \vert \vert  \rho(\rho(y'',z'')''' ,x'')'\vert} \cdot \\
 \!\!\!\!\!\!\!\!\!&\!\!\!&\!\!\! \quad (-1)^{\vert xy\vert\vert z'z'''\vert+ \vert z'''\vert \vert z' z'' \vert +\vert  x'\vert \vert y''z'' \rho(y'',z'')'\vert} \cdot   \\
 \!\!\!\!\!\!\!\!\!&\!\!\!&\!\!\!   \quad z'''_{d \curlyvee t^{(3)} }\,   \left(z' s\big(\rho(y'',z'')'\big)y' \right) _{ t^{(4)} }\, y'''_{ c \curlyvee t^{(5)}   } \cdot \\
 \!\!\!\!\!\!\!\!\!&\!\!\!&\!\!\!   \quad  \left( x' s\big(\rho\big(\rho(y'',z'')''' ,x''\big)'\big)\, \rho(y'',z'')''\right)_{t^{(2)}}\, \rho\big(\rho(y'',z'')''' ,x''\big)''_{t^{(6)}}\, x'''_{b \curlyvee  t^{(1)} } \\
 \!\!\!\!\!\!\!\!\!&\!\!\! =\!\!\! &\!\!\!  -(-1)^{\vert  y''' \vert \vert  z'' x  \vert+  \vert y' \vert \vert \rho(y'',z'')' \vert +\vert x \vert_n \vert y z'' \vert+ \vert  \rho(y'',z'')'' \vert \vert  \rho(\rho(y'',z'')''' ,x'')'\vert} \cdot \\
 \!\!\!\!\!\!\!\!\!&\!\!\!&\!\!\! \quad (-1)^{\vert xy\vert\vert z'z'''\vert+ \vert z'''\vert \vert z' z'' \vert +\vert  x'\vert \vert y''z'' \rho(y'',z'')'\vert}   \, z'''_{d \curlyvee t^{(3)} }\,   \left(z' s\big(\rho(y'',z'')'\big)y' \right) _{ t^{(4)} }  \cdot \\
 \!\!\!\!\!\!\!\!\!&\!\!\!&\!\!\!   \quad  \left( x' s\big(\rho\big(\rho(y'',z'')''' ,x''\big)'\big)\, \rho(y'',z'')''\right)_{t^{(2)}}\, \rho\big(\rho(y'',z'')''' ,x''\big)''_{t^{(6)}}\, x'''_{b \curlyvee  t^{(1)} }\, y'''_{ c \curlyvee t^{(5)}   } \\
 \!\!\!\!\!\!\!\!\!& \!\!\!= \!\!\!& \!\!\!  -(-1)^{\vert  y''' \vert \vert  z'' x  \vert+  \vert y' \vert \vert \rho(y'',z'')' \vert +\vert x \vert_n \vert y z'' \vert+ \vert  \rho(y'',z'')'' \vert \vert  \rho(\rho(y'',z'')''' ,x'')'\vert} \cdot \\
 \!\!\!\!\!\!\!\!\!&\!\!\!&\!\!\! \quad (-1)^{\vert xy\vert\vert z'\vert+\vert  x'\vert \vert y''z'' \rho(y'',z'')'\vert} \cdot   \\
 \!\!\!\!\!\!\!\!\!&\!\!\!&\!\!\!   \quad \,   \left(z' s\big(\rho(y'',z'')'\big)y' \right) _{ t^{(4)} }\,   \left( x' s\big(\rho\big(\rho(y'',z'')''' ,x''\big)'\big)\, \rho(y'',z'')''\right)_{t^{(2)}} \cdot \\
 \!\!\!\!\!\!\!\!\!&\!\!\!&\!\!\!   \quad  \rho\big(\rho(y'',z'')''' ,x''\big)''_{t^{(6)}}\, x'''_{b \curlyvee  t^{(1)} }\, y'''_{ c \curlyvee t^{(5)}   }\, z'''_{d \curlyvee t^{(3)} } \\
 \!\!\!\!\!\!\!\!\!&\!\!\! =\!\!\! &\!\!\!   -(-1)^{\vert  y''' \vert \vert  z'' x  \vert+ \vert xy\vert\vert z'\vert+\vert  x'\vert \vert y''z''\vert + \vert x \vert_n \vert y z'' \vert} \cdot \\
 \!\!\!\!\!\!\!\!\!&\!\!\! \!\!\! &\!\!\! \quad (-1)^{ \vert x'y' \vert \vert \rho(y'',z'')' \vert + \vert  \rho(y'',z'')'' \vert \vert  \rho(\rho(y'',z'')''' ,x'')'\vert} \cdot \\
 \!\!\!\!\!\!\!\!\!&\!\!\!&\!\!\!   \quad \,   \left(z' s\big(\rho(y'',z'')'\big)y' \right) _{ t^{(4)} }\,   \left( x' s\big(\rho\big(\rho(y'',z'')''' ,x''\big)'\big)\, \rho(y'',z'')''\right)_{t^{(2)}} \cdot \\
 \!\!\!\!\!\!\!\!\!&\!\!\!&\!\!\!   \quad  \rho\big(\rho(y'',z'')''' ,x''\big)''_{t^{(6)}}\, x'''_{b \curlyvee  t^{(1)} }\, y'''_{ c \curlyvee t^{(5)}   }\, z'''_{d \curlyvee t^{(3)} } \\
 \!\!\!\!\!\!\!\!\!&\!\!\! =\!\!\! &\!\!\!   -(-1)^{ \vert y z' \vert \vert x \vert_n + \vert y'' \vert \vert x z'\vert} \cdot \\
 \!\!\!\!\!\!\!\!\!&\!\!\!&\!\!\!  \quad \digamma(y',z',x')^\ell_{t^{(4)}}  \digamma(y',z',x')^m_{t^{(2)}}  \digamma(y',z',x')^r_{t^{(6)}}   x''_{b \curlyvee  t^{(1)} }\, y''_{ c \curlyvee t^{(5)}   }\, z''_{d \curlyvee t^{(3)} } \\
 \!\!\!\!\!\!\!\!\!&\!\!\! \by{cyclic_inv}\!\!\! &\!\!\!    -(-1)^{ \vert y z' \vert \vert x \vert_n + \vert y'' \vert \vert x z'\vert } \cdot   \\
 \!\!\!\!\!\!\!\!\!&\!\!\!&\!\!\!  \quad \digamma(y',z',x')^\ell_{t^{(5)}}  \digamma(y',z',x')^m_{t^{(3)}}  \digamma(y',z',x')^r_{t^{(1)}}   x''_{b \curlyvee  t^{(2)} }\, y''_{ c \curlyvee t^{(6)}   }\, z''_{d \curlyvee t^{(4)} } .
\end{eqnarray*}}
It follows that the left-hand side of \eqref{P2} is equal to
\begin{eqnarray*}
\!\!\!&\!\!\!\!\!\!&\!\!\!  -(-1)^{ \vert y z'' \vert_n \vert x \vert_n + \vert y'' \vert \vert x z'\vert} \cdot \\
\!\!\!&\!\!\!&\!\!\!  \quad \digamma(y',z',x')^\ell_{t^{(5)}}  \digamma(y',z',x')^m_{t^{(3)}}  \digamma(y',z',x')^r_{t^{(1)}}   x''_{b \curlyvee  t^{(2)} }\, y''_{ c \curlyvee t^{(6)}   }\, z''_{d \curlyvee t^{(4)} } \\
 \!\!\!&\!\!\!&\!\!\! -(-1)^{ \vert z x'' \vert_n \vert y \vert_n + \vert z'' \vert \vert y x'\vert } \cdot \\
\!\!\!&\!\!\!&\!\!\!  \quad \digamma(z',x',y')^\ell_{t^{(5)}}  \digamma(z',x',y')^m_{t^{(3)}}  \digamma(z',x',y')^r_{t^{(1)}}   y''_{c \curlyvee  t^{(2)} }\, z''_{d \curlyvee t^{(6)}   }\, x''_{b \curlyvee t^{(4)} } \\
\!\!\! &\!\!\!&\!\!\! -(-1)^{ \vert x y'' \vert_n \vert z \vert_n + \vert x'' \vert \vert z y'\vert   } \cdot \\
\!\!\!&\!\!\!&\!\!\!  \quad \digamma(x',y',z')^\ell_{t^{(5)}}  \digamma(x',y',z')^m_{t^{(3)}}  \digamma(x',y',z')^r_{t^{(1)}}   z''_{d \curlyvee  t^{(2)} }\, x''_{ b \curlyvee t^{(6)}   }\, y''_{c \curlyvee t^{(4)} } \\
\!\!\!&\!\!\! \by{cyclic_inv} \!\!\!&\!\!\! -(-1)^{ \vert y z'' \vert_n \vert x \vert_n + \vert y'' \vert \vert x z'\vert} \cdot \\
\!\!\!&\!\!\!&\!\!\!  \quad \digamma(y',z',x')^\ell_{t^{(5)}}  \digamma(y',z',x')^m_{t^{(3)}}  \digamma(y',z',x')^r_{t^{(1)}}   x''_{b \curlyvee  t^{(2)} }\, y''_{ c \curlyvee t^{(6)}   }\, z''_{d \curlyvee t^{(4)} } \\
\!\!\! &\!\!\!&\!\!\! -(-1)^{ \vert z x'' \vert_n \vert y \vert_n + \vert z'' \vert \vert y x'\vert } \cdot   \\
\!\!\!&\!\!\!&\!\!\!  \quad \digamma(z',x',y')^\ell_{t^{(3)}}  \digamma(z',x',y')^m_{t^{(1)}}  \digamma(z',x',y')^r_{t^{(5)}}   y''_{c \curlyvee  t^{(6)} }\, z''_{d \curlyvee t^{(4)}   }\, x''_{b \curlyvee t^{(2)} } \\
 \!\!\!&\!\!\!&\!\!\! -(-1)^{  \vert x y'' \vert_n \vert z \vert_n + \vert x'' \vert \vert z y'\vert   } \cdot  \\
\!\!\!&\!\!\!&\!\!\!  \quad \digamma(x',y',z')^\ell_{t^{(1)}}  \digamma(x',y',z')^m_{t^{(5)}}  \digamma(x',y',z')^r_{t^{(3)}}   z''_{d \curlyvee  t^{(4)} }\, x''_{ b \curlyvee t^{(2)}   }\, y''_{c \curlyvee t^{(6)} } \\
\!\!\!&\!\!\! = \!\!\! &\!\!\!  -(-1)^{\vert y z'' \vert_n \vert x \vert_n + \vert y'' \vert \vert x z'\vert} \cdot\\
\!\!\!&\!\!\!&\!\!\!  \quad \digamma(y',z',x')^\ell_{t^{(5)}}  \digamma(y',z',x')^m_{t^{(3)}}  \digamma(y',z',x')^r_{t^{(1)}}   x''_{b \curlyvee  t^{(2)} }\, y''_{ c \curlyvee t^{(6)}   }\, z''_{d \curlyvee t^{(4)} } \\
\!\!\! &\!\!\!&\!\!\! -(-1)^{ \vert z x'' \vert_n \vert y \vert_n + \vert z'' \vert \vert y x'\vert + \vert x'' \vert \vert y'' z'' \vert} \cdot \\
\!\!\!&\!\!\!&\!\!\!  \quad \digamma(z',x',y')^\ell_{t^{(3)}}  \digamma(z',x',y')^m_{t^{(1)}}  \digamma(z',x',y')^r_{t^{(5)}}   x''_{b \curlyvee t^{(2)} }\, y''_{c \curlyvee  t^{(6)} }\, z''_{ d \curlyvee t^{(4)}   }  \\
\!\!\! &\!\!\!&\!\!\! -(-1)^{ \vert x y'' \vert_n \vert z \vert_n + \vert x'' \vert \vert z y'\vert + \vert z'' \vert \vert x'' y'' \vert } \cdot  \\
\!\!\!&\!\!\!&\!\!\!  \quad \digamma(x',y',z')^\ell_{t^{(1)}}  \digamma(x',y',z')^m_{t^{(5)}}  \digamma(x',y',z')^r_{t^{(3)}}    x''_{ b \curlyvee t^{(2)}   }\, y''_{c \curlyvee t^{(6)} }\, z''_{d \curlyvee  t^{(4)} } \\
\!\!\!&\!\!\!=\!\!\!&\!\!\!  - (-1)^{\vert x'' \vert \vert y' z' \vert + \vert y'' \vert \vert z'\vert_n + \vert x \vert_n \vert z \vert_n} \cdot \\
\!\!\!&\!\!\!&\!\!\!  \triplep{x'}{y'}{z'}^\ell_{t^{(1)}} \triplep{x'}{y'}{z'}^m_{t^{(5)}}\, \triplep{x'}{y'}{z'}^r_{t^{(3)}}  x''_{b \curlyvee t^{(2)}}   y''_{ c \curlyvee t^{(6)}} z''_{d \curlyvee t^{(4)}}.
\end{eqnarray*}
 This  proves \eqref{P2} and concludes the proof of the lemma.
\end{proof}

\subsection{Remark}
 The formula \eqref{||-,-,-||} may be rewritten in the form
\begin{eqnarray}
\notag  &&   \bracket{x_b}{\bracket{y_c}{z_d}} + (-1)^{\vert x y \vert \vert z \vert_n  } \bracket{z_d}{\bracket{x_b}{y_c}}
+ (-1)^{  \vert x \vert_n \vert y z\vert } \bracket{y_c}{\bracket{z_d}{x_b}} \\
\notag &=&   - (-1)^{\vert x'' \vert \vert y' z' \vert + \vert y'' \vert \vert z'\vert_n } \cdot \\
\label{Jac} && \quad \triplep{x'}{y'}{z'}^\ell_{p^{(1)}} \triplep{x'}{y'}{z'}^m_{p^{(5)}}\, \triplep{x'}{y'}{z'}^r_{p^{(3)}}
x''_{b \curlyvee p^{(2)}}   y''_{ c \curlyvee p^{(6)}} z''_{d \curlyvee p^{(4)}} .
\end{eqnarray}
 Though we will not need it, we mention an  equivalent version  of  \eqref{Jac}.
Assume the conditions of   Theorem \ref{MAINMAIN}     and set $v=s_B(t)$.
It can be proved that,  for any  homogeneous $x,y,z\in A$  and   any $b,c,d\in B$, we have
\begin{eqnarray*}
 && \triplep{x'}{y'}{z'}^\ell_{v^{(1)}} \triplep{x'}{y'}{z'}^m_{v^{(5)}}\,
 \triplep{x'}{y'}{z'}^r_{v^{(3)}}  x''_{b \curlyvee v^{(2)}}   y''_{ c \curlyvee v^{(6)}} z''_{d \curlyvee v^{(4)}}\\
 &=& -(-1)^{\vert x' \vert_n \vert y' \vert_n}  \triplep{y'}{x'}{z'}^\ell_{t^{(3)}} \triplep{y'}{x'}{z'}^m_{t^{(1)}} \triplep{y'}{x'}{z'}^r_{t^{(5)}}
  x''_{b \curlyvee t^{(2)}} y''_{c \curlyvee t^{(4)}} z''_{d \curlyvee t^{(6)}} .
\end{eqnarray*}
Therefore \eqref{Jac} is equivalent to the following identity:
\begin{eqnarray*}
 &&   \bracket{x_b}{\bracket{y_c}{z_d}} + (-1)^{\vert x y \vert \vert z \vert_n  } \bracket{z_d}{\bracket{x_b}{y_c}} + (-1)^{  \vert x \vert_n \vert y z\vert } \bracket{y_c}{\bracket{z_d}{x_b}} \\
&=& (-1)^{\vert x \vert_n \vert y' \vert_n + \vert x'' y'' \vert  \vert z' \vert_n} \cdot \\
&& \triplep{y'}{x'}{z'}^\ell_{t^{(3)}} \triplep{y'}{x'}{z'}^m_{t^{(1)}} \triplep{y'}{x'}{z'}^r_{t^{(5)}}  x''_{b \curlyvee t^{(2)}} y''_{c \curlyvee t^{(4)}} z''_{d \curlyvee t^{(6)}}  \\
 &&  - (-1)^{\vert x'' \vert \vert y' z' \vert + \vert y'' \vert \vert z'\vert_n} \cdot \\
 && \quad \triplep{x'}{y'}{z'}^\ell_{t^{(1)}} \triplep{x'}{y'}{z'}^m_{t^{(5)}}\, \triplep{x'}{y'}{z'}^r_{t^{(3)}}  x''_{b \curlyvee t^{(2)}}   y''_{ c \curlyvee t^{(6)}} z''_{d \curlyvee t^{(4)}}
\end{eqnarray*}

\section{Quasi-Poisson brackets  in representation algebras}\label{The quasi-Poisson case}

 Quasi-Poisson   structures  on manifolds  were introduced by Alekseev,  Kosmann-Schwarz\-bach, and Meinrenken    \cite{AKsM};   see also \cite{VdB}.
 We adapt their definition to  an algebraic     set-up
 and  establish   a  version of Theorem~\ref{MAINMAIN}  producing   quasi-Poisson brackets.

\subsection{Quasi-Poisson   brackets} \label{qP_algebras}

  Let   $B$ be a commutative ungraded Hopf algebra  endowed with a trace-like element $t\in B$.
Let  $M$ be an ungraded algebra with a $B$-coaction $ \Delta_M: M
\to M \otimes B$. A    bilinear map   $\bracket{-}{-}:M \times M \to M$ is a   \emph{quasi-Poisson bracket} with respect  to $t$  if   it is
 $B$-equivariant, antisymmetric, and   if it  satisfies  the   {Leibniz rules} and the following \emph{quasi-Jacobi identity}:    for any $u,v,w\in M$,
\begin{eqnarray}
\label{qP}   \!\!\!\! \!\!\!\!  \!\!\!\!    && \bracket{u}{\bracket{v}{w}} + \bracket{w}{\bracket{u}{v}} + \bracket{v}{\bracket{w}{u}} \\
\notag \!\!\!\! \!\!\!\! \!\!\!\!   & =&  ( u^r \bullet t') (v^r\bullet
t'') (w^r\bullet t''') \, u^\ell v^\ell  w^\ell - ( u^r \bullet t') (v^r\bullet
t''') (w^r\bullet t'') \, u^\ell v^\ell  w^\ell
\end{eqnarray}
where $\bullet=\bullet_t:B \times B \to \kk$ is the  balanced biderivation  associated with $t$ and   we  use Sweedler's notation   $\Delta_M(m)=m^\ell \otimes m^r$ for  $m \in M$.

Observe that a    quasi-Poisson bracket in $M$ restricts to a Poisson bracket on the  subalgebra $M^{\operatorname{inv}} $ of $M$ consisting of $B$-invariant elements:
$$
M^{\operatorname{inv}} = \{ m \in M : \Delta_M(m) = m \otimes 1_B \}.
$$

\subsection{The quasi-Jacobi identity in representation algebras}

 A Fox pairing  (of degree~0) in
an ungraded Hopf algebra $A$ is \emph{quasi-Poisson} if it is antisymmetric and the induced
  tritensor map     (defined in Section~\ref{triple}) satisfies
\begin{eqnarray}
\label{qP_triple}  \!\!\!\! \!\!\!\!    \triplep{x}{y}{z} &=&
1_A \otimes y \otimes  xz + yx \otimes 1_A \otimes  z  +  x \otimes zy\otimes  1_A +   y \otimes z \otimes x\\
 \!\!\!\! \!\!\!\!   \nonumber && - 1_A \otimes  zy \otimes  x  -  y \otimes 1_A  \otimes  xz -   yx  \otimes z \otimes  1_A- x \otimes y\otimes  z   
\end{eqnarray}
  for any $x,y,z\in A$. A geometric example of a quasi-Poisson Fox pairing will be given in Section~\ref{surfaces}.

\begin{theor} \label{qPoisson_case}
Let
  $A$ be  a cocommutative ungraded Hopf algebra
carrying a quasi-Poisson Fox pairing $\rho $. Let   $B$ be a commutative ungraded Hopf algebra   endowed with a trace-like element $t $.
Then the bracket $\bracket{-}{-}$ in $A_B$  produced by    Theorem \ref{double_to_simple} from $\rho$ and
$\bullet=\bullet_t$ is quasi-Poisson with respect to  $t$.
\end{theor}

\begin{proof}
 The  bilinear map $\bracket{-}{-}:A_B \times A_B \to A_B $    is antisymmetric and
satisfies the  Leibniz rules by Theorem
\ref{double_to_simple}. This bracket is $B$-equivariant by Theorem \ref{equivariant_bracket}.  It  remains to verify the
  identity \eqref{qP}.   It is easily seen that both
sides of    \eqref{qP} define trilinear maps $A_B \times A_B \times A_B \to A_B$ which are derivations in each variable.  Therefore it
is enough to check \eqref{qP} on the generators of  the   algebra $A_B$. We need to prove that, for any $x,y,z\in A$ and $b,c,d\in B$,
\begin{eqnarray}
\label{qP_gen} && \bracket{x_b}{\bracket{y_c}{z_d}} + \bracket{z_d}{\bracket{x_b}{y_c}} + \bracket{y_c}{\bracket{z_d}{x_b}} \\
\notag &=& V(x,y,z;b,c,d;t  )- V(x,z,y;b,d,c ;t )
\end{eqnarray}
where
  $$V(x,y,z;b,c,d  ;t  )= ((x_b)^r\bullet t') ((y_c)^r\bullet t'') ((z_d)^r\bullet t''')\, (x_b)^\ell(y_c)^\ell(z_d)^\ell. $$
  Here and below we use Sweedler's notation   for the  $B$-coaction on $A_B$.

In the sequel, we denote the comultiplication, the counit and the
antipode  in $A$ and $B$ by the same letters $\Delta$,
$\varepsilon$ and $s$, respectively. Lemma \ref{Jacobitor}  gives
\begin{eqnarray}
\notag && \bracket{x_b}{\bracket{y_c}{z_d}} + \bracket{z_d}{\bracket{x_b}{y_c}} + \bracket{y_c}{\bracket{z_d}{x_b}}\\
\label{U_U} &=& -U(x,y,z;b,c,d;t) - U(x,y,z;b,c,d;s(t))
 \end{eqnarray}
where, for any cosymmetric  $u\in B$,
$$ U(x,y,z;b,c,d;u)
=   \triplep{x'}{y'}{z'}^\ell_{u^{(1)}} \triplep{x'}{y'}{z'}^m_{u^{(5)}}\, \triplep{x'}{y'}{z'}^r_{u^{(3)}}
x''_{b \curlyvee u^{(2)}}   y''_{ c \curlyvee u^{(6)}} z''_{d \curlyvee u^{(4)}}.
$$
We compute  the latter expression  using
\eqref{qP_triple}:
\begin{eqnarray*}
  \!\!\!\! \!\!\!\!  \!\!\!\!   && \!\!\!\!  \!\!\!\!   U(x,y,z;b,c,d;u) \\
 \!\!\!\!  \!\!\!\! \!\!\!\!  &=& \!\!\!\!  \!\!\!\!  \varepsilon(u^{(1)}) y'_{u^{(5)}}  (x'z')_{u^{(3)}} x''_{b \curlyvee u^{(2)}}   y''_{ c \curlyvee u^{(6)}} z''_{d \curlyvee u^{(4)}} \\
 \!\!\!\!  \!\!\!\!  \!\!\!\!  && \!\!\!\!  \!\!\!\! + (y'x')_{u^{(1)}}  \varepsilon(u^{(5)}) z'_{u^{(3)}}  x''_{b \curlyvee u^{(2)}}   y''_{ c \curlyvee u^{(6)}} z''_{d \curlyvee u^{(4)}} \\
\!\!\!\!  \!\!\!\!  \!\!\!\!   && \!\!\!\!  \!\!\!\!  +  x'_{u^{(1)}}   (z'y')_{u^{(5)}}  \varepsilon(u^{(3)})  x''_{b \curlyvee u^{(2)}}   y''_{ c \curlyvee u^{(6)}} z''_{d \curlyvee u^{(4)}} \\
\!\!\!\!  \!\!\!\!  \!\!\!\!  && \!\!\!\!  \!\!\!\!  +   y'_{u^{(1)}}   z'_{u^{(5)}}   x'_{u^{(3)}} x''_{b \curlyvee u^{(2)}}   y''_{ c \curlyvee u^{(6)}} z''_{d \curlyvee u^{(4)}}  \\
\!\!\!\!  \!\!\!\!  \!\!\!\!   && \!\!\!\!  \!\!\!\!  -  \varepsilon(u^{(1)})  (z'y')_{u^{(5)}}  x'_{u^{(3)}} x''_{b \curlyvee u^{(2)}}   y''_{ c \curlyvee u^{(6)}} z''_{d \curlyvee u^{(4)}}   \\
\!\!\!\!  \!\!\!\!  \!\!\!\!   && \!\!\!\!  \!\!\!\!  -  y'_{u^{(1)}}  \varepsilon(u^{(5)})   (x'z')_{u^{(3)}}  x''_{b \curlyvee u^{(2)}}   y''_{ c \curlyvee u^{(6)}} z''_{d \curlyvee u^{(4)}}  \\
\!\!\!\!  \!\!\!\! \!\!\!\!   && \!\!\!\!  \!\!\!\! -   (y'x')_{u^{(1)}}   z'_{u^{(5)}}  \varepsilon(u^{(3)}) x''_{b \curlyvee u^{(2)}}   y''_{ c \curlyvee u^{(6)}} z''_{d \curlyvee u^{(4)}} \\
\!\!\!\!  \!\!\!\! \!\!\!\!   && \!\!\!\!  \!\!\!\!  - x'_{u^{(1)}}  y'_{u^{(5)}}  z'_{u^{(3)}} x''_{b \curlyvee u^{(2)}}   y''_{ c \curlyvee u^{(6)}} z''_{d \curlyvee u^{(4)}}  \\
 \!\!\!\!  \!\!\!\! \!\!\!\!   &=& \!\!\!\!  \!\!\!\!  y'_{u^{(4)}}  (x'z')_{u^{(2)}} x''_{b \curlyvee u^{(1)}}   y''_{ c \curlyvee u^{(5)}} z''_{d \curlyvee u^{(3)}} + (y'x')_{u^{(1)}}   z'_{u^{(3)}}  x''_{b \curlyvee u^{(2)}}   y''_{ c \curlyvee u^{(5)}} z''_{d \curlyvee u^{(4)}} \\
 \!\!\!\!  \!\!\!\! \!\!\!\!   && \!\!\!\!  \!\!\!\!  +  x'_{u^{(1)}}   (z'y')_{u^{(4)}}    x''_{b \curlyvee u^{(2)}}   y''_{ c \curlyvee u^{(5)}} z''_{d \curlyvee u^{(3)}}  +   y'_{u^{(1)}}   z'_{u^{(5)}}   x'_{u^{(3)}} x''_{b \curlyvee u^{(2)}}   y''_{ c \curlyvee u^{(6)}} z''_{d \curlyvee u^{(4)}}  \\
\!\!\!\!  \!\!\!\!  \!\!\!\!  && \!\!\!\!  \!\!\!\!  -   (z'y')_{u^{(4)}}  x'_{u^{(2)}} x''_{b \curlyvee u^{(1)}}   y''_{ c \curlyvee u^{(5)}} z''_{d \curlyvee u^{(3)}}   -  y'_{u^{(1)}}     (x'z')_{u^{(3)}}  x''_{b \curlyvee u^{(2)}}   y''_{ c \curlyvee u^{(5)}} z''_{d \curlyvee u^{(4)}}  \\
\!\!\!\!  \!\!\!\!  \!\!\!\!   && \!\!\!\!  \!\!\!\! -   (y'x')_{u^{(1)}}   z'_{u^{(4)}} x''_{b \curlyvee u^{(2)}}   y''_{ c \curlyvee u^{(5)}} z''_{d \curlyvee u^{(3)}} - x'_{u^{(1)}}  y'_{u^{(5)}}  z'_{u^{(3)}} x''_{b \curlyvee u^{(2)}}   y''_{ c \curlyvee u^{(6)}} z''_{d \curlyvee u^{(4)}}\\
 \!\!\!\!  \!\!\!\! \!\!\!\!   &=& \!\!\!\!  \!\!\!\!   y'_{u^{(5)}}  x'_{u^{(2)}} z'_{u^{(3)}} x''_{b \curlyvee u^{(1)}}   y''_{ c \curlyvee u^{(6)}} z''_{d \curlyvee u^{(4)}} + y'_{u^{(1)}} x'_{u^{(2)}}   z'_{u^{(4)}}  x''_{b \curlyvee u^{(3)}}   y''_{ c \curlyvee u^{(6)}} z''_{d \curlyvee u^{(5)}} \\
 \!\!\!\!  \!\!\!\! \!\!\!\!   && \!\!\!\!  \!\!\!\!   +  x'_{u^{(1)}}   z'_{u^{(4)}} y'_{u^{(5)}}    x''_{b \curlyvee u^{(2)}}   y''_{ c \curlyvee u^{(6)}} z''_{d \curlyvee u^{(3)}}  +   y'_{u^{(1)}}   z'_{u^{(5)}}   x'_{u^{(3)}} x''_{b \curlyvee u^{(2)}}   y''_{ c \curlyvee u^{(6)}} z''_{d \curlyvee u^{(4)}}  \\
\!\!\!\!  \!\!\!\! \!\!\!\!   && \!\!\!\!  \!\!\!\!  -   z'_{u^{(4)}} y'_{u^{(5)}}  x'_{u^{(2)}} x''_{b \curlyvee u^{(1)}}   y''_{ c \curlyvee u^{(6)}} z''_{d \curlyvee u^{(3)}}   -  y'_{u^{(1)}}     x'_{u^{(3)}} z'_{u^{(4)}}  x''_{b \curlyvee u^{(2)}}   y''_{ c \curlyvee u^{(6)}} z''_{d \curlyvee u^{(5)}}  \\
\!\!\!\!  \!\!\!\! \!\!\!\!   && \!\!\!\!  \!\!\!\!  -   y'_{u^{(1)}} x'_{u^{(2)}}   z'_{u^{(5)}} x''_{b \curlyvee u^{(3)}}   y''_{ c \curlyvee u^{(6)}} z''_{d \curlyvee u^{(4)}} - x'_{u^{(1)}}  y'_{u^{(5)}}  z'_{u^{(3)}} x''_{b \curlyvee u^{(2)}}   y''_{ c \curlyvee u^{(6)}} z''_{d \curlyvee u^{(4)}}\\
 \!\!\!\!  \!\!\!\!  \!\!\!\!  &=& \!\!\!\!  \!\!\!\!  x_{(b \curlyvee u^{(1)})u^{(2)}}   y_{u^{(5)}(c \curlyvee u^{(6)})}  z_{u^{(3)}(d \curlyvee u^{(4)})}   + x_{u^{(2)}( b \curlyvee u^{(3)})}  y_{(c \curlyvee u^{(6)})u^{(1)} }     z_{u^{(4)}(d \curlyvee u^{(5)})}   \\
 \!\!\!\!  \!\!\!\!  \!\!\!\!   && \!\!\!\!  \!\!\!\!   +  x_{u^{(1)}(b \curlyvee u^{(2)})}  y_{u^{(5)}( c \curlyvee u^{(6)})}  z_{(d \curlyvee u^{(3)})u^{(4)}}      +   x_{(b \curlyvee u^{(2)})u^{(3)}}  y_{(c \curlyvee u^{(6)}) u^{(1)}}   z_{(d \curlyvee u^{(4)}) u^{(5)}}    \\
\!\!\!\!  \!\!\!\! \!\!\!\!   && \!\!\!\!  \!\!\!\!  -   x_{(b \curlyvee u^{(1)}) u^{(2)}}  y_{u^{(5)}(c \curlyvee u^{(6)})}    z_{(d \curlyvee u^{(3)})u^{(4)}}   -   x_{(b \curlyvee u^{(2)}) u^{(3)}}  y_{(c \curlyvee u^{(6)}) u^{(1)}}    z_{u^{(4)}(d \curlyvee u^{(5)})}    \\
\!\!\!\!  \!\!\!\!  \!\!\!\!   && \!\!\!\!  \!\!\!\!  -  x_{u^{(2)}(b \curlyvee u^{(3)})}   y_{( c \curlyvee u^{(6)}) u^{(1)}}   z_{(d \curlyvee u^{(4)})u^{(5)}}   - x_{u^{(1)}(b \curlyvee u^{(2)}) }  y_{u^{(5)}(c \curlyvee u^{(6)})}  z_{u^{(3)}(d \curlyvee u^{(4)})} \\
\!\!\!\!  \!\!\!\!  \!\!\!\!   &\by{cyclic_inv}&   \!\!\!\!  \!\!\!\!  x_{(b \curlyvee u^{(1)})u^{(2)}}  y_{u^{(5)}(c \curlyvee u^{(6)})} z_{u^{(3)}(d \curlyvee u^{(4)})}   +  x_{u^{(3)}( b \curlyvee u^{(4)})}    y_{(c \curlyvee u^{(1)})u^{(2)} }  z_{u^{(5)}(d \curlyvee u^{(6)})}   \\
\!\!\!\!  \!\!\!\!  \!\!\!\!   && \!\!\!\!  \!\!\!\!   +  x_{u^{(1)}(b \curlyvee u^{(2)})}   y_{u^{(5)}( c \curlyvee u^{(6)})}    z_{(d \curlyvee u^{(3)})u^{(4)}}  +   x_{(b \curlyvee u^{(3)})u^{(4)}} y_{(c \curlyvee u^{(1)}) u^{(2)}}   z_{(d \curlyvee u^{(5)}) u^{(6)}}     \\
\!\!\!\!  \!\!\!\! \!\!\!\!   && \!\!\!\!  \!\!\!\!  -     x_{(b \curlyvee u^{(1)}) u^{(2)}} y_{u^{(5)}(c \curlyvee u^{(6)})}   z_{(d \curlyvee u^{(3)})u^{(4)}} -   x_{(b \curlyvee u^{(3)}) u^{(4)}} y_{(c \curlyvee u^{(1)}) u^{(2)}}     z_{u^{(5)}(d \curlyvee u^{(6)})}    \\
\!\!\!\!  \!\!\!\!  \!\!\!\!   && \!\!\!\!  \!\!\!\!  -    x_{u^{(3)}(b \curlyvee u^{(4)})} y_{( c \curlyvee u^{(1)}) u^{(2)}}  z_{(d \curlyvee u^{(5)})u^{(6)}}   - x_{u^{(1)}(b \curlyvee u^{(2)}) }  y_{u^{(5)}(c \curlyvee u^{(6)})}  z_{u^{(3)}(d \curlyvee u^{(4)})} \\
\!\!\!\!  \!\!\!\!  \!\!\!\! &\stackrel{\eqref{ltimes}, \eqref{ltimes_bis}}{=}  & (b'' \bullet u') (c' \bullet u''') (d' \bullet u'')\,  x_{b'}  y_{c''} z_{d''}   +  (b' \bullet u'') (c'' \bullet u') (d' \bullet u''')\,  x_{b''}  y_{c'} z_{d''}  \\
 \!\!\!\!  \!\!\!\! \!\!\!\!  \!\!\!\!  \!\!\!\! && \!\!\!\!  \!\!\!\!  +  (b' \bullet u') (c' \bullet u''') (d'' \bullet u'')\,  x_{b''}  y_{c''} z_{d'}     +   (b'' \bullet u'') (c'' \bullet u') (d'' \bullet u''')\,  x_{b'}  y_{c'} z_{d'}       \\
\!\!\!\!  \!\!\!\!  \!\!\!\!  \!\!\!\!  \!\!\!\!  && \!\!\!\!  \!\!\!\!  -     (b'' \bullet u') (c' \bullet u''') (d'' \bullet u'')\,  x_{b'}  y_{c''} z_{d'}   -   (b'' \bullet u'') (c'' \bullet u') (d' \bullet u''')\,  x_{b'}  y_{c'} z_{d''}     \\
\!\!\!\!  \!\!\!\!  \!\!\!\!  \!\!\!\!  \!\!\!\! && \!\!\!\!  \!\!\!\!   -  (b' \bullet u'') (c'' \bullet u') (d'' \bullet u''')\,  x_{b''}  y_{c'} z_{d'}     - (b' \bullet u') (c' \bullet u''') (d' \bullet u'')\,  x_{b''}  y_{c''} z_{d''}.
\end{eqnarray*}
In particular, we have
\begin{eqnarray*}
 \!\!\!\!    && \!\!\!\!    U(x,y,z;b,c,d;s(t)) \\
\!\!\!\!    &=& \!\!\!\!    (b'' \bullet s(t''')) (c' \bullet s(t')) (d' \bullet s(t''))\,  x_{b'}  y_{c''} z_{d''}   \\
\!\!\!\!    && \!\!\!\!    +  (b' \bullet s(t'')) (c'' \bullet s(t''')) (d' \bullet s(t'))\,  x_{b''}  y_{c'} z_{d''}  \\
\!\!\!\!     && \!\!\!\!     +  (b' \bullet  s(t''')) (c' \bullet  s(t')) (d'' \bullet  s(t''))\,  x_{b''}  y_{c''} z_{d'}    \\
 \!\!\!\!    && \!\!\!\!    +   (b'' \bullet  s(t'')) (c'' \bullet  s(t''')) (d'' \bullet  s(t'))\,  x_{b'}  y_{c'} z_{d'}       \\
\!\!\!\!    && \!\!\!\!    -     (b'' \bullet  s(t''')) (c' \bullet  s(t')) (d'' \bullet  s(t''))\,  x_{b'}  y_{c''} z_{d'}   \\
\!\!\!\!    && \!\!\!\!    -   (b'' \bullet  s(t'')) (c'' \bullet  s(t''')) (d' \bullet  s(t'))\,  x_{b'}  y_{c'} z_{d''}     \\
\!\!\!\!    && \!\!\!\!     -  (b' \bullet  s(t'')) (c'' \bullet  s(t''')) (d'' \bullet  s(t'))\,  x_{b''}  y_{c'} z_{d'} \\
\!\!\!\!    && \!\!\!\!       - (b' \bullet  s(t''')) (c' \bullet  s(t')) (d' \bullet  s(t''))\,  x_{b''}  y_{c''} z_{d''} \\
\!\!\!\!    &=& \!\!\!\!    -(b'' \bullet t''') (c' \bullet t') (d' \bullet t'')\,  x_{b'}  y_{c''} z_{d''}   -  (b' \bullet t'') (c'' \bullet t''') (d' \bullet t')\,  x_{b''}  y_{c'} z_{d''}  \\
 \!\!\!\!    && \!\!\!\!    -  (b' \bullet  t''') (c' \bullet  t') (d'' \bullet  t'')\,  x_{b''}  y_{c''} z_{d'}     -  (b'' \bullet  t'') (c'' \bullet  t''') (d'' \bullet  t')\,  x_{b'}  y_{c'} z_{d'}       \\
\!\!\!\!    && \!\!\!\!     +   (b'' \bullet  t''') (c' \bullet  t') (d'' \bullet  t'')\,  x_{b'}  y_{c''} z_{d'}   +   (b'' \bullet  t'') (c'' \bullet  t''') (d' \bullet  t')\,  x_{b'}  y_{c'} z_{d''}     \\
\!\!\!\!    && \!\!\!\!     +  (b' \bullet  t'') (c'' \bullet  t''') (d'' \bullet  t')\,  x_{b''}  y_{c'} z_{d'}     +  (b' \bullet  t''') (c' \bullet  t') (d' \bullet  t'')\,  x_{b''}  y_{c''} z_{d''} \\
\!\!\!\!    & \by{cyclic_inv} & \!\!\!\!     -U(x,z,y;b,d,c;t)
\end{eqnarray*}
and it follows from \eqref{U_U} that
\begin{eqnarray*}
&&  \bracket{x_b}{\bracket{y_c}{z_d}} + \bracket{z_d}{\bracket{x_b}{y_c}} + \bracket{y_c}{\bracket{z_d}{x_b}} \\
 & = & -U(x,y,z;b,c,d;t) + U(x,z,y;b,d,c;t).
 \end{eqnarray*}

 Formula \eqref{qP_gen} will follow from the equality \begin{equation}\label{vuvu} V(x,z,y;b,d,c;t)= U(x,y,z;b,c,d;t), \end{equation} which  we now prove.
Computing the   coaction    on  $x_b$ by Lemma \ref{coaction_M}, we obtain
\begin{eqnarray*}
V(x,y,z;b,c,d ; t )  &=&     ((s(b') b''')\bullet t')\  ((y_c)^r\bullet t'') ((z_d)^r\bullet t''')\, x_{b''} (y_c)^\ell(z_d)^\ell \\
&=& - \varepsilon(b''') (b' \bullet t')  ((y_c)^r\bullet t'') ((z_d)^r\bullet t''') \, x_{b''} (y_c)^\ell(z_d)^\ell \\
&& + \varepsilon (b') (b'''\bullet t')   ((y_c)^r\bullet t'') ((z_d)^r\bullet t''')\, x_{b''} (y_c)^\ell(z_d)^\ell \\
&=& - (b' \bullet t')   ((y_c)^r\bullet t'') ((z_d)^r\bullet t''') \, x_{b''} (y_c)^\ell(z_d)^\ell \\
&& +  (b''\bullet t')   ((y_c)^r\bullet t'') ((z_d)^r\bullet t''')\, x_{b'} (y_c)^\ell(z_d)^\ell
\end{eqnarray*}
 Further computing the coaction   on  $y_c$ by Lemma \ref{coaction_M}, we obtain
\begin{eqnarray*}
 V(x,y,z;b,c,d ; t  )   &=& - (b' \bullet t')   (c'\bullet t'') ((z_d)^r\bullet t''')  x_{b''} y_{c''} (z_d)^\ell  \\
&& + (b' \bullet t')   (c''\bullet t'') ((z_d)^r\bullet t''')  x_{b''} y_{c'} (z_d)^\ell  \\
&&  +  (b''\bullet t')   (c'\bullet t'') ((z_d)^r\bullet t''')\, x_{b'} y_{c''}(z_d)^\ell  \\
&& -    (b''\bullet t')   (c''\bullet t'') ((z_d)^r\bullet t''')\, x_{b'} y_{c'}(z_d)^\ell.
\end{eqnarray*}
Finally,  applying Lemma \ref{coaction_M} to $z_d$, we obtain
\begin{eqnarray*}
&& V(x,y,z;b,c,d ; t )   \\
&=& - (b' \bullet t')   (c'\bullet t'') (d'\bullet t''')\,  x_{b''} y_{c''} z_{d''}  + (b' \bullet t')   (c'\bullet t'') (d''\bullet t''')\,  x_{b''} y_{c''} z_{d'} \\
&&  + (b' \bullet t')   (c''\bullet t'') (d'\bullet t''')\,  x_{b''} y_{c'} z_{d''}  - (b' \bullet t')   (c''\bullet t'') (d''\bullet t''')\,  x_{b''} y_{c'} z_{d'} \\
&& +  (b''\bullet t')   (c'\bullet t'') (d'\bullet t''')\, x_{b'} y_{c''} z_{d''}  -  (b''\bullet t')   (c'\bullet t'') (d''\bullet t''')\, x_{b'} y_{c''}z_{d'} \\
&& -    (b''\bullet t')   (c''\bullet t'') (d'\bullet t''')\, x_{b'} y_{c'} z_{d''}   + (b''\bullet t')   (c''\bullet t'') (d''\bullet t''')\, x_{b'} y_{c'}z_{d'}.
\end{eqnarray*}
 The equality \eqref{vuvu} follows. This concludes the proof of the  theorem.
\end{proof}

\subsection{Remark}

The definition of a quasi-Poisson bracket  in Section \ref{qP_algebras}  involves  a trace-like element in a  commutative ungraded Hopf algebra $B$.
One can give a more general definition  depending only on the choice of  a balanced biderivation in $B$
whose associated symmetric bilinear form in $I/I^2$ (where $I=\ker \varepsilon_B$) is nonsingular.
 A~quasi-Poisson algebra in this general sense is  also a  ``quasi-Poisson algebra over a Lie pair'' in the sense of \cite{MT_dim_2}.
We do not study this general definition here.

\section{Computations on invariant elements} \label{invariant_subalgebra}

  We   discuss  algebraic operations associated with a Fox pairing and use them to  compute  our brackets
  on certain invariant elements of the representation algebras.

\subsection{Operations derived from a Fox pairing} \label{further_operations}

Let $A$ be  a cocommutative graded Hopf algebra.
We   define     modules  $\check{A}$ and  $A^r\! \otimes_A\! {}^\ell\! A$    as follows.
Set  $\check{A} =A / [A,A]$,
where $[A,A]$  is   the submodule  of $A$   generated by the commutators $yz-(-1)^{\vert y \vert \vert z \vert}zy$ for any homogeneous $y,z \in A$.
The class   of an $x \in A$ in $\check{A}$   is  denoted by $\check x$.
 Next define the \emph{left adjoint action} $\operatorname{ad}^\ell: A \times A \to A$
  and the \emph{right adjoint action} $\operatorname{ad}^r: A \times A \to A$   by
$$
\operatorname{ad}^\ell(y, z) = (-1)^{\vert z \vert \vert y''\vert} y'\, z\, s_A(y'') \quad \hbox{and} \quad
\operatorname{ad}^r(z,y) = (-1)^{\vert z \vert \vert y'\vert} s_A(y')\, z\, y''
$$
 for any homogeneous $y,z\in A$.
These actions    yield   left and right $A$-module structures on~$A$; the  resulting   left and right $A$-modules are denoted by ${}^\ell\! A$ and $A^r$ respectively.
The  tensor product $A^r\! \otimes_A\! {}^\ell\! A$ over $A$ is the   module  mentioned above.  Note the following  linear maps:
\begin{eqnarray}
\label{multi} A^r\! \otimes_A\! {}^\ell\! A \longrightarrow \check{A}, && x\otimes y \longmapsto xy\!\!\!\mod [A,A],\\
\label{multi^s} A^r\! \otimes_A\! {}^\ell\! A \longrightarrow \check{A}, && x\otimes y \longmapsto x\, s_A(y)\!\!\!\mod [A,A],\\
\label{proj}  A^r\! \otimes_A\! {}^\ell\! A \longrightarrow \check{A} \otimes \check{A} &&  x\otimes y \longmapsto \check x  \otimes \check y.
\end{eqnarray}

\begin{lemma} \label{Theta}
Let  $\rho $ be a Fox pairing    of degree $n\in \ZZ$ in~$A$.
Then there is  a bilinear map
$$
\Theta= \Theta_\rho: \check A \times \check A \longrightarrow A^r\! \otimes_A\! {}^\ell\! A
$$
 such that for any $x,y \in A$,
$$
\Theta (\check x , \check y ) =  \operatorname{ad}^r(x',\rho(x'',y'))  \otimes y''
=  x' \otimes \operatorname{ad}^\ell(\rho(x'',y') , y'').
$$
\end{lemma}

\begin{proof}
For any  $x,y \in A$, we set
$$
\Theta(  x ,  y ) =  \operatorname{ad}^r(x',\rho(x'',y'))  \otimes y'' =
  x' \otimes \operatorname{ad}^\ell(\rho(x'',y') , y'') \in A^r\! \otimes_A\! {}^\ell\! A.
$$
We need to check that
\begin{equation} \label{left--}
\Theta( x_1 x_2, y ) = (-1)^{\vert x_1 \vert \vert x_2 \vert }\, \Theta( x_2 x_1, y)
\end{equation}
for any homogeneous $x_1,x_2,y \in A$, and  that
\begin{equation} \label{right--}
\Theta( x,y_1 y_2) = (-1)^{\vert y_1 \vert \vert y_2 \vert }\, \Theta( x, y_2 y_1)
\end{equation}
for any homogeneous $x,y_1,y_2\in A$.
We have
\begin{eqnarray*}
\Theta( x_1 x_2, y )
&=&  (-1)^{ \vert x_1'' \vert \vert x'_2 \vert }
\operatorname{ad}^r(x'_1 x_2',\rho(x''_1 x''_2,y'))  \otimes y''  \\
&=&  (-1)^{ \vert x_1'' \vert \vert x'_2 \vert }
\operatorname{ad}^r(x'_1 x_2',x''_1 \rho(x''_2,y'))  \otimes y''\\
 && + (-1)^{ \vert x_1'' \vert \vert x'_2 \vert }  \varepsilon(x''_2)\,
\operatorname{ad}^r(x'_1 x_2',\rho(x''_1 ,y'))  \otimes y''\\
&=&  (-1)^{\vert x_1'' \vert \vert x'_2 \vert }
\operatorname{ad}^r\big(\operatorname{ad}^r(x'_1 x_2',x''_1), \rho(x''_2,y')\big)  \otimes y''\\
 && +    (-1)^{\vert x_1'' \vert \vert x_2 \vert }    \operatorname{ad}^r(x'_1 x_2,\rho(x''_1 ,y'))  \otimes y''\\
 &=&  (-1)^{\vert  x_1'' x_1'''\vert \vert x'_2\vert+ \vert x_1'' \vert \vert x_1'x_2'\vert }
\operatorname{ad}^r\big( s(x_1'') x_1' x_2' x_1''' , \rho(x''_2,y')\big)  \otimes y''\\
 && +  (-1)^{\vert x_1'' \vert \vert x_2 \vert }    \operatorname{ad}^r(x'_1 x_2,\rho(x''_1 ,y'))  \otimes y''\\
  &=&  (-1)^{\vert  x_1'''\vert \vert x'_2\vert}
\operatorname{ad}^r\big( s(x_1') x_1'' x_2' x_1''' , \rho(x''_2,y')\big)  \otimes y'' \\
&&+    (-1)^{\vert x_1'' \vert \vert x_2 \vert }     \operatorname{ad}^r(x'_1 x_2,\rho(x''_1 ,y'))  \otimes y''\\
 &=&  (-1)^{\vert  x_1\vert \vert x'_2\vert} \operatorname{ad}^r\big( x_2' x_1 , \rho(x''_2,y')\big)  \otimes y'' \\
 && +   (-1)^{\vert x_1'' \vert \vert x_2 \vert }     \operatorname{ad}^r(x'_1 x_2,\rho(x''_1 ,y'))  \otimes y''
\end{eqnarray*}
which immediately implies \eqref{left--}. The identity \eqref{right--} is  verified similarly.
\end{proof}

The  compositions of $ \Theta $ with the maps  \eqref{multi}, \eqref{multi^s}, \eqref{proj} are denoted, respectively, by
\begin{eqnarray*}
\langle - ,- \rangle&=&\langle - ,- \rangle_\rho: \check{A} \times \check{A}  \longrightarrow \check{A},\\
  \langle - ,- \rangle\!^s& =& \langle - ,- \rangle\!^s_\rho: \check{A} \times \check{A}  \longrightarrow \check{A},\\
\vert-,-\vert&=& \vert-,-\vert_\rho:  \check{A} \times \check{A}  \longrightarrow  \check{A} \otimes \check{A}
\end{eqnarray*}
  Explicitly, we have for any $x,y \in A$,
\begin{eqnarray}
\label{<-,->}  \!\!\!\! \!\!\!\!  \!\!\!\!   \langle \check x ,\check y \rangle \!\!\!\! &=&  \!\!\!\!  (  (-1)^{\vert x' \vert \vert \rho(x'',y')'\vert }\, s_A(\rho(x'',y')')x'\rho(x'',y')'' y''  \!\!\mod [A,A], \\
\label{<-,->^s}  \!\!\!\!  \!\!\!\!  \!\!\!\!   \langle \check x ,\check y \rangle\!^s  \!\!\!\!  &=&  \!\!\!\!   (-1)^{\vert x' \vert \vert \rho(x'',y')'\vert }\, s_A(\rho(x'',y')')x'\rho(x'',y')'' s_A(y'')  \!\!\mod [A,A], \\
\label{|-,-|}  \!\!\!\! \!\!\!\!  \!\!\!\!   \vert \check x, \check y \vert  \!\!\!\!  &=&  \!\!\!\!  \varepsilon_A \rho(x'',y')\,  \check{x}' \otimes \check{y}''.
\end{eqnarray}

 \subsection{The invariant subalgebra}
Let $A$ be  a cocommutative graded Hopf algebra,   and let $B$ be  an ungraded commutative Hopf algebra.
 Recall the $B$-coaction $\Delta: {A_B \to A_B \otimes B}$
   given by Lemma \ref{coaction_M} and the    subalgebra   $A_B^{\operatorname{inv}} \subset A_B $    of  $B$-invariant elements of $A_B$.
    Lemma~\ref{lemma-cosym} implies that $x_b \in A_B^{\operatorname{inv}}$  for all $x \in A$ and all cosymmetric $b\in B$.
  The defining relations of the algebra $A_B$ imply that such $x_b $   depends only  on  $\check x \in \check A$.
We will compute our bracket in $A_B$ on   such elements.   We start with the following lemma.

\begin{lemma}
Let $\bullet:B \times B \to \kk$ be a balanced biderivation. For any  cosymmetric elements $b,c\in B$, there is   a linear map
$$
b \smile c: A^r\! \otimes_A\! {}^\ell\! A \longrightarrow A_B^{\operatorname{inv}}
$$
carrying any $x\otimes y\in A^{\otimes 2}$ to  $(b' \bullet c')\, x_{b''}\, y_{c''} \in A_B$.
\end{lemma}

\begin{proof}
Fix  cosymmetric  $b,c\in B$.
For any $x,y \in A$,  set
$$
(b \smile c)(x,y)= (b' \bullet c')\, x_{b''} y_{c''} \in A_B .
$$
Let $s=s_B$ denote the antipode of $B$.
For any homogeneous $x,y,z\in A$, we have
\begin{eqnarray*}
&& (b \smile c)\big(\operatorname{ad}^r(x,y),z\big) \\
&=& (-1)^{\vert x \vert \vert y' \vert} (b' \bullet c')  \big(s_A(y') x y'' \big)_{b''} z_{c''} \\
&=& (-1)^{\vert x \vert \vert y' \vert} (b^{(1)} \bullet c')\, y'_{s(b^{(2)})} x_{b^{(3)}}  y''_{b^{(4)}} z_{c''} \\
&=& (-1)^{\vert x \vert \vert y \vert} (b^{(1)} \bullet c')\, y_{s(b^{(2)}) b^{(4)}} x_{b^{(3)}}   z_{c''} \\
& \by{cyclic_inv} & (-1)^{\vert x \vert \vert y \vert} (b^{(2)} \bullet c')\, y_{s(b^{(3)}) b^{(1)}} x_{b^{(4)}}   z_{c''}  \\
& \by{balanced} &(-1)^{\vert x \vert \vert y \vert}  (b' \bullet c^{(2)})\, y_{s(c^{(1)}) c^{(3)}} x_{b''}   z_{c^{(4)}} \\
&=& (-1)^{\vert x \vert \vert y \vert + \vert y' \vert \vert y'' \vert } (b' \bullet c^{(2)})\, \big(s_A(y'')\big)_{c^{(1)}} y'_{c^{(3)}} x_{b''}   z_{c^{(4)}} \\
&=& (-1)^{ \vert y'' \vert \vert z \vert } (b' \bullet c^{(2)})\,  x_{b''}  y'_{c^{(3)}}  z_{c^{(4)}} \big(s_A(y'')\big)_{c^{(1)}}  \\
& \by{cyclic_inv} & (-1)^{ \vert y'' \vert \vert z \vert } (b' \bullet c^{(1)})\,  x_{b''}  y'_{c^{(2)}}  z_{c^{(3)}} \big(s_A(y'')\big)_{c^{(4)}}  \\
&=&  (-1)^{ \vert y'' \vert \vert z \vert } (b' \bullet c')\,  x_{b''}  \big(y' z s_A(y'')\big)_{c''}  \ = \ (b \smile c)\big(x,\operatorname{ad}^\ell (y,z)\big).
\end{eqnarray*}
Thus we obtain a linear map $b \smile c: A^r\! \otimes_A\! {}^\ell\! A \to A_B$, and it remains to verify that it takes values in  $A_B^{\operatorname{inv}}$.
Indeed, for any $x,y\in A$, we have
\begin{eqnarray*}
\Delta\big((b \smile c)(x\otimes y)\big) &=& (b' \bullet c')\,  \Delta(x_{b''})\, \Delta(y_{c''}) \\
&=& (b^{(1)} \bullet c^{(1)})\, x_{b^{(3)}} y_{c^{(3)}} \otimes s(b^{(2)}) b^{(4)}  s(c^{(2)}) c^{(4)} \\
&\by{cyclic_inv} &  (b^{(2)} \bullet c^{(2)})\, x_{b^{(4)}} y_{c^{(4)}} \otimes s(b^{(3)}) b^{(1)}  s(c^{(3)}) c^{(1)} \\
& \by{balanced_bis} &  (b^{(1)} \bullet c^{(3)})\, x_{b^{(4)}} y_{c^{(4)}} \otimes s(b^{(3)}) b^{(2)}  s(c^{(2)}) c^{(1)} \\
&=& (b' \bullet c')\, x_{b''} y_{c''} \otimes 1_B \ = \ (b \smile c)(x\otimes y) \otimes 1_B
\end{eqnarray*}
which shows that $(b \smile c)(x\otimes y)  \in A_B^{\operatorname{inv}}  $.
\end{proof}

\begin{theor} \label{smile_to_cup}
Assume the conditions of Theorem~\ref{double_to_simple}.
Then, for any $x,y\in A$ and any cosymmetric  $b,c\in B$, we have
$$
\bracket{x_b}{y_c} = (b\smile c)\, \Theta_\rho(\check x,\check y) \ \in A_B^{\operatorname{inv}}.
$$
\end{theor}

\begin{proof}
Indeed,
\begin{eqnarray*}
\bracket{x_b}{y_c}  &=&  (-1)^{ \vert   x''\vert   \vert y' \vert_n}  (c''\bullet b^{(2)})\, \rho(x',y')_{ {s}_B(b^{(3)}) b^{(1)}}\,   x''_{b^{(4)}}\,   y''_{c' } \\
&=&  (-1)^{ \vert   x''\vert   \vert y' \vert_n}  (c''\bullet b^{(2)})\, s_A(\rho(x',y')')_{ b^{(3)}} \rho(x',y')''_{ b^{(1)}}\,   x''_{b^{(4)}}\,   y''_{c' } \\
&=&   (-1)^{ \vert   x''\vert   \vert y' \rho(x',y')'' \vert_n}  (c''\bullet b^{(2)})\, s_A(\rho(x',y')')_{ b^{(3)}} x''_{b^{(4)}}\, \rho(x',y')''_{ b^{(1)}}\,      y''_{c' } \\
& \by{cyclic_inv} &   (-1)^{ \vert   x''\vert   \vert y' \rho(x',y')'' \vert_n}  (c''\bullet b^{(1)})\, s_A(\rho(x',y')')_{ b^{(2)}} x''_{b^{(3)}}\, \rho(x',y')''_{ b^{(4)}}\,      y''_{c' } \\
&=&  (-1)^{ \vert   x''\vert   \vert y' \rho(x',y')'' \vert_n}  (c''\bullet b')\, \big(s_A(\rho(x',y')') x'' \rho(x',y')''\big)_{b''}\,      y''_{c' } \\
&=&  (-1)^{ \vert   x''\vert   \vert \rho(x',y')' x' \vert}  (c''\bullet b')\, \big(s_A(\rho(x',y')') x'' \rho(x',y')''\big)_{b''}\,      y''_{c' } \\
&=&  (-1)^{ \vert   x'\vert   \vert \rho(x'',y')' \vert}  (c''\bullet b')\, \big(s_A(\rho(x'',y')') x' \rho(x'',y')''\big)_{b''}\,      y''_{c' } \\
&=& (b'\bullet c'')\, \big(\operatorname{ad}^r(x',\rho(x'',y'))\big)_{b''}\, y''_{c'} \ = \ (b \smile c)\, \Theta_\rho (\check x,\check y).
\end{eqnarray*}

\up
\end{proof}

 We  illustrate Theorem \ref{smile_to_cup} by considering the examples of Section~\ref{trace-like_examples}, where
  $B$ is  the coordinate algebra of one of the classical group schemes $\GL_N$, $\hbox{SL}_N$ or $\hbox{O}_N$,
and the balanced biderivation $\bullet=\bullet_t:B \times B \to \kk$ is induced by the usual trace $t\in B$.
The  map $t\smile t:  A^r\! \otimes_A\! {}^\ell\! A \to A_B^{\operatorname{inv}}$ is easily
computed from the formulas given  there:   for  any $x,y \in A$,
$$
(t \smile t)(x \otimes y) =
\left\{\begin{array}{ll} (xy)_t & \hbox{if } B =\kk[\GL_N], \\
(xy)_t - \frac{1}{N} x_t y_t & \hbox{if } B =\kk[\hbox{SL}_N], \\
\frac{1}{2}(xy)_t  - \frac{1}{2}\big(xs_A(y)\big)_t & \hbox{if } B =\kk[\hbox{O}_N].
\end{array}\right.
$$
  Theorem \ref{smile_to_cup} implies that, in   the notations \eqref{<-,->}--\eqref{|-,-|},
\begin{equation} \label{classical_groups}
\bracket{x_t}{y_t} =
\left\{\begin{array}{ll}
\langle \check x , \check y \rangle_t & \hbox{if } B =\kk[\GL_N], \\
\langle \check x , \check y \rangle_t -\frac{1}{N} \vert \check x, \check y\vert^\ell_t\, \vert \check x, \check y\vert^r_t   & \hbox{if } B =\kk[\hbox{SL}_N], \\
\frac{1}{2} \langle \check x , \check y \rangle_t  - \frac{1}{2} \langle \check x , \check y \rangle^s_t & \hbox{if } B =\kk[\hbox{O}_N]
\end{array}\right.
\end{equation}
where we expand   $\vert-,-\vert =  \vert-,-\vert^\ell \otimes \vert-,-\vert^r$ in the second  formula.

\section{From surfaces to Poisson brackets} \label{surfaces}

In this section, $\Sigma$ is a connected    oriented surface with non-empty bondary
 and      $\pi=\pi_1(\Sigma,\ast)$ is  the  fundamental group of $\Sigma$ based at a point $\ast \in \partial \Sigma$.
 We assume   that 2 is invertible in the ground ring $\kk$ and   view the group algebra $A =\kk \pi$ as an ungraded Hopf algebra in the standard way.

\subsection{The intersection pairing}

The    Hopf   algebra $A =\kk \pi$ carries a canonical Fox pairing~$\rho$   called the \emph{homotopy   intersection pairing of~$\Sigma$}.
It was first  introduced (in a slightly different form) in \cite{Tu_loops}, and is defined as follows.
Pick  a   point $\star \in \partial \Sigma\setminus \{\ast \}$ lying slightly before~$\ast$ with respect to  the  orientation  of~$\partial \Sigma$ induced by that of~$\Sigma$.
Let $\nu_{\star \ast}$ be a short   path in $\partial \Sigma$ running from~$\star$ to~$\ast$   in the positive direction,
and let $\overline{\nu}_{\ast \star }$ be the inverse path.
 Given a loop~$\beta$ in~$\Sigma$ based at~$\ast$, we let    $[\beta]\in \pi$ be the homotopy class of $\beta$.
Every simple point~$p $ on  such a~$\beta$  splits~$\beta$ as a product of the path  $\beta_{\ast p}$  running   from~$\ast$ to~$p$ along~$\beta$
 and the path    $\beta_{p\ast}$    running  from $p$ to $\ast$  along $\beta$. Similar notation applies to simple points of loops based at~$\star$.
For any $x,y\in \pi$,  set
\begin{equation} \label{intersection}
\rho(x,y) = \sum_{p\in \alpha \cap \beta}\varepsilon_p(\alpha,\beta)  \,
\left[\overline{\nu}_{\ast \star }\alpha_{\star p} \beta_{p \ast}\right] + \frac{1}{2}(x-1)(y-1) \in A
\end{equation}
where   we use the following notation:
$\alpha$ is a loop  in $\Sigma$  based at $\star$  such that $[\overline{\nu}_{\ast \star} \alpha \nu_{\star \ast}]=x$;
$\beta$~is a loop    in $\Sigma$ based at $*$ such that  $[\beta]=y$ and $  \beta$ meets $\alpha$ transversely in a
finite set  $ \alpha \cap \beta$ of simple points;  for
    $p\in \alpha \cap \beta$, we set  $\varepsilon_p(\alpha,\beta)=+ 1$ if  the
frame (the positive tangent vector of $\alpha$ at $p$, the positive
tangent vector of $\beta$ at $p$) is positively oriented and $\varepsilon_p(\alpha,\beta)= -1$ otherwise.
Then $\rho :A \times A \to A $ is a well-defined  antisymmetric Fox pairing.
As explained in   Appendix~\ref{FPvDB},   this Fox pairing    is quasi-Poisson.

  \subsection{The operation $\Theta$}
 The module $\check A= A/[A,A]$  can be identified with the module $\kk \check{\pi}$ freely generated by the set $\check{\pi}$ of homotopy classes of free loops in $\Sigma$.
 By   Lemma \ref{Theta}, the intersection pairing $\rho : A \times A \to A$ induces   a bilinear   pairing
$$
\Theta=\Theta_\rho : \kk \check{\pi} \times \kk \check{\pi} \longrightarrow (\kk \pi)^r \otimes_{\kk \pi}\! {}^\ell (\kk \pi)
$$
where the left and right $\kk\pi$-module structures ${}^\ell (\kk \pi)$ and $(\kk \pi)^r$ on $\kk\pi$ are induced  by the conjugation action of  $\pi$.
A direct   computation shows that, for any   $\check x,\check y \in \check \pi$,
\begin{equation} \label{generalized_Goldman_bracket}
\Theta(\check x,\check y) =  \sum_{p\in \alpha \cap \beta} \varepsilon_p(\alpha,\beta)  \, [\gamma_{p} \alpha_{p} \overline{\gamma_p}] \otimes  [\gamma_{p} \beta_{p} \overline{\gamma_p}]
\end{equation}
  where we use the following notations:  $\alpha, \beta$ are free loops
 in $\Sigma$  representing, respectively,  $\check x, \check y$  and meeting transversely in a finite set $ \alpha \cap \beta$ of simple points;
 $\alpha_p,\beta_p$ are  the loops $\alpha,\beta$ based at $p \in \alpha \cap \beta$, and   $\gamma_p$ is an arbitrary path  in $\Sigma$   from $\ast$ to $p$.

 \subsection{Brackets in $A_B$}   Let  $B$ be a commutative ungraded Hopf algebra carrying  a balanced biderivation~$\bullet$.
By Theorem~\ref{double_to_simple}, $\rho$ and~$\bullet$ induce   a  bracket $\bracket{-}{-}$  in the representation algebra~$A_B$.
Using \eqref{main_formula} and \eqref{balanced_bis}, we obtain   for  any  $x,y \in \pi$ and   $b,c\in B$,
{\small \begin{eqnarray}
 \label{LSNi_revisited}  \bracket{x_b}{y_c} & = &\sum_{p\in \alpha \cap \beta}\varepsilon_p(\alpha,\beta)\,   (c'' \bullet b^{(2)}) \,
\left[\overline{\nu}_{\ast \star }\alpha_{\star p} \beta_{p \ast}\right]_{s_B(b^{(3)}) b^{(1)} } x_{b^{(4)}} y_{c'}\\
\notag && + \frac{1}{2}\big( (c' \bullet b'')\, x_{b'} y_{c''} + (c''\bullet b')\, x_{b''} y_{c'} - (c''\bullet b'')\, x_{b'} y_{c'} - (c' \bullet b')\, y_{c''} x_{b''}\big),
\end{eqnarray}}
where we use the same notation as in  \eqref{intersection}.
If      $b$ and $c $     are {cosymmetric},  then      Theorem~\ref{smile_to_cup}  gives a   simpler expression for the bracket:
\begin{equation}
\label{Goldman_revisited}
\bracket{x_b}{y_c}  =  \sum_{p\in \alpha \cap \beta} \varepsilon_p(\alpha,\beta)\, (b' \bullet c') \, [\gamma_{p} \alpha_{p} \overline{\gamma_p}]_{b''}\,  [\gamma_{p} \beta_{p} \overline{\gamma_p}]_{c''}.
\end{equation}

 Assume  now   that $\bullet=\bullet_t$ where $t\in B$ is a trace-like element.  It follows from  Theorem~\ref{qPoisson_case} that
 the  bracket $\bracket{-}{-}$ in $A_B$ is quasi-Poisson with respect to~$t$. As observed in Section~\ref{qP_algebras},
 this  bracket   restricts  to a Poisson bracket in $(A_B)^{\operatorname{inv}}$.
According to Section~\ref{ex_monoid_algebras}, the algebra $A_B$ is the coordinate algebra of the affine scheme $\Hom_{\Gr}(\pi,\Gscr(-))$
where $\Gscr$ is the group scheme associated to $B$. By Appendix \ref{Lie_algebra}, the   biderivation $\bullet=\bullet_t$
   is tantamount to a metric on the Lie algebra of~$\Gscr$.
 The  Poisson bracket $\bracket{-}{-}$  in $(A_B)^{\operatorname{inv}}$ is
 an algebraic version of the Atiyah--Bott--Goldman Poisson structure
 on the moduli space of representations of~$\pi$ in a Lie group  whose Lie algebra is endowed with a metric,
see  \cite{AB,Go1}.
Indeed, formula~\eqref{Goldman_revisited}  is the algebraic analogue of  Goldman's  formula \cite[Theorem~3.5]{Go2},
where the operation \eqref{generalized_Goldman_bracket}  appears implicitly;
for instance, the  formulas \eqref{classical_groups} correspond to \cite[Theorems 3.13--3.15]{Go2}.
  The     quasi-Poisson bracket   $\bracket{-}{-}$   in $A_B$ is an algebraic version of the quasi-Poisson refinement  of the Atiyah--Bott--Goldman     bracket
   introduced    in \cite{AKsM} and studied in \cite{LS,Ni}.
Indeed,  formula~\eqref{LSNi_revisited}    is the algebraic  analogue of the quasi-Poisson refinement  of Goldman's formulas
obtained by  Li-Bland \& \v Severa \cite[Theorem 3]{LS} and Nie \cite[Theorem 2.5]{Ni}.
For $\Gscr=\GL_N$, formula~\eqref{LSNi_revisited}  was obtained in \cite{MT_dim_2}  using Van den Bergh's theory (see Appendix \ref{relation_to_VdB}).
  Note that our proof of the quasi-Jacobi identity in $A_B$
  (and, consequently, of the Jacobi identity in  $(A_B)^{\operatorname{inv}}$) is purely algebraic  and involves neither
   infinite-dimensional methods of \cite{AB,Go1, Go2} nor the inductive   ``fusion''   method of   \cite{AKsM,LS,Ni}.

\appendix

\section{Group schemes} \label{group_schemes}

We       review      group schemes     (following  mainly   \cite{Ja})
  and   reformulate in terms of   group schemes some of the notions introduced in the main body of the  paper.
 In this appendix, by a module/algebra we mean an ungraded module/algebra  over the ground ring~$\kk$.

\subsection{Affine schemes}

Let $\Com$ be the category of  commutative algebras and algebra homomorphisms. Let $\Set$  be the category of sets and maps.
By a \emph{$\kk$-functor} we mean a    (covariant)  functor $  \Com \to \Set$;
  for instance, the forgetful functor  $\Com \to \Set$  is a $\kk$-functor which we denote by $\Kscr$.
A \emph{morphism} $\alpha:\Xscr \to \Yscr$ of $\kk$-functors is a natural transformation of  functors; for  a  commutative algebra $C$,
we let $\alpha_C: \Xscr(C) \to \Yscr(C)$ be the   map determined by~$\alpha$. Such a morphism~$\alpha$ is an  \emph{isomorphism}   if
$\alpha_C$ is a bijection for all~$C$.

For example, a commutative  algebra  $B$   determines a $\kk$-functor  $\Hom_{\Com}( B  , -) $
and a homomorphism (respectively,   isomorphism) of commutative algebras ${B\to B'}$ determines
a morphism (respectively,   isomorphism) of  $\kk$-functors $\Hom_{\Com}( B'  , -)\to \Hom_{\Com}( B  , -)$.

An  \emph{affine scheme $ \Xscr$ (over $\kk$) with coordinate algebra $B$}   is a triple consisting of a   $\kk$-functor $ \Xscr$,
  a commutative algebra~$B$, and an isomorphism of $\kk$-functors $\eta: \Xscr \to \Hom_{\Com}(B, -) $.  For shorteness,  such a triple is    denoted simply by $\Xscr$,
  the algebra $B$  is  denoted by  $\kk[\Xscr]$, and  the  isomorphism  $\eta$ is  suppressed from notation.
  The \emph{evaluation} of an
    $f\in B$ at      $x\in \Xscr(C) $ is defined by   $f\vert_x =\eta_C(x) (f) \in C$.
  By the Yoneda lemma,  these    evaluations define a  canonical  bijection
  \begin{equation} \label{usual_K[X]}
  B \stackrel{\simeq}{\longrightarrow} \hbox{Mor}(\Xscr,\Kscr),
  \end{equation}
  where $\hbox{Mor}(\Xscr,\Kscr)$ is   the set of morphisms of $\kk$-functors   $\Xscr \to \Kscr$.

 A    \emph{morphism}    of  affine schemes is  a morphism of the underlying $\kk$-functors.
By the  Yoneda   lemma, given a  morphism   of  affine schemes $\alpha:\Xscr \to \Yscr$,
 there is a unique  algebra homomorphism $\alpha^*:\kk[\Yscr] \to \kk[\Xscr]$ such that
\begin{equation}\label{dual_morphism}
\alpha_C(x) = x \circ \alpha^*   \in  \Hom_{\Com}(\kk[\Yscr], C)   \simeq   \Yscr(C)
\end{equation}
for any commutative algebra $C$ and any $x\in  \Hom_{\Com}(\kk[\Xscr], C) \simeq    \Xscr(C)$.
 Note that
\begin{equation} \label{evaluate_morphisms}
\alpha^*(g)\vert_x=x (\alpha^*(g))= \big(\alpha_C(x)\big) (g) = g \vert_{\alpha_C(x)}
\end{equation}
for any  $g\in \kk[\Yscr]$.  

Given two  $\kk$-functors $\Xscr$ and $\Yscr$,  the product $\kk$-functor $\Xscr \times \Yscr$ carries any commutative algebra $C$ to
 $(\Xscr\times \Yscr) (C)=\Xscr(C) \times \Yscr(C)$.
If $\Xscr, \Yscr$ are affine schemes, then so is $\Xscr\times \Yscr$  with coordinate algebra   $\kk[\Xscr\times \Yscr]=\kk[\Xscr] \otimes \kk[\Yscr]$.

Given an affine scheme $\Xscr$ and a commutative algebra $C$,
we let  $\Xscr_C$ be   the $C$-functor which assigns to any commutative $C$-algebra $D$    the set $\Xscr(D)$. Then
$$
\Xscr_C(D) \simeq \Hom_{\Com}(\kk[\Xscr], D) = \Hom_{C\hbox{-}\Com}(\kk[\Xscr] \otimes C, D),
$$
where $C\hbox{-}\Com$  is  the category of commutative algebras over $C$.
Hence, $\Xscr_C$ is an affine scheme over (the underlying ring of) $C$ with coordinate algebra $C[\Xscr_C]=\kk[\Xscr] \otimes C$.

\subsection{Monoid schemes}\label{Monoid schemes}

Let $\Mon$ be the category of monoids   and monoid homomorphisms.
  A \emph{monoid scheme}  (over $\kk$) is  an affine scheme whose underlying $\kk$-functor
is lifted to the category $\Mon$ with respect to  the forgetful functor $\Mon \to \Set$.
The theory of monoid schemes   is equivalent to the theory of  commutative
 bialgebras. If~$B$ is a commutative bialgebra, then $\Hom_{\Com}(B,-)$ is   a monoid scheme:
for any commutative algebra $C$, the  set   $  \Hom_{\Com}(B,C)$ with the convolution product is a monoid with unit $\varepsilon_B   1_C:B\to C$.
 Conversely, if $\Gscr$ is a monoid scheme, then its coordinate algebra $B=\kk[\Gscr]$  is  a commutative bialgebra:
the   counit $ \varepsilon_B : B \to \kk$ is the  neutral   element
of the monoid $\Gscr(\kk)\simeq   \Hom_{\Com}(B, \kk)$;  the
comultiplication $ \Delta_B : B \to B \otimes B=  \kk[\Gscr \times \Gscr]$ is evaluated on any $f \in B$ by    $ \Delta_B
(f)\vert_{(x,y)}= f\vert_{xy}  \in C $ for any commutative algebra~$C$ and any  $x,y\in \Gscr(C)$.

A (left) \emph{action  of a monoid scheme $\Gscr$  on  a $\kk$-functor~$\Xscr$} is a  morphism   of $\kk$-functors  $  \omega :\Gscr \times \Xscr \to \Xscr$
such that for any commutative algebra~$C$, the map   $$\omega_C: (\Gscr  \times \Xscr)(C)=\Gscr(C) \times \Xscr(C) \to \Xscr(C)$$
is a (left)   action of  the monoid  $\Gscr(C)$ on the set  $\Xscr(C)$.
The   image  under $\omega_C$   of a pair  $(g\in \Gscr(C), x\in \Xscr(C))$  is denoted by $g\cdot x$.

 A (left) \emph{action of a monoid scheme  $\Gscr$ on a  module~$M$} (over $\kk$)  is an action of~$\Gscr$
 on the $\kk$-functor  $  \Mscr  = M \otimes  (-)$ such that, for any commutative algebra~$C$,
the monoid $\Gscr(C)$ acts on $\Mscr(C)=M\otimes  C$ by $C$-linear transformations.
The study of  such actions  is equivalent to the study of  (right) $B$-comodules, where $B=\kk[\Gscr]$.
 Specifically,   a  comodule map $\Delta_M:M \to M \otimes B$   induces  the following  action of $\Gscr$  on $M$:
for  a   commutative algebra~$C$,   an element~$g$ of $ \Gscr(C) \simeq   \Hom_{\Com}(B, C)$ acts  on    $m\otimes c \in  M  \otimes C$   by
\begin{equation}  \label{action}
g \cdot (m\otimes c ) = m^\ell \otimes \big(m^r\vert_g\, c\big)=   m^\ell \otimes g(m^r)\, c
\end{equation}
where $\Delta_M(m)=m^\ell \otimes m^r \in M \otimes B$ in Sweedler's notation.
Conversely, an action of~$\Gscr$  on~$M$ induces a  comodule map  ${ \Delta_M: M \to M \otimes B}$   carrying  any $m\in M$ to
$ \id_{B}\cdot (m\otimes 1)$,  where $\id_{B}\in \Gscr(B) \simeq   \Hom_{\Com}(B, B) $  acts on $  M \otimes B$.
 Note that an element $m\in M$ is $\Gscr$-invariant (in the sense that    $g\cdot (m\otimes 1_C)= m\otimes 1_C$  for any $g\in \Gscr(C)$
and any commutative algebra $C$)
if and only if $\Delta_M(m)= m\otimes 1_B$.
 Note    also that, when~$M$ is an algebra, $\Delta_M$ is an algebra homomorphism if and only if
 for   all  commutative algebras $C$, the monoid $\Gscr(C)$ acts on $M\otimes C$ by   $C$-algebra endomorphisms;
we    speak then of an   \emph{action of~$\Gscr$ on the algebra $M$}.

\subsection{Group schemes}\label{induced_actions}

A \emph{group scheme}  (over $\kk$) is a   monoid scheme   $ \Gscr  $ such that  the monoid $\Gscr (C)$ is a group for every commutative algebras $C$.
Under the equivalence between monoid schemes and commutative bialgebras,   the group schemes correspond to
the commutative Hopf algebras. The antipode in the coordinate   algebra~$B$   of a group scheme
is the inverse of $\id_B$ in the group $\Hom_{\Com}(B,B)$. Conversely, if~$B$ is a commutative Hopf algebra and~$C$ is a commutative algebra,
then the monoid   $ \Hom_{\Com}(B,C)$ is a group with inversion obtained by composition with the antipode of~$B$.

 In the next two lemmas,  $\Gscr$ is    a group scheme with
coordinate  algebra $B=\kk[\Gscr] $ and $ s_B:B \to B$ is the antipode of  $B$.

\begin{lemma} \label{general_lemma}
A (left) action $\omega$ of $\Gscr$ on an affine scheme $\Xscr$
induces  a   (left) action of $\Gscr$  on the algebra $\kk[\Xscr]$.   Specifically,
for any   commutative algebra $C$, the action of a   $g\in \Gscr(C)$ on  an  $f\in \kk[\Xscr] \otimes C = C[\Xscr_C]  $   is given by
\begin{equation} \label{(g.f)(x)}
(g\cdot f)\vert_{x} = f\vert_{g^{-1}\cdot x}  = f\vert_{\omega_{D}(g^{-1}, x)} \in D
\end{equation}
for any commutative $C$-algebra $D$ and any  $x\in \Xscr_C(D)=  \Xscr(D)$.
The corresponding (right) comodule map $\Delta = \Delta_{\kk[\Xscr]}: \kk[\Xscr]  \to \kk[\Xscr] \otimes B$  is computed by
\begin{equation}  \label{decompose_Delta}
\Delta= (\id_{\kk[\Xscr]} \otimes s_B) \Perm \omega^*
\end{equation}
where $\omega^*:  \kk[\Xscr] \longrightarrow  \kk[\Gscr \times \Xscr] = B  \otimes \kk[\Xscr]$ is  induced by the morphism $\omega: \Gscr \times \Xscr \to \Xscr$
 and $\Perm: B  \otimes \kk[\Xscr] \to \kk[\Xscr] \otimes B$ is the permutation.
\end{lemma}

\begin{proof}
Fix a commutative algebra $C$ and  $g\in \Gscr(C)$.
For a  commutative $C$-algebra~$D$, the  algebra homomorphism $C \to D, c\mapsto c \cdot 1_D$
 allows    us to  view  $g\in \Gscr(C)$   as an element of $\Gscr(D)$.
Consider the    automorphism of the affine scheme  $\Xscr_C $
  given, for any commutative $C$-algebra $D$, by  the bijection $$\omega_D(g,-): \Xscr_C(D) =\Xscr(D) \to \Xscr(D)=   \Xscr_C(D).$$
Thus we obtain   a group homomorphism $\Gscr(C) \to \operatorname{Aut}(\Xscr_C)$,
which we compose and pre-compose with the following group anti-isomorphisms:{\small
$$
\xymatrix{
\Gscr(C) \ar[rr]^-{\operatorname{group\, inversion}}_{ \simeq } && \Gscr(C) \ar[r]  & \Aut(\Xscr_C)  \ar[r]_-\simeq
&  \Aut_{C\hbox{-}\Com}(C[\Xscr_C]) \subset  \Aut_{\Com}(\kk[\Xscr]\otimes C).
}
$$}
The composed  group homomorphisms  $\Gscr(C) \to  \Aut_{\Com}(\kk[\Xscr]\otimes C)$ defined for all commutative algebras $C$
 yield an  action   of $\Gscr$ on the algebra $\kk[ \Xscr  ]$.
The formula \eqref{(g.f)(x)} easily follows from this definition.
 To show \eqref{decompose_Delta}, we need to verify that for all $m\in \kk[\Xscr]$,
 $$
 \Delta (m)=(\id_{\kk[\Xscr]} \otimes s_B) \Perm \omega^*(m) \in
 \kk[\Xscr] \otimes B = \kk[\Xscr \times \Gscr].
 $$
 It  suffices to show that the  evaluations of both sides at $(x,g)$ coincide
for any  commutative algebra $C$
and any   $$x \in \Xscr(C)\simeq \Hom_{\Com}(\kk[\Xscr],C)  \quad {\text {and}} \quad g \in \Gscr(C)\simeq \Hom_{\Com}(B,C)  .$$
Observe that
\begin{eqnarray} 
 \notag  g^{-1} \cdot x  &=&  \omega_C(g^{-1},x) \\
  \label{dual_of_action} &\by{dual_morphism}&   \hbox{mult}_C\circ (g^{-1}\otimes x) \circ\omega^*   \in \Xscr(C) \simeq \Hom_{\Com}(\kk[\Xscr],C).
\end{eqnarray}
By definition,
$$
\Delta (m)\vert_{(x,g)} =  \hbox{mult}_C (x \otimes g) \big( \Delta (m)\big) = m^\ell \vert_x\, m^r \vert_g \in C
$$
where   $\Delta (m)=m^\ell \otimes m^r$  in Sweedler's notation. On the other hand,
\begin{eqnarray*}
(\id_{\kk[\Xscr]} \otimes s_B) \Perm \omega^*(m)\vert_{(x,g)} &=& \hbox{mult}_C (x \otimes g) \big( (\id_{\kk[\Xscr]} \otimes s_B) \Perm \omega^*(m) \big)\\
&=&  \hbox{mult}_C (g^{-1} \otimes x) \big( \omega^*(m)\big) \\
&\by{dual_of_action} & (g^{-1} \cdot x)(m) \\
&=& (g^{-1} \cdot x)(m \otimes 1_C) \\
& \by{(g.f)(x)} &   x  \big( g \cdot (m\otimes 1_C)\big)\\
& \by{action} & x\big(m^\ell \otimes g(m^r) \big) \ = \  x(m^\ell) g(m^r) \ = \
m^\ell \vert_x  m^r \vert_g
\end{eqnarray*}
where, in  the last three lines, $\Xscr(C)\simeq \Hom_{\Com}(\kk[\Xscr],C) $  is also identified with the set $ \Hom_{C\hbox{-}\Com}(\kk[\Xscr] \otimes C,C) $.
Hence the two evaluations at $(x,g)$ coincide.
\end{proof}

We  now recall  the ``group scheme'' interpretation  of the adjoint coaction \eqref{adjoint_coaction}.

\begin{lemma}   \label{conjugation_action}
The action of~$\Gscr$  on itself     by   conjugation   induces an action  of $\Gscr$ on the algebra $B=\kk[\Gscr] $.
The  corresponding    comodule map $\Delta: {B} \to {B} \otimes {B}$  is  given~by
\begin{equation} \label{Delta^c}
\Delta (f) = f'' \otimes s_B(f') f'''
\end{equation}
for any $f\in B$, where    $(\Delta_B \otimes \id) \Delta_B(f)=f' \otimes f'' \otimes f'''$   in Sweedler's notation.
\end{lemma}

\begin{proof}
The first statement directly follows from Lemma \ref{general_lemma} and we only need to prove \eqref{Delta^c}.
The conjugation action of $\Gscr$ on itself is given by the morphism $\omega: \Gscr \times \Gscr \to \Gscr$
defined by $\omega_C(g,m)= gmg^{-1}$ for any commutative algebra $C$ and any $g,m\in \Gscr(C)$.
The  value of $\omega^*: B \to B \otimes B$ on  an element    $f\in B$ is  computed by
\begin{eqnarray*}
\omega^*(f)\vert_{(g,m)}&  \by{evaluate_morphisms} & f\vert_{\omega_C(g,m)} \\
&=& f\vert_{gmg^{-1}}\\
&= &f'\vert_g f''\vert_m f'''\vert_{g^{-1}} \\
& = & \big(f's_B(f''')\big)\vert_g f''\vert_m \ = \ \big( f's_B(f''') \otimes f'' \big)\vert_{(g,m)}.
\end{eqnarray*}
Thus,   $\omega^*(f) = f's_B(f''') \otimes f''$ and   \eqref{Delta^c} follows from \eqref{decompose_Delta}.
\end{proof}

By   Remark \ref{about_tle}.1,   an element of $  B=\kk[\Gscr]$ is cosymmetric if and only if
it is invariant under the $\Gscr$-action given by Lemma \ref{conjugation_action}.
This observation   can be used to produce cosymmetric elements    from linear representations of~$\Gscr$.
Consider an action~$\omega$ of the group scheme~$\Gscr$ on a finitely-generated free module~$M$.
The \emph{character}  of $\omega$ is  the morphism $\chi_\omega: \Gscr   \to  \Kscr$
defined, for any commutative algebra $C$,  by the composition
$$
\xymatrix{
\Gscr(C) \ar[r]^-{\omega_C} & \Aut_C(M\otimes C) \ar[r]^-{\operatorname{trace}} & C.
}
$$
Clearly, $\chi_\omega$ is $\Gscr$-equivariant if $\Kscr$ is endowed with the trivial $\Gscr$-action.
Therefore the element $t_\omega\in B$  corresponding to $\chi_\omega\in \hbox{Mor}(\Gscr,\Kscr)$ via the bijection \eqref{usual_K[X]} is $\Gscr$-invariant,
and we deduce   that $t_\omega$ is cosymmetric.
  The examples of trace-like elements given in Section \ref{trace-like_examples} for $\Gscr= \GL_N, \hbox{SL}_N, \hbox{O}_N$ arise in this way;
further examples are obtained by considering other matrix group schemes.

 \subsection{Equivariant pairings} \label{equivariance}

 Let $\Gscr$ be a group scheme acting on   modules $M,N$  and  let $q :M \times M \to N$ be a bilinear map.
Given a commutative algebra~$C$, we define
 a $C$-bilinear  map $q_C: (M\otimes C) \times (M\otimes C) \to   N\otimes C $   by
$$q_C(m_1\otimes c_1, m_2 \otimes c_2) = q(m_1,m_2) \otimes c_1c_2$$ for any  $m_1,m_2 \in M$ and $c_1,c_2 \in C$.
The pairing $q  $ is said to be \emph{$\Gscr$-equivariant} if
$$q_C \big(g\cdot (m_1\otimes c_1), g\cdot (m_2 \otimes c_2)\big) = g \cdot  q_C(m_1\otimes c_1, m_2 \otimes c_2) $$
for all $C,  m_1,m_2 ,  c_1,c_2$ as above and   all $g \in \Gscr(C)$.
If $N=\kk$ with the trivial $\Gscr$-action, then the pairing  $q$ is said to be \emph{$\Gscr$-invariant}.

\begin{lemma} \label{i_b_f}
A bilinear map $q :M \times M \to N$ is  $\Gscr$-equivariant
if, and only if, it is $B$-equivariant in the sense of
Section~\ref{comodules}, where $B=\kk[\Gscr]$ and
$M,N$ are viewed as $B$-comodules.
\end{lemma}

\begin{proof}
Assume first  that $q$ is $\Gscr$-equivariant.   We must prove that
\begin{equation} \label{G-equivariance}
q(m_1, m_2^{\ell}) \otimes  m_2^{r} = \big(q(m_1^{\ell},m_2)\big)^{\ell} \otimes  \big(q(m_1^{\ell},m_2)\big)^{r}  s_B(m_1^{r}) \ \in N \otimes B
\end{equation}
for any $m_1, m_2 \in M$. For a  commutative algebra $C$ and   $g\in \Gscr(C)$, we have
\begin{eqnarray*}
q(m_1, m_2^{\ell})\otimes m_2^{r}\vert_g & = & q_C(m_1 \otimes 1_C ,  m^{\ell}_2 \otimes m_2^{r}\vert_g)\\
 &\by{action} & q_C\big( m_1 \otimes 1_C,g \cdot (m_2\otimes 1_C) \big) \\
&=&   g \cdot q_C\big( g^{-1} \cdot ( m_1 \otimes 1_C), m_2\otimes 1_C \big) \\
& \by{action}     & g \cdot q_C\big(m_1^{\ell} \otimes m_1^{r}\vert_{g^{-1}} , m_2\otimes 1_C\big) \\
&=& g\cdot \big(q(m_1^{\ell},m_2)\otimes  s_B(m_1^{r})\vert_g\big) \\
& \by{action} & \big( q(m_1^{\ell},m_2)\big)^{\ell} \otimes   \big(q(m_1^{\ell},m_2)\big)^{r}\big\vert_g\, s_B(m_1^{r})\vert_g.
\end{eqnarray*}
Hence the formula \eqref{G-equivariance} follows.

Conversely, assume \eqref{G-equivariance} and let $m_1,m_2 \in M$, $c_1,c_2 \in C$ and $g\in \Gscr(C)$ where~$C$ is an arbitrary commutative algebra. Consider the comodule map $\Delta_M:M\to M\otimes B$ and expand  
$(\Delta_M\otimes \id_B)\Delta_M(m)= m^\ell\otimes m^{\ell r} \otimes m^r$ for all $m\in M$.  Then
\begin{eqnarray*}
&&  q_C \big(g\cdot (m_1\otimes c_1), g\cdot (m_2 \otimes c_2)\big) \\
&\by{action} & q_C\big(m_1^\ell \otimes m_1^{r}\vert_g c_1, m_2^\ell \otimes m_2^{r}\vert_g \, c_2\big) \\
 &=& q(m_1^\ell,m_2^\ell) \otimes m_2^{r}\vert_g    m_1^{r}\vert_g \, c_1c_2  \\
 & \by{G-equivariance} &\left.\big(q(m_1^\ell,m_2)\big)^\ell \otimes  \left( \big(q(m_1^\ell,m_2)\big)^r  s_B (m_1^{\ell r})\right)\right\vert_g   m_1^r \vert_g \, c_1c_2  \\
 &=&  \big(q(m_1^\ell,m_2)\big)^\ell \otimes     \big(q(m_1^\ell,m_2)\big)^r \big\vert_g  \left(s_B(m_1^{\ell r}) m_1^r\right)\vert_g \,  c_1c_2  \\
 &=&  \big( q(m_1,m_2)\big)^\ell \otimes     \big(q(m_1,m_2)\big)^r \big\vert_g \, c_1c_2  \\
&  \by{action}  & g \cdot \big(  q(m_1,m_2) \otimes c_1c_2\big) \
= \ g \cdot  q_C(m_1\otimes c_1, m_2 \otimes c_2).
\end{eqnarray*}
This proves the $\Gscr$-equivariance of $q$.
\end{proof}

For example, consider the action of $\Gscr$ on the algebra
$B=\kk[\Gscr]$ induced by conjugation (see Lemma \ref{conjugation_action}).
  Remark \ref{about_bd}.2
and Lemma~\ref{i_b_f} imply  that a bilinear form $ B \times B \to \kk$ is balanced
if and only if it is symmetric and  $\Gscr$-invariant.

\subsection{The Lie algebra of a group scheme} \label{Lie_algebra}

 Let $\Gscr$ be a group scheme with coordinate algebra
 $B=\kk[\Gscr]$. Consider the ideal   $I = \Ker(\varepsilon_B)$
 of~$B$, where $\varepsilon_B:B \to \kk$ is the counit.

\begin{lemma} \label{stabilizes_I}
The action   of $\Gscr$ on ${B}$  induced by  conjugation   (see
Lemma \ref{conjugation_action}) stabilizes the ideal $I$ and all
its powers.
\end{lemma}

\begin{proof}
Since~$\Gscr$ acts on $B$ by algebra automorphisms, it suffices to
prove that the action   of~$\Gscr$ on~${B}$ stabilizes~$I$, i.e.,
that
\begin{equation} \label{g(I.C)}
g\cdot(I \otimes C) \subset I\otimes C
\end{equation}
for any commutative algebra $C$ and   any $g\in \Gscr(C)$. Observe that
$$
I \otimes C = \Ker(\varepsilon_B \otimes \id_C) = \Ker(\varepsilon_{C[\Gscr_C]})
$$
where $\Gscr_C$ denotes the group scheme over $C$ induced by $\Gscr$ and $C[\Gscr_C]$
 is the corresponding  commutative   Hopf algebra over $C$.
Then, for any $b\in I$ and $c\in C$, we~have
\begin{eqnarray*}
\varepsilon_{C[\Gscr_C]}\big(g \cdot(b\otimes c)\big)  &=& \big(g \cdot(b\otimes c)\big) \vert_1 \\
&\by{(g.f)(x)} & (b\otimes c)\vert_{g^{-1}1g} \\
&=& (b\otimes c)\vert_1 \ = \  \varepsilon_{C[\Gscr_C]}(b\otimes c) \ = \ \varepsilon_B(b) \otimes c \ = \ 0
\end{eqnarray*}
where $1$ denotes the   neutral    element    of the group $\Gscr_C(C)=
\Gscr(C)$.  It  follows that $g \cdot(b\otimes c) \in I
\otimes C $. This proves \eqref{g(I.C)}.
\end{proof}

Denote by $\mathfrak{g}$  the module of derivations $B\to \kk$,
 where~$\kk$ is regarded as a  $B$-module via~$\varepsilon_B$.
 Then $\mathfrak{g}$ is a Lie algebra  with Lie bracket
 $$
 [\mu,\nu](b) = \mu(b') \nu(b'') - \nu(b') \mu(b'')
 $$
 for any   $\mu,\nu \in \mathfrak{g}$ and $b\in B$, where  $\Delta_B(b)=b' \otimes b''$ is the comultiplication
 in~$B$.
   Any derivation $ B \to \kk $ induces  a linear map
$ I/I^2\to \kk$, and this   identifies~$\mathfrak{g}$   with  $(I/I^2)^*$.

  We say that  the group scheme $\Gscr$ is \emph{infinitesimally-flat} if  the module $I/I^2$ is finitely-generated and projective,
or, equivalently, is   finitely-presented   and flat, cf.\ \cite[Chap$.$ II, \S 5.2]{Bo}.
   Our  condition of infinitesimal flatness is   somewhat weaker than   the one  in \cite{Ja} and, when
  $\kk$ is a field, is equivalent to the condition that  the Lie algebra $\mathfrak{g}$ is   finite-dimensional.
If $\Gscr$ is  infinitesimally-flat, then  for any commutative algebra~$C$, we have isomorphisms
\begin{equation} \label{inf-flat}
\mathfrak{g} \otimes C = (I/I^2)^* \otimes C \simeq \Hom(I/I^2,C)
\simeq \Hom_C\big((I/I^2) \otimes C,C\big).
\end{equation}
 Lemma \ref{stabilizes_I} yields    an  action   of~$\Gscr$ on
the module $I/  I^2$, which  induces   an action of~$\Gscr$
on~$\mathfrak{g} $ via  \eqref{inf-flat}.  The latter action
  is  called the \emph{adjoint representation} of~$\Gscr$.
 A~\emph{metric} in $\mathfrak{g}$ is a  $\Gscr$-invariant nonsingular symmetric bilinear form  $ \mathfrak{g} \times \mathfrak{g} \to \kk$.

\begin{lemma} \label{metric_on_g}
Let $\Gscr$ be an infinitesimally-flat group scheme with Lie algebra $\mathfrak{g}$.
There is a  canonical    embedding of the set of  metrics  in $\mathfrak{g}$
into the set of balanced biderivations in~$B=\kk[\Gscr]$.
\end{lemma}

\begin{proof}
The  linear map  $p:B \to I/I^2$ defined by $b \mapsto (b-\varepsilon_B(b)1_B)\! \mod I^2$ is $\Gscr$-equivariant,
and the $\Gscr$-invariance of a  symmetric bilinear form $\bullet: B \times B \to \kk$ is equivalent to the condition~\eqref{balanced}
as we saw at the end of Section~\ref{equivariance}.
Therefore,    we have a  bijective  correspondence
$$
\xymatrix{
\left\{\hbox{\small balanced biderivations in $B$} \right\}  \ar@/^2pc/[rr]^-{\hbox{\scriptsize restriction}}
\!\!\!\!   \!\!\!\!  && \ar@/^2pc/[ll]^-{\hbox{\scriptsize pre-composition with $p\times p$}} 
\!\!\!\!  \!\!\!\!   \left\{\hbox{\small $\Gscr$-invariant symmetric bilinear forms in $I/I^2$} \right\}.
}
$$
 Since   $\mathfrak{g}=(I/I^2)^*$, there is a canonical bijective correspondence
  between the set of nonsingular $\Gscr$-invariant symmetric bilinear forms in $I/I^2$ and the set of  metrics in $\mathfrak{g}$.
\end{proof}

\subsection{Representation algebras}
 We   give an interpretation  of the coaction    in
Lemma~\ref{coaction_M} (at least in the ungraded case).

\begin{lemma}
\label{conjugation_action_bis} Let $\Gscr$ be    a group scheme
with coordinate   algebra $B=\kk[\Gscr] $ and let $ s_B:{B \to B}$
be the antipode of  $B$. For any (ungraded) cocommutative bialgebra~$A$, the action of $\Gscr$ on $B$ given by
Lemma \ref{conjugation_action} induces, in a natural way, an
action of $\Gscr$ on the affine scheme $\Yscr_B^A$ of
$B$-representations of $A$. The corresponding comodule map
$\Delta: A_B \to A_B \otimes B$ is the unique algebra homomorphism such that
\begin{equation} \label{Delta(x_b)}
\Delta (x_b) = x_{b''} \otimes s_B(b') b'''
\end{equation}
for any  $x\in A$ and  $b\in B$.
\end{lemma}

\begin{proof}
Let $C$ be a commutative algebra. Recall that  in our
notation,
$$H_B(C)= \Hom(B,C)     \simeq \Hom_C(B\otimes C,C)$$
and that $H_B(C)$   is an algebra with   convolution
multiplication.    For   $g\in \Gscr(C)$ and  $u\in \Yscr^A_B(C) \subset \Hom(A, H_B(C))$,
we define a linear map $g\cdot u: A \to H_B(C)$   by
\begin{equation} \label{g.u}
(g\cdot u)(x)(b) = u(x)\big(g^{-1}\cdot(b\otimes 1_C)\big) \in C,
\end{equation}
 where  $x$ runs over $ A$ and $b$ runs over $B$.  It can be
 easily verified that $g\cdot u$ is an algebra homomorphism
  satisfying \eqref{strange-property}, that is   $g\cdot u \in \Yscr^A_B(C)$.
  This  defines an action of the group $\Gscr(C)$ on the set $\Yscr^A_B(C)$.
Varying $C$, we obtain an  action,  $\omega$, of
the  group scheme $\Gscr $  on the  affine scheme
$\Yscr^A_B$.  The induced map   $\omega^*: A_B \to B \otimes
A_B$ is evaluated on any   generator $x_b$ of  $A_B$ as follows:
for any   $C$,   $g $,  $u$ as above,
\begin{eqnarray*}
\omega^*(x_b)\vert_{(g,u)} \  \by{evaluate_morphisms} \ (x_b)\vert_{\omega_C(g,u)}
&=& (x_b)\vert_{g\cdot u}\\
&=& (g\cdot u)(x)(b) \\
&\by{g.u}&   u(x)\big(g^{-1}\cdot (b\otimes 1_C)\big)  \\
& \by{action} & u(x)\big( b^\ell  \otimes b^r\vert_{g^{-1}} \big) \\
&=& u(x) \big( b'' \otimes (s_B(b')b''')\vert_{g^{-1}} \big)\\
&=& \big(u(x)(b'')\big) \cdot  (s_B(b')b''')\vert_{g^{-1}} \\
&=& (s_B(b')b''')\vert_{g^{-1}} \cdot   \big(u(x)(b'')\big) \\
&=& (b's_B(b'''))\vert_g\cdot (x_{b''})\vert_u \\
&=& \big(b's_B(b''') \otimes x_{b''}\big)\vert_{(g,u)}.
\end{eqnarray*}
Here the third and the ninth identity follow from the description
of the bijection $\Yscr^A_B(C) \simeq \Hom_{\Com}(A_B,C)$
in the proof of Lemma \ref{just--+},
 we  use Lemma~\ref{conjugation_action} in the sixth identity
 to compute $\Delta(b)=b^\ell \otimes b^r\in B\otimes B$,  and, starting from the seventh identity,
 the dot denotes the multiplication in $C$. We deduce that
 $$
 \omega^*(x_b)= b's_B(b''') \otimes x_{b''} \in B \otimes A_B
 $$
 and \eqref{Delta(x_b)} then follows from \eqref{decompose_Delta}.
\end{proof}

We illustrate Lemma \ref{conjugation_action_bis}  with  two examples.
In both, $\Gscr$ is   a group scheme with coordinate   algebra $B$.
Let $A=\kk \Gamma$ be the   bialgebra of  a monoid~$\Gamma$ and recall from Section~\ref{ex_monoid_algebras}
the isomorphism of group schemes
$$
\Yscr^A_B \simeq \Hom_{\Mon}(\Gamma,\Gscr(-)).
$$
 Under this isomorphism, the action of $\Gscr$ on $\Yscr^A_B$ given  in    Lemma \ref{conjugation_action_bis} corresponds to the action of~$\Gscr$
on $\Hom_{\Mon}(\Gamma,\Gscr(-))$ by conjugation on the target.

Assume now that   $\Gscr$ is infinitesimally-flat,  and let
 $A=U(\mathfrak{p})$ be the enveloping algebra of a Lie algebra $\mathfrak{p}$.
 Recall  from Section~\ref{enveloping_algebras} the  isomorphism of group schemes
$$
\Yscr^A_B \simeq \Hom_{\Lie}\big(\mathfrak{p}, \mathfrak{g}\otimes(-)\big).
$$
 Under this isomorphism,   the action of $\Gscr$ on $\Yscr^A_B $ given in    Lemma
\ref{conjugation_action_bis} corresponds to  the action of $\Gscr$
on ${\Hom_{\Lie}\big(\mathfrak{p}, \mathfrak{g}\otimes(-)\big)}$ induced by the   adjoint representation of~$\mathfrak{g}$.

\section{Relations to Van den Bergh's    double brackets} \label{relation_to_VdB}

   We   outline    relations between our work and  Van den Bergh's  theory  of double brackets.

\subsection{Double brackets} \label{d_b}

Van den Bergh \cite{VdB} introduced
  double brackets in   algebras
as  non-commutative versions of   Poisson brackets in commutative algebras.
We  recall here    his    main  definitions and results reformulated for  graded algebras.

Fix an   integer $n $.  An  {\it $n$-graded double bracket}     in a graded algebra $A$ is a  linear map
$\double{-}{-}:A^{\otimes 2} \to A^{\otimes 2}$ satisfying  certain conditions formulated in \cite{VdB} (see also  \cite{MT_high_dim} for     the graded case).
 These conditions   amount to an $n$-graded version  of  the    Leibniz derivation rule  with respect to each of the two variables,   the inclusion
$$
\double{A^p}{A^q}  \subset  \bigoplus_{{i+j=p+q+{{n}}}}\, A^i \otimes A^j
$$
for any   $p,q \in \ZZ$, and   the {\it ${{n}}$-antisymmetry}
\begin{equation} \label{n-antisymmetry}
\double{y}{x} = - (-1)^{\vert x \vert_n \vert y \vert_n+ \left\vert \double{x}{y}^\ell\right\vert  \left\vert \double{x}{y}^r\right\vert }
\double{x}{y}^r \otimes \double{x}{y}^\ell
\end{equation}
 for  any homogeneous $x,y\in A$ where  we expand   $\double{x}{y}={\double{x}{y}^\ell \otimes \double{x}{y}^r}$.

Pick  now an  integer $N\geq 1$.
The functor $$\Hom_{\grAlg}(A,\hbox{Mat}_N(-)):\grCom   \longrightarrow   \Set, $$
which assigns to any  commutative graded algebra $C$ the set of $C$-linear actions  of $A$ on $C^N$
is   representable:  one   defines by generators and relations a  commutative graded algebra  $A_N$ such that $  \Hom_{\grAlg}(A,\hbox{Mat}_N(C)) \simeq \Hom_{\grCom}(A_N,C) $ for any $C$.
The  generators of $A_N$ are the symbols $x_{ij}$ where $x$ runs over $  A$ and $i,j$ run over the set $\overline{N}=  \{1,\dots, N\}$.   Van den Bergh   shows that   any    $n$-graded double bracket   $\double{-}{-}$ in $A$
induces  an  $n$-antisymmetric $n$-graded bracket in $A_N$   by
\begin{equation} \label{VdB_formula}
\bracket{x_{ij}}{y_{uv}}=  \double{x}{y }_{uj}^\ell  \double{x }{y }_{iv}^r
\end{equation}
for   $x,y  \in A$ and $i,j,u,v\in \overline{N}$.
To study the Jacobi identity, Van den Bergh   associates to $\double{-}{-}$   an  endomorphism $\triple{-}{-}{-}$ of $A^{\otimes 3}$,    the  {\it     triple bracket},   by
\begin{equation}\label{triplebracket}
\triple{-}{-}{-} = \sum_{i=0}^2  \Perm_{312}^i ( \double{-}{-}
\otimes \id_A) (\id_A \otimes \double{-}{-}) \Perm_{312,n}^{-i}
\end{equation}
where    $\Perm_{312}, \Perm_{312,n}\in \End(A^{\otimes 3})$ are as    in Section~\ref{triple}.
The double bracket $\double{-}{-}$ is \emph{Gerstenhaber of degree $n$} if $\triple{-}{-}{-}=0$.
This condition implies  that, for any  $N\geq 1$,   the bracket \eqref{VdB_formula} in $A_N$ is Gerstenhaber of degree $n$.
  Gerstenhaber double brackets of degree~$0$ in ungraded algebras are called \emph{Poisson double brackets}.

\subsection{Fox pairings versus double brackets} \label{FPvDB}
As we now explain (without proofs),   Fox pairings and   double brackets  in Hopf algebras  are closely related.
Consider an   involutive   graded Hopf algebra $A$ with counit $\varepsilon_A$, comultiplication $\Delta_A$, and   antipode $s_A$.
Any  antisymmetric  Fox   pairing~$\rho$ of degree $n \in \ZZ $  in~$A$
induces an $n$-graded  double bracket   $\double{-}{-}_\rho   $  in~$A$~by
\begin{equation} \label{double_rho}
  \double{x}{y}_{ \rho}   = (-1)^{\vert y' \vert  \vert x\vert_n + \vert x'\vert \vert  (\rho(x'',y''))'  \vert}\,    y's_A \!\left(\big(\rho( x'', y'')\big)'\right)  x' \otimes \big(\rho( x'', y'')\big)''
\end{equation}
where   $x,y$ run over all      homogeneous   elements of $  A$.
The pairing $\rho$ may be recovered from this double bracket by     $\rho=(\varepsilon_A  \otimes \id_A)  \double{-}{-}_{\rho} $.
Thus,  for an involutive graded Hopf algebra $A$,   we can view   $n$-graded double brackets in~$A$ as generalizations of antisymmetric  Fox pairings in~$A$ of degree $n$.
  For   cocommutative $A$, we make  a few further claims:
\begin{itemize}
\item[(i)] An  $n$-graded double bracket $\double{-}{-}$ in $A$
arises from an antisymmetric  Fox pairing of degree~$n$ in~$A$  if and only if $ \double{-}{-}$  is \emph{reducible}
in the   sense that
 \begin{equation*}
  x' s_A\big( \double{x''}{y'}^{ \ell}  \big) \otimes  \double{x''}{y'}^r  s_A(y'')  \in  \Delta_A  (A)  \subset  A \otimes A
\end{equation*}
  for any $x,y \in A$.
If   $\double{-}{-}=\double{-}{-}_\rho $    arises   from an antisymmetric  Fox pairing $\rho$ in $A$ of degree $n$,
then the   tritensor map  \eqref{triring}   and the tribracket \eqref{triplebracket} are related~by
\begin{equation} \label{tripleS}
\triple{-}{-}{-} = \Perm_{213} \circ \triplep{-}{-}{-} \circ \Perm_{213,n}.
\end{equation}
\item[(ii)] If  $\double{-}{-}$ is an  $n$-graded double bracket in $A$ and  $B$ is a commutative ungraded  Hopf algebra   equipped with a trace-like    $t\in B$,
then there  is a unique   $n$-graded $n$-antisymmetric bracket $\bracket{-}{-}$ in the graded algebra $A_B  $
 such that
\begin{equation} \label{xbyc}
\bracket{x_b}{y_c} =(-1)^{ \vert x''\vert   \vert y'\vert_n   }
\double{x'}{y'}_{t^{(1)}}^{\ell}\, \double{x'}{y'}^{r}_{t^{(3)}}\,
x''_{b \curlyvee t^{(2)}} \, y''_{c \curlyvee t^{(4)}}
\end{equation}
for any   $x,y \in A$ and   $b,c \in B$, where  $\curlyvee$ is defined by  \eqref{ltimes}.
When   $\double{-}{-}=\double{-}{-}_\rho $    arises  from an antisymmetric  Fox pairing $\rho$  of degree $n$,
  the bracket \eqref{xbyc} coincides with the bracket  \eqref{main_formula}
derived from $\rho$ and $\bullet=\bullet_t$.
\end{itemize}

  These claims have the following consequences.
First of all, it follows from \eqref{tripleS} that   an antisymmetric  Fox pairing $\rho$ in
a  cocommutative  graded Hopf algebra $A$    is Gerstenhaber in our sense
if and only if the double bracket $ \double{-}{-}_\rho$ defined by \eqref{double_rho} is Gerstenhaber in the sense of   \cite{VdB}.
For instance, it is proved in \cite{MT_high_dim} that   $\double{-}{-}_\rho$ is Gerstenhaber for the Fox pairing $\rho$   evoked    in Example~\ref{Examples-rho}.3
which implies that $\rho$ is Gerstenhaber.

  Next, the above claims and   Theorem \ref{MAINMAIN} imply   that
the bracket \eqref{xbyc}  in $A_B$ derived from  a reducible Gerstenhaber double bracket   of degree $n$ in $A$ is   Gersthenhaber of degree~$n$.
We wonder
 whether   this   extends to non-reducible double brackets. Note that the reducibility property   is quite restrictive:
  for   instance, most of the    Poisson    double brackets     in    free associative algebras
  constructed in \cite{PW,ORS} are \emph{not} reducible with respect to the standard Hopf algebra structures in these algebras.
When   $B=\kk[\GL_N]$ and~$t$ is the usual trace (see Section~\ref{GL_ex}),    the bracket  \eqref{VdB_formula}  is carried to the bracket    \eqref{xbyc}
by the   algebra homomorphism $A_N \to A_B, x_{ij} \mapsto x_{(t_{ij})}$ for  all    $x\in A$ and $i,j\in \overline{N}$.
  Since the image of this homomorphism    generates  $A_B$,
 the bracket  \eqref{xbyc} in $A_B$ determined by any (possibly,  non-reducible) Gerstenhaber double bracket in~$A$ is  Gerstenhaber   for these~$B$ and~$t$.

  Finally, the above claims have similar implications in the quasi-Poisson case.
In particular, \eqref{tripleS}  implies   that   an antisymmetric  Fox pairing $\rho$ in a  cocommutative ungraded Hopf algebra     is quasi-Poisson
if and only if the double bracket $ \double{-}{-}_\rho$ defined by \eqref{double_rho} is quasi-Poisson in the sense of Van den Bergh \cite{VdB}.
By  \cite{MT_dim_2},   $\double{-}{-}_\rho$ is quasi-Poisson for the Fox pairing $\rho$   in  Section~\ref{surfaces}    so that    $\rho$ is quasi-Poisson   in our sense.

\section{Free commutative Hopf algebras} \label{free}

 We consider    commutative   Hopf algebras freely generated by  coalgebras in the sense of Takeuchi,
 see \cite{Ta} and \cite[Appendix B]{AK}.
We show that  balanced biderivations in these Hopf algebras  naturally arise from  cyclic bilinear forms  on  coalgebras    defined  in~\cite{Tu_rep_alg}.
 Unless otherwise mentioned, in this appendix,  by a module/algebra/coalgebra we mean an ungraded module/algebra/coalgebra.

    Let $M$ be a   coalgebra with comultiplication $\Delta_M$ and  counit $\varepsilon_M$. Takeuchi   introduced
    a commutative   Hopf algebra $ F(M)$ called the
    \emph {free commutative Hopf algebra generated by~$M$}.
  This  Hopf algebra   is  defined in \cite[Section 11]{Ta} as an initial object in the category of
   commutative Hopf algebras~$X$ endowed with a coalgebra homomorphism $M\to X$.
     It can be explicitly constructed as follows,  see \cite[Appendix B]{AK}.
 As an algebra, $F(M)$ is generated by the symbols $\{m^+,m^-\}_{m\in M}$  subject to  the following relations: for any $k\in \kk$ and $l,m\in M$,
\begin{equation} \label{relations_Takeuchi_1}
(km)^\pm = k m^\pm, \quad  (l+m)^\pm = l^\pm + m^\pm,
\end{equation}
\begin{equation} \label{relations_Takeuchi_2}
l^\pm\, m^\pm = m^\pm\, l^\pm, \quad  l^\pm\, m^\mp = m^\mp\, l^\pm,
\end{equation}
\begin{equation} \label{relations_Takeuchi_3}
\quad (m')^+ (m'')^- = \varepsilon_M(m)\, 1=  (m')^- (m'')^+,
\end{equation}
 where $\Delta_M(m)=m' \otimes m''$   in Sweedler's notation.
  It is easily verified that there are  algebra homomorphisms
 $$\Delta: F(M) \to F(M) \otimes F(M), \quad \varepsilon: F(M) \to \kk, \quad s:F(M) \to F(M)$$
  defined on  the generators by {\small
$$
\Delta(m^+) = (m')^+ \otimes (m'')^+, \ \Delta(m^-) =  (m'')^- \otimes (m')^-, \quad \varepsilon(m^\pm) =\varepsilon_M(m), \quad s(m^\pm) = m^\mp,
$$}
for any $m \in M$.
These homomorphisms turn~$F(M)$ into a commutative Hopf algebra which, together with the   map $  M \to F(M), m\mapsto  m^+$,
  has the desired universal property.

Consider the ideal  $I=\Ker \varepsilon \subset F(M)$. Computing  the module  $I/I^2$   from the presentation of $F(M)$,
we obtain that the linear map
$$
M\longrightarrow I/I^2, m\longmapsto m^+-\varepsilon (m^+)\! \mod I^2
$$ is an isomorphism.
As a consequence, the linear map $  M \to F(M), m\mapsto  m^+$  is   injective. This allows us to treat $M$ as a submodule of $F(M)$.
 By   the correspondence \eqref{correspondence},
every bilinear form $M\times M\to \kk$ extends uniquely to a biderivation   $F(M) \times F(M) \to \kk$.

A    bilinear form $\bullet_M: M \times M \to \kk$ is   said to be  \emph{cyclic}, see \cite{Tu_rep_alg},  if
\begin{equation} \label{cyclic_condition}
(l\bullet_M m'')\, m' \otimes m''' = (m \bullet_M l'')\, l''' \otimes l'
\end{equation}
for any $l,m\in M$.
 Applying $\varepsilon_M \otimes \varepsilon_M$ to both sides of \eqref{cyclic_condition}, we obtain   that   cyclic bilinear forms are  symmetric.

\begin{lemma} \label{cyclic_to_balanced}
Any cyclic bilinear form $\bullet_M:M \times M \to \kk$ extends uniquely to a balanced biderivation $  F(M) \times F(M) \to \kk$.
\end{lemma}

\begin{proof}
By the above, $\bullet_M$ extends uniquely to a biderivation $\bullet$ in $ F(M)  $. We only need to show that $\bullet$ is balanced.
Define a bilinear map $\kappa: F(M) \times F(M) \to F(M)$   by
$$
\kappa(b,c)=(b\bullet c'')\, s(c') c'''- (c \bullet b'') s(b''') b'
$$
for any $b,c\in F(M)$.  Straightforward   computations show that
$$
\kappa(b_1b_2,c) = \varepsilon(b_2)\, \kappa(b_1,c) + \varepsilon(b_1)\, \kappa(b_2,c), \quad \kappa(b,c_1 c_2) = \varepsilon(c_2)\, \kappa(b,c_1) + \varepsilon(c_1)\, \kappa(b,c_2)
$$
for any   $b,b_1,b_2,c,c_1,c_2 \in F(M)$.
Therefore the submodule  $\kk\cdot 1+I^2 $  of $F(M)$   is contained both in the left  and the  right annihilators of $\kappa$.
 Hence, $\kappa$ is fully determined by its restriction to $M\times M$. The   condition  \eqref{cyclic_condition}  implies that $\kappa(M, M)=0$. Hence $\kappa=0$ and $\bullet$ is   balanced.
\end{proof}

  The following   claim is a particular case  of Theorem~\ref{double_to_simple}.

\begin{corol} \label{free_case}
Let~$\rho$ be an antisymmetric Fox pairing    of degree $n\in \ZZ$ in a cocommutative  graded Hopf algebra~$A$.
Let ~$\bullet_M$  be a cyclic bilinear form in a   coalgebra $M$ and let   $B=F(M)$.
Then there is a unique $n$-graded   bracket $\bracket{-}{-}$ in $ A_B$    such that
\begin{eqnarray}  
 \bracket{x_b}{y_c}
\notag & =&  (-1)^{  \vert x' \vert \vert \rho(x'',y'')' \vert + \vert y' \vert \vert x \vert_n}   (b\bullet_M\! c'') \cdot\\
\label{specase} &&  \big(y' s_A(\rho(x'',y'')') x'\big)_{c'}\,  \big(\rho(x'',y'')''\big)_{c'''}
\end{eqnarray}
for any homogeneous $x,y \in A$ and $b,c \in M$.  This $n$-graded   bracket is antisymmetric.
\end{corol}

\begin{proof}
Since the algebra $B$ is generated by  the set $\{m, s(m)\}_{m\in M}$,  the algebra $A_B$ is generated by the set  $\{x_m\}_{x\in A, m\in M}$ which proves the unicity.
To prove the existence,   consider the balanced biderivation $\bullet$ in $B$  extending  $\bullet_M$ and the bracket $\bracket{-}{-}$ in $A_B$
 obtained by an application of Theorem~\ref{double_to_simple}   to $\rho$ and $\bullet$.
Note that \eqref{cyclic_condition} implies that
\begin{equation} \label{cyclic_bis}
(l \bullet  m'')\, m' = (m \bullet  l')\, l''
\end{equation}  for any $l,m \in M$.
Then, for any homogeneous $x,y \in A$ and $b,c \in M$,  we have
\begin{eqnarray*}
 && \bracket{x_b}{y_c} \\
 &=&          (-1)^{ \vert   x''\vert   \vert y' \vert_n}  (c''\bullet b^{(2)})\, \rho(x',y')_{ {{s}}(b^{(3)}) b^{(1)}}\,   x''_{b^{(4)}}\,   y''_{c' } \\
 & \by{cyclic_bis} &    (-1)^{ \vert   x''\vert   \vert y' \vert_n}  (c \bullet b^{(2)})\, \rho(x',y')_{ {{s}}(b^{(4)}) b^{(1)}}\,   x''_{b^{(5)}}\,   y''_{b^{(3)} } \\
 &=&   (-1)^{ \vert   x''\vert   \vert y' \vert_n}  (c \bullet b^{(2)})\, s_A(\rho(x',y')')_{ b^{(4)}}\, \rho(x',y')''_{  b^{(1)}} \,   x''_{b^{(5)}}\,   y''_{b^{(3)} } \\
 &=&   (-1)^{ \vert   x''\vert   \vert y' \rho(x',y')'' \vert_n +\vert y'' \vert \vert x y' \vert_n  }  (c \bullet b^{(2)})\, y''_{b^{(3)}}\, s_A(\rho(x',y')')_{ b^{(4)}}\,  x''_{b^{(5)}}\,  \rho(x',y')''_{  b^{(1)}} \\
 &=&   (-1)^{ \vert   x''\vert   \vert y' \rho(x',y')'' \vert_n +\vert y'' \vert \vert x y' \vert_n  }  (c \bullet b'')\, \big(y''\, s_A(\rho(x',y')')  x''\big)_{b'''}\,  \rho(x',y')''_{  b'} \\
 &=&   (-1)^{ \vert   x''\vert   \vert  \rho(x',y')' x' \vert +\vert y'' \vert \vert x y' \vert_n  }  (c \bullet b'')\, \big(y''\, s_A(\rho(x',y')')  x''\big)_{b'''}\,  \rho(x',y')''_{  b'} \\
 &=&    (-1)^{ \vert   x'\vert   \vert  \rho(x'',y')'  \vert +\vert y'' \vert \vert x y' \vert_n  }  (c \bullet b'')\, \big(y''\, s_A(\rho(x'',y')')  x'\big)_{b'''}\,  \rho(x'',y')''_{  b'} \\
 &=&    (-1)^{ \vert   x'\vert   \vert  \rho(x'',y'')'  \vert +\vert y' \vert \vert x  \vert_n  }  (c \bullet b'')\, \big(y'\, s_A(\rho(x'',y'')')  x'\big)_{b'''}\,  \rho(x'',y'')''_{  b'}
 \end{eqnarray*}
 so that \eqref{specase} now  follows from \eqref{cyclic_condition}.
\end{proof}

We briefly discuss trace-like elements of $F(M)$ lying in $M\subset F(M)$. Assume for simplicity that the underlying module of~$M$ is free of finite rank.
An element $t\in M  $ is cosymmetric  if and only if the bilinear form
$$
(-,-)_t:M^* \times M^* \longrightarrow \kk, \ (l,m) \longmapsto l(t')\, m(t'')
$$
is symmetric. An element $t\in M  $ is infinitesimally-nonsingular if and only if the form $(-,-)_t$ is nonsingular.
Consequently,  $t\in M  $ is   trace-like   if and only if the  algebra $M^*$ dual to $M$ and equipped with the bilinear form $(-,-)_t$ is a symmetric Frobenius algebra.
For a trace-like   $t\in M$, the  balanced biderivation $\bullet_t$ in $F(M)$ restricts to  a cyclic bilinear form on~$M$.
This connection between symmetric Frobenius algebras and cyclic bilinear forms   on coalgebras
was first observed    in \cite{Tu_rep_alg}. The bracket \eqref{xbyc} specializes in this case to the bracket   in \cite{Tu_rep_alg};
 this    directly follows from \eqref{specase} if the double bracket    in \eqref{xbyc} is reducible.

  For example, consider the  coalgebra $M= \big(\Mat_N(\kk)\big)^*$ dual to the algebra of $N\times N$ matrices $ \Mat_N(\kk)$.
Then the Hopf algebra $F(M)$ is nothing but the Hopf algebra $B$  from Section~\ref{GL_ex}.
Indeed, it is easily checked that $B$ verifies the universal property of $F(M)$  (cf.\@  \cite[Example B.3]{AK}).
 The trace-like element $t\in B$ pointed out  in Section~\ref{GL_ex} belongs to   $M\subset B$.
 Thus, the balanced biderivation $\bullet_t$ in $B$ is induced by  a cyclic bilinear form in $M$.
  Corollary \ref{free_case} applies and yields again the formula \eqref{almost_VdB-}.

\end{document}